\documentclass{amsart}
		
\usepackage[utf8]{inputenc}
\usepackage{amsmath, amssymb, amsfonts, amsthm}
\usepackage[english]{babel}
\usepackage[all]{xy}
\usepackage{hyperref}

\hypersetup{
    colorlinks = true,
		}

\usepackage{tikz}
\usepackage{mathrsfs}
\usepackage{etoolbox}
\usepackage{lscape}
\makeatletter
\patchcmd{\@verbatim}
  {\verbatim@font}
  {\verbatim@font\tiny}
  {}{}
\makeatother

\usepackage{longtable}

\usetikzlibrary{calc}
\usetikzlibrary{graphs}
\usetikzlibrary{graphs.standard}

\DeclareFontFamily{U}{wncy}{}
\DeclareFontShape{U}{wncy}{m}{n}{<->wncyr10}{}
\DeclareSymbolFont{mcy}{U}{wncy}{m}{n}
\DeclareMathSymbol{\sha}{\mathord}{mcy}{"58}

\theoremstyle{plain}
\newtheorem{theoremintro}{Theorem}
\newtheorem{theorem}{Theorem}[section]

\newtheorem{proposition}[theorem]{Proposition}
\newtheorem{lemma}[theorem]{Lemma}

\theoremstyle{definition}
\newtheorem{definition}[theorem]{Definition}
\theoremstyle{remark}
\newtheorem{remark}[theorem]{Remark}

\numberwithin{equation}{section}

\DeclareMathOperator{\Pic}{Pic}
\DeclareMathOperator{\Gal}{Gal}

\DeclareMathOperator{\Hom}{Hom}
\DeclareMathOperator{\Coker}{Coker}
\DeclareMathOperator{\Bl}{Bl}
\DeclareMathOperator{\Cl}{Cl}
\DeclareMathOperator{\Br}{Br}
\DeclareMathOperator{\Proj}{Proj}
\DeclareMathOperator{\Sing}{Sing}

\DeclareMathOperator{\Tr}{Tr}

\DeclareMathOperator{\irr}{irr}
\DeclareMathOperator{\eff}{eff}
\DeclareMathOperator{\Stab}{Stab}
\DeclareMathOperator{\Id}{Id}

\DeclareMathOperator{\tors}{tors}
\DeclareMathOperator{\ann}{ann}

\DeclareMathOperator{\res}{res}
\DeclareMathOperator{\Mat}{Mat}

\def\F{\mathbf{F}}
\def\P{\mathbf{P}}
\def\I{\mathbf{I}}
\def\E{\mathbf{E}}
\def\Z{\mathbf{Z}}
\def\A{\mathbf{A}}
\def\B{\mathbf{B}}
\def\D{\mathbf{D}}
\def\Q{\mathbf{Q}}
\def\C{\mathbf{C}}

\def\W{\mathrm{W}}
\def\R{\mathrm{R}}

\def\FFF{\mathcal{F}}

\def\O{\mathcal{O}}

\def\Ccal{\mathcal{C}}
\def\Dcal{\mathcal{D}}
\def\N{\mathcal{N}}

\def\S{\mathfrak{S}}

\def\RR{\mathscr{R}}
\def\CC{\mathscr{C}}

\makeatletter
\@namedef{subjclassname@2020}{\textup{2020} Mathematics Subject Classification}
\makeatother

\title{Classification of singular del Pezzo surfaces over finite fields}

\author[R. Blache]{Régis Blache}
\address{LAMIA, Université des Antilles}
\email{regis.blache@univ-antilles.fr}

\author[E. Hallouin]{Emmanuel Hallouin}
\address{Institut de Math\'ematiques de Toulouse, UMR 5219}
\email{hallouin@univ-tlse2.fr}

\date{\today} 



\begin{document}

\begin{abstract}
In this article, we consider weak del Pezzo surfaces defined over a finite field, and their associated, singular, anticanonical models. 

We first define arithmetic types for such surfaces, by considering the Frobenius actions on their Picard groups; this extends the classification of Swinnerton-Dyer and Manin for ordinary del Pezzo surfaces. We also show that some invariants of the surfaces only depend on the above type.

Then we study an inverse Galois problem for singular del Pezzo surfaces having degree $3\leq d\leq 6$: we describe which types can occur over a given finite field (of odd characteristic when $3\leq d\leq 4$).
\end{abstract}

\maketitle

\tableofcontents

\section*{Introduction}

In this article, we study certain del Pezzo surfaces defined over a finite field. Recall that a smooth projective surface $X$ is a \emph{weak del Pezzo surface} when its anticanonical divisor $-K_X$ is big and nef; its \emph{degree} is the self-intersection number $d:=K_X^{\cdot2}$. If moreover $-K_X$ is ample, we call $X$ an \emph{ordinary} del Pezzo surface. When a weak del Pezzo surface $X$ is not ordinary, it contains absolutely irreducible curves with self intersection $-2$, that we shall call $(-2)$-curves in the following. 

The anticanonical model of a weak, non ordinary del Pezzo surface $X$ is a singular surface, that we denote by $X_s$, and call a \emph{singular} del Pezzo surface. Note that $X_s$ has Du Val singularities, and is Gorenstein; for these reasons such surfaces are sometimes called \emph{Du Val del Pezzo} or \emph{Gorenstein del Pezzo} in the literature. 

The study of complex singular cubic surfaces (degree $3$ del Pezzo surfaces over $\C$) dates back to the nineteenth century \cite{cay,sch}. It has been generalized to del Pezzo surfaces along the twentieth century \cite{duv}. In particular we know all types of singularities (sometimes called the Dynkin types) that can occur in characteristic zero \cite[Chapter 8]{dolga}. Over an algebraically closed field of positive characteristic, some new types occur, but only in characteristic $2$ \cite{kn2}.

The interest on del Pezzo surfaces over finite fields is more recent. If $X$ is an ordinary del Pezzo surface defined over the finite field $\F_q$, where $q=p^m$ is a power of a prime, the Frobenius action $\sigma^\ast$ on $\Pic(X\otimes\overline{\F}_q)$ must preserve the anticanonical class and the intersection product. The group of automorphisms of $\Pic(X\otimes\overline{\F}_q)$ with these properties is a (finite) Weyl group depending on the degree of the surface. Thus the image of the Galois group $\Gal(\overline{\F}_q/\F_q)$ is a cyclic subgroup generated by the image of the Frobenius morphism $\sigma$; its conjugacy class is the \emph{arithmetic type} of $X$. Swinnerton-Dyer \cite{swi} and Manin \cite{manin} construct tables of conjugacy classes in these Weyl groups in order to classify ordinary del Pezzo surfaces over finite fields (the table for degree $3$ has been corrected in \cite{bfl}).

Many invariants of a del Pezzo surface only depend on its arithmetic type; this is the case for its zeta function \cite[Theorem 27.1, Corollary 27.1.2]{manin}

\begin{equation}
\label{zeta}
Z(X,T)^{-1}=(1-T)(1-q^2T)\det\left(\I-qT\sigma^\ast |\Pic(X\otimes\overline{k})\right)
\end{equation}

The first aim of this paper is to extend the classification of Swinnerton-Dyer and Manin to weak del Pezzo surfaces, by defining their \emph{arithmetic type}, and to give an expression for their zeta functions. 

We begin by classifying the possible singularities over the algebraic closure. Following \cite {ct,der}, we define a \emph{geometric type}. This is the configuration of negative curves on the surface $X\otimes \overline{k}$ (corresponding to lines and singularities on the anticanonical model), up to the action of the Weyl group. It is finer than the Dynkin type. The types are orbits of the Weyl action on the root bases in the lattice $\E_{9-d}$ corresponding to the degree $d$.

Coming back to the surface $X$, the Galois action must preserve the set $\R_{\irr}$ of $(-2)$-curves. Thus the Frobenius maps to an element of its stabilizer $\Stab(\R_{\irr})$. We define the \emph{arithmetic type} of $X$ as the conjugacy class of the Frobenius action in this last group.  

Here again, many properties of a weak del Pezzo surface only depend on its arithmetic type; this is true for its zeta function from (\ref{zeta}), and we show that it remains true for its anticanonical model $X_s$ (note that its geometric Picard lattice is the orthogonal of $\R_{\irr}$ in the geometric Picard group of $X$).

\begin{theoremintro}
\label{zetasing}
We have the equality
\[
Z(X_s,T)^{-1}=(1-T)(1-q^2T)\det\left(\I-qT\sigma^\ast|\Pic(X_s\otimes\overline{k})\right)
\]
\end{theoremintro}

Our second aim is to construct such surfaces. When the finite field is small, not all are constructible \cite{trepa}. For instance, over the field $\F_2$, there are no split ordinary del Pezzo surfaces of degree $d\leq 4$ since there are at most four points in general position in $\P^2(\F_2)$.

This is the inverse Galois problem for singular del Pezzo surfaces over finite fields: in the ordinary case, it asks for which conjugacy classes of the Weyl group can arise as the conjugacy class of the Frobenius action. It has been solved for ordinary degree four surfaces in \cite[Theorem 1.4]{trepa}, and for ordinary degree three and two surfaces in \cite{lougtrep}. There is also a weaker version of this problem, asking for which integers can arise as the trace of the Frobenius action. It has been solved for ordinary del Pezzo surfaces, see \cite{bfl}.

It should be easier for non ordinary del Pezzo surfaces since when we consider them as blowups of the projective plane, we relax the condition of blowing up points in general position to almost general position. But when the degree is at most four and the base field is not algebraically closed, some of the surfaces are no longer birational to the projective plane, and we have to provides other constructions.

In this paper we solve the inverse Galois problem for singular del Pezzo surfaces of degree at least $5$ over any finite field, and for surfaces of degree $3$ or $4$ when the characteristic is odd. We refer to Appendix \ref{gt} for the tables of arithmetic types for each degree $3\leq d \leq 6$.

\begin{theoremintro}
\label{alltypes} 
There exists a weak, non ordinary del Pezzo surface of degree $d$ and any arithmetic type $\mathfrak{T}$ over the finite field $\F_q$ in the following cases
\begin{itemize}
	\item[1.] \label{alltypesdeg5} we have $d\geq 5$;
	\item[2.] \label{alltypesdeg4} we have $d=4$ and $\F_q$ has odd characteristic.
	\item[3.] \label{alltypesdeg3} we have $d=3$, $\F_q$ has odd characteristic and $(q,\mathfrak{T})\notin \{(3,1),(3,12)\}$.
\end{itemize}

There does not exist any surface of degree $3$ and arithmétic type $1$ or $12$ over $\F_3$.

\end{theoremintro}

The types corresponding to the degrees $d\geq 5$ are not very difficult to construct. It turns out that these surfaces are birational to $\P^2$, and it is sufficient to blowup the projective plane at well chosen points, and to contract exceptional curves. 

In the case of degree four, we no longer consider a blowup model. We exploit an idea which is present in \cite{BBFL,fly,Skoro_Del4}. The anticanonical model of a del Pezzo surface a degree four is the base locus of a pencil of quadrics in $\P^4$. To such a pencil we associate a quadratic $\F_q|T]$-module (equivalently, a Frobenius algebra). Now fixing a geometric type, then an arithmetic type for a del Pezzo surface, gives a precise description of the quadratic module. We construct such modules with prescribed arithmetic properties, and their existence is sufficient to prove the second assertion of the above theorem. Note that these quadratic modules are rather different in even characteristic, since there a bilinear form over an odd dimensionnal vector space must be degenerate. This is why we do not consider that case here. However, nice normal forms have been described in this case \cite{dd2}, that should help solving this problem.

Note that we have used the mathematical software {\tt magma} to construct explicit models (ie a couple of quadrics in $\P^4$) for all types of degree four singular del Pezzo surfaces over a given finite field. The code is freely available on the second author's webpage.

Finally, we construct degree three surfaces in different ways. Our main construction is by blowing up a point -- not lying on any negative curve -- on a well chosen degree four del Pezzo surface. We count the numbers of such points for each arithmetic type belonging to degree four; this allows us to show the existence of many degree three surfaces with given arithmetic type, but also -- when there is no such point -- the non existence result stated in the last sentence of the above Theorem. We also use other types of blow up (of the projective plane, or a degree four surface at a point lying on one or more exceptional curve) in order to construct the surfaces belonging to the remaining arithmetic types.

The paper is organised as follows: in section \ref{sec1}, we describe weak del Pezzo surfaces, and recall their principal properties. This allows us to introduce the classification and to define the different types (geometric, then arithmetic) that we use in the rest of the article. Then we turn to the description of singular del Pezzo surfaces; we briefly describe their different groups of divisors, and we show Theorem \ref{zetasing} in section \ref{sec2}. The next section is devoted to the proof of the first assertion in Theorem \ref{alltypes}. The fourth section is more technical: we treat the degree four del Pezzo surfaces over finite fields of odd characteristic; to such a variety, we associate a quadratic $\F_q|T]$-module. Then most of the work is devoted to describing the link between the arithmetic properties of the module and the arithmetic type of the surface; this allows us to prove the second assertion of Theorem \ref{alltypes}. The last section mainly builds on the preceding one: we blow up the degree four surfaces at well chosen points in order to construct degree three surfaces. We also use some direct constructions; this allows us to prove the last two assertions of Theorem \ref{alltypes}. 

We give the description of the different arithmetic types for degree $d\geq 3$ in the Appendix.

\section{Weak del Pezzo surfaces}
\label{sec1}

We first define the smooth surfaces we shall consider in this article

\begin{definition}
A smooth projective surface $X$ defined over a field $k$ is a \emph{weak del Pezzo surface} when its anticanonical divisor $-K_X$ is
\begin{itemize}
	\item[(i)] big, i.e. $K_X^{\cdot 2}>0$;
	\item[(ii)] and nef, i.e. for any effective divisor $D$ on $X$, $(-K_X)\cdot D\geq 0$.
\end{itemize}

It is an \emph{ordinary del Pezzo surface} when $-K_X$ is ample. Its \emph{degree} is $d:=K_X^{\cdot 2}$.
\end{definition}

Note that since $-K_X$ is nef, the adjunction formula ensures that for any absolutely irreducible curve $C$ on $X$, we have $C\cdot C\geq C\cdot(C+K_X)=2p_a(C)-2\geq -2$. Moreover, if for such a curve this inequality is an equality, then we have $C\cdot K_X=0$, and from the Nakai-Moishezon criterion the anticanonical divisor is not ample: the surface $X$ is a not an ordinary del Pezzo surface.

It is well known \cite{dem,ct} that the geometry of del Pezzo surfaces depends to a large extent of its absolutely irreducible curves with negative self-intersection, the so-called \emph{negative curves}. These are the generalization of the celebrated $27$ lines on an smooth cubic hypersurface in $\P^3$.

\begin{definition}
Let $X$ denote a weak del Pezzo surface over the field $k$.

An element $D$ in the geometric Picard group $\Pic(X\otimes\overline{k})$ is an \emph{exceptional divisor} when $D^{\cdot 2}=D\cdot K_X=-1$.

An absolutely irreducible curve $C$ on $X$ whose class is an exceptional divisor is an \emph{exceptional curve}.

An element $D$ in $\Pic(X\otimes\overline{k})$ is a \emph{root} when it satisfies $D^{\cdot 2}=-2$ and $D\cdot K_X=0$. We denote by $\R(X)$ the set of roots.

When $C$ is a curve on $X$ whose class is a root, we say that $C$ (or its class) is an \emph{effective root}. If moreover $C$ is absolutely irreducible, then it is a \emph{$(-2)$-curve}.

We denote by $A$ the union of the $(-2)$-curves on $X$; this is a closed subscheme of $X$. We denote by $U$ the complementary of $A$ on $X$.
\end{definition}

\begin{remark}
First remark that the negative curves which are absolutely irreducible are isomorphic to $\P^1$ from the adjunction formula.

The sets of exceptional divisors and roots are finite and depend up to isomorphism only on the degree of the surface \cite[II. Tables 2 et 3]{dem}.

Note also (contrary to the case of ordinary del Pezzo surfaces) that all exceptional divisors need not correspond to exceptional curves; see Lemma \ref{cbeexc} for a numerical criterion.

Finally, the sets of exceptional and $(-2)$-curves will be crucial in this article since they determine (up to an isomorphism) the geometric type.
\end{remark}

We first give a construction of weak del Pezzo surfaces as blow-ups of the projective plane.

\subsection{A blow-up model}

We assume $k$ algebraically closed in this section.

\begin{definition}

We denote by $X(\Sigma)$ the surface obtained from the projective plane by successively blowing up the points in $\Sigma:=\{p_1,\ldots,p_{r}\}$
\[
\pi : X(\Sigma)\stackrel{\Bl_{p_{r}}}{\longrightarrow} X_{r}\rightarrow \ldots \rightarrow X_2 \stackrel{\Bl_{p_1}}{\longrightarrow} X_1=\P^2
\]
where each $p_i$, $1\leq i\leq r$ is a closed point in the surface $X_{i}$.

For $1\leq i\leq r$, we denote by $E_i$ the total transform in $X(\Sigma)$ of the exceptional divisor of the blowing-up $\Bl_{p_{i}} : X_{i+1}\rightarrow X_i$, and we write $p_{i}\prec p_{i+1}$ when $p_{i+1}$ is \emph{infinitely near to} $p_i$, ie when it lies on the exceptional divisor of $\Bl_{p_{i}}$ in $X_{i+1}$.
\end{definition}

The Picard lattice of the surface $X(\Sigma)$ is the group generated by $E_0:=\pi^\ast L$, the total transform of the class $L$ of a line in $\P^2$, and the $E_i$, $1\leq i\leq r$
\[
\Pic(X(\Sigma))=\Z^{r+1}=\Z E_0+\Z E_1+\cdots+\Z E_r
\]
endowed with the intersection product given by $E_0^{\cdot 2}=1$, $E_i^{\cdot 2}=-1$ for $1\leq i\leq r$ and $E_i\cdot E_j=0$ for $i\neq j$. The canonical class is given by
\[
K_{X(\Sigma)} = -3E_0+E_1+\cdots+E_r
\]
From \cite[III. Theorem 1]{dem}, the surface $X(\Sigma)$ is a weak del Pezzo surface of degree $d=9-r$ if, and only if $r\leq 8$, and the points in $\Sigma$ are in \emph{almost general position}: at each stage, the point $p_i$ does not lie on a $(-2)$-curve on $X_i$.

The converse statement is almost true; actually we have the following description of weak del Pezzo surfaces \cite[Proposition 0.4]{ct}

\begin{proposition}
\label{class1}
Let $X$ denote a weak del Pezzo surface of degree $d$ over an algebraically closed field $k$. Then we have $1\leq d\leq 9$, and if we set $r=9-d$, we must have one of the following
\begin{itemize}
  \item[(i)] $r=1$ and $X\simeq \P^1\times\P^1$;
	\item[(ii)] $r=1$ and $X\simeq F_2$ the Hirzebruch surface;
	\item[(iii)] $0\leq r\leq 8$, and $X\simeq X(\Sigma)$, where $\Sigma:=\{p_1,\ldots,p_r\}$ consists of points in almost general position.
\end{itemize}
\end{proposition} 

\subsection{Roots, exceptional curves and geometric types}

In this section, we consider a weak del Pezzo surface $X$ of degree $d\leq 7$ over an algebraically closed field $k$, and we set $r:=9-d$ as above. Note that the description of del Pezzo surfaces of degrees $8$ and $9$ follows immediately from the preceding result.

There exists some $\Sigma:=\{p_1,\ldots,p_r\}$ consisting of points in almost general position such that $X\simeq X(\Sigma)$; this choice allows us to identify $\Pic(X)\simeq\Z^{r+1}$ as in the preceding section.

\begin{definition}
Recall that $\R(X)$ is the set of roots in $\Pic(X)$
\[
\R(X) := \{D\in \Pic(X),~D^{\cdot2}=-2,~D.K_X=0\}
\] 

We denote by $\R_{\eff}(X)\subset \R(X)$ (\emph{resp.} $\R_{\irr}(X)\subset \R(X)$) the subset of effective roots (\emph{resp.} of $(-2)$-curves) in $\Pic(X)$.

Let $\RR(X)$ denote the \emph{root module}; it is the sub-$\Z$-module of $\Pic(X)$ generated by $\R_{\irr}(X)$.
\end{definition}

It is well known \cite{dem} that the set $\R(X)$ is a root system in the orthogonal $(K_X)^\perp \otimes \Q$ of the canonical divisor $K_X$. Under the identification of $\Pic(X)$ and $\Z^{r+1}$, it is sent on the root system $\R_d$ with basis $\{E_0-E_1-E_2-E_3,E_1-E_2,\cdots,E_{r-1}-E_r\}$. We denote by $\E_{9-d}$ the intersection graph of this basis. It is the Dynkin diagram associated to the degree $d$, namely

\[
\begin{array}{c||c|c|c|c|c|c}
{\rm Degree~ } d & 6 & 5 & 4 & 3 & 2 & 1 \\
\hline
{\rm Dynkin~ diagram~ } \E_{9-d} & \A_2\times\A_1 & \A_4 & \D_5 & \E_6 & \E_7 & \E_8 \\

\end{array}
\]

The group of automorphisms of the Picard group $\Pic(X)$ preserving the canonical divisor and the intersection product is isomorphic to the Weyl group associated to $\E_{9-d}$, which we denote by $\W(\E_{9-d})$. It is generated by the reflections through the hyperplanes orthogonal to the roots, i.e. the $s_\alpha : x\mapsto x+(x\cdot\alpha) \alpha$, $\alpha\in\R_d$. 

The following result \cite[III Théorème 2]{dem} is fundamental for the geometric classification of weak del Pezzo surfaces

\begin{proposition}
\label{freeroots}
Let $X$ denote a weak del Pezzo surface of degree $d\leq 6$. Then the set $\R_{\eff}(X)\cup(-\R_{\eff}(X))$ is a closed and symmetric part of $\R(X)$.

It is a root system in the space $\RR(X)\otimes\Q$, of which the set $\R_{\irr}(X)$ forms a basis (and we call it a \emph{root basis}). 

As a consequence, the free $\Z$-module generated by $\R_{\irr}(X)$ is equal to $\RR(X)$.
\end{proposition}

An immediate consequence is that since we have $\RR(X)\subset K_X^\perp$, and this last module has rank $r=9-d$, there are at most $r$ $(-2)$-curves on $X$. Their intersection graph has a strong geometric significance. It is sometimes called the Dynkin type of $X$.

We are ready to define the first, geometric part of our classification

\begin{definition}
The \emph{geometric type} of $X$ is the orbit of the image of its set of $(-2)$-curves $\R_{\irr}(X)$ under the action of $\W(\E_{9-d})$ on the set of bases for closed and symmetric parts of $\R_d$.

When $X$ is ordinary, we say it has \emph{ordinary geometric type}.
\end{definition}

Note that two isomorphisms between the lattices $\Pic(X)$ and $\Z^{r+1}$ differ by an element of the Weyl group. This is why we define the type as an orbit under the action of this group: this makes it independent of the choice of such an isomorphism, and of the blown-up points. Note also that the geometric type above is equivalent to the type from \cite[Definition 3]{der}.

The possible orbits can be deduced recursively from a theorem by Borel and de Siebenthal that classifies the closed symmetric parts of maximal rank in a root system up to the action of the Weyl group \cite[p 29]{merindol}, \cite[p 404]{dolga}. They are given, degree by degree, in \cite[Chapters 8 and 9]{dolga}. We recall them in Appendix \ref{gt} as a column in the table of arithmetic types.

For each orbit, one can choose a root basis in such a way that its intersection graph is a subgraph of the Dynkin diagram $\E_{9-d}$ when $d\geq 5$, and of the extended (or affine) Dynkin diagram $\widetilde{\E}_{9-d}$ when $3\leq d\leq 4$; this is no longer true when $d\in\{1,2\}$. 

\begin{remark}
The geometric type we have just defined is a finer invariant than the Dynkin type. For instance, if $d=6$ there are two orbits of Dynkin type $\A_1$, depending on whether the $(-2)$-curve lies on the $\A_1$ or the $\A_2$ component of the root system: surfaces in the corresponding geometric types differ by the number of exceptional curves they contain. There are also two orbits for each of the Dynkin types $2\A_1$ and $\A_3$ when $d=4$.
\end{remark} 

Note that for $d\geq 3$, the possible geometric types of del Pezzo surfaces are in bijection with the orbits given by Borel-de Siebenthal theorem. This is no longer true when $d\leq 2$: for instance, the orbits corresponding to certain types for $d=1$ or $d=2$ only occur as geometric types of del Pezzo surfaces in characteristic $2$ \cite[Remark 1.5]{kn2}.

A convenient way to represent the geometric type is to consider a new graph, containing the intersection graph of the $(-2)$-curves as a subgraph \cite{ct}. 

\begin{definition}
\label{negacurves}
We define the \emph{graph of negative curves} associated to the surface $X$ as the graph whose vertices are circles corresponding to $(-2)$-curves, and points corresponding to exceptional curves. Two vertices corresponding to curves $C$ et $C'$ are joined by $n$ edges if we have $C\cdot C'=n$.
\end{definition}

Since the Weyl group preserves the intersection product, two surfaces sharing the same geometric type have the same graph of negative curves. The converse is true \cite[Remark 4]{der}; actually the geometric type of a weak del Pezzo surface is completely determined by its degree, its Dynkin type and the number of its exceptional curves. As a consequence, this is the way we will encode it in the tables of the Appendix.

Here is a criterion for an exceptional divisor to be irreducible \cite[Corollaire au Théorème III.2]{dem}

\begin{lemma}
\label{cbeexc}
Let $D$ denote an exceptional divisor of $X$. Then $D$ is the class of an exceptional curve if, and only if we have $D\cdot R\geq 0$ for any effective root (or $(-2)$-curve) $R$.
\end{lemma}

We end with \cite[Lemme IV.2]{dem}, that will be useful we we decribe the fibers of the desingularization of a songular del Pezzo surface

\begin{lemma}
Each connected component $B$ of $A$ is the support of a unique fundamental cycle, which is the least effective root $C$ such that for any irreducible component $D$ of $B$ we have $C\cdot D\leq 0$.
\end{lemma}

The fundamental cycle depends only on the corresponding connected component of the Dynkin type. Actually it is the highest root of the root system associated to this component \cite[Section 8.2.7]{dolga}.

\subsection{Arithmetic types}

We assume here that $k=\F_q$ is a finite field. We denote by $\sigma$ a generator of the absolute Galois group $\Gamma:=\Gal(\overline{k}/k)$.

Let $X$ denote a weak del Pezzo surface defined over $k$, having degree $d$. We denote by $\R_{\irr}(X)$ the set of the $(-2)$-curves of $X\otimes\overline{k}$, and by $\R_{\irr}$ a representative of the geometric type of $X$. We fix an isomorphism between $\Pic(X\otimes \overline{k})$ and $\Z E_0+\ldots+\Z E_{9-d}$ that send $\R_{\irr}(X)$ to $R_{\irr}$.

The automorphism $\sigma$ induces the automorphism $\Id\times \sigma$ of the surface $X\otimes \overline{k}$, and an automorphism $(\Id\times \sigma)^\ast$ of the group $\Pic(X\otimes\overline{k})$, that we denote by $\sigma^\ast$ in the following; it preserves the intersection pairing, and the canonical class since $X$ is defined over $k$. Under the action of the isomorphism between $\Pic(X\otimes \overline{k})$ and $\Z E_0+\ldots+\Z E_{9-d}$, the image of $\sigma^\ast$ is an element of the Weyl group $\W(\E_{9-d})$, and we get a morphism from $\Gamma$ to $\W(\E_{9-d})$, whose image is a finite cyclic group.

Finally, $\sigma^\ast$ preserves the set of $(-2)$-curves, and its image in $\W(\E_{9-d})$ must lie in $\Stab( \R_{\irr} )$, the stabilizer of $\R_{\irr}$ in $\W(\E_{9-d})$ which is the subgroup consisting of the $\theta$ such that $\theta\R_{\irr}=\R_{\irr}$. 

This motivates the following

\begin{definition}
Let $X$ denote a weak del Pezzo surface defined over $k$, and $\R_{\irr}$ the image of the set of its $(-2)$-curves described above. 

The \emph{arithmetic type} $\mathfrak{T}$ of $X$ is the conjugacy class of the image of $\sigma^\ast$ in $\Stab(\R_{\irr})$.
\end{definition}

Two isomorphisms between $\Pic(X\otimes\overline{k})$ and $\Z E_0+\ldots+\Z E_{9-d}$ sending the $(-2)$-curves of $X$ to $\R_{\irr}$ differ by an element of the above stabilizer, which is a subgroup of $W(\E_{9-d})$. Defining the type as a conjugacy class in this group makes it independent of the choice of such an isomorphism.

\begin{remark}
We will see below that two singular del Pezzo surfaces such that the Frobenius actions on the Picard groups of their associated weak del Pezzo surfaces lie in the same conjugacy class in $W(\E_{9-d})$ (not in the stabilizer) can have different arithmetic properties, in particular different zeta functions. This is easily seen on the tables given in Appendix \ref{gt} that describe the different arithmetic types for degrees $3\leq d\leq 6$, in particular in degree $4$ where we precise the corresponding conjugacy classes in $W(\D_{5})$.
\end{remark}

To end this section, we remark that two weak del Pezzo surfaces sharing the same arithmetic type have the same zeta function. Actually, since a weak del Pezzo surface is rational, we have the following \cite[Theorem 27.1, Corollary 27.1.2]{manin}

\begin{proposition}
\label{zetafaible}
For any $n\geq 1$, the number of rational points of the weak del Pezzo surface $X$ over the finite field $\F_{q^n}$ is
\[
\#X(\F_{q^n})=q^{2n}+q^n\Tr(\sigma^{\ast n})+1
\]
As a consequence, the zeta function of the weak del Pezzo surface $X$ is
\[
Z(X,T)^{-1}=(1-T)(1-q^2T)\det\left(\I-qT\sigma^\ast |\Pic(X\otimes\overline{k})\right)
\]
\end{proposition}

\section{Singular del Pezzo surfaces}
\label{sec2}

In this section, we consider a weak del Pezzo surface $X$ of degree $d$ defined over the finite field $k=\F_q$.

If $X$ is not ordinary, its anticanonical divisor is no longer ample, and the morphism it (or one of its multiples) defines is no longer an embedding; its image $X_s$ is singular.

In this section, we first describe the geometry of this image, then we determine the divisor class groups of $X$ and the associated singular variety in order to prove Theorem \ref{zetasing}.

\subsection{Anticanonical models: geometric aspects}

\begin{definition}
\label{anticano}
The anticanonical model of the surface $X$ is the variety

\[
X_s := \Proj \bigoplus_{n=0}^\infty H^0(X,-nK_X)
\]

and we denote by $\varphi : X\rightarrow X_s$ the associated morphism.

The variety $X_s$ just defined is the \emph{singular del Pezzo surface} associated to $X$.
\end{definition}

\begin{remark}
One can also consider for $i\geq 1$, the plurianticanonical linear system $|-iK_X|$ and the image $X^{(i)}$ of the morphism it defines. This variety is isomorphic to $X_s$ as long as $di\geq 3$ \cite[V Théorème 1]{dem}. As a consequence, the variety $X_s$ can be identified with the anticanonical image of $X$ when $d\geq 3$.

The anticanonical model of a degree $d$ del Pezzo surface is \cite[Theorem 3.5]{kollar} 
\begin{itemize}
	\item[(i)] if $d=4$, the complete intersection of two quadrics in $\P^4$;
	\item[(ii)] if $d=3$, a cubic in $\P^3$;
	\item[(iii)] if $d=2$, a degree four hypersurface in weighted projective space $\P(1,1,1,2)$;
	\item[(iv)] if $d=1$, a degree six hypersurface in $\P(1,1,2,3)$.
\end{itemize}
\end{remark}

We list below some properties of the surface $X_s$, and of the morphism $\varphi$ \cite[IV Théorème 1, V Proposition 1, V Théorèmes 1 et 2]{dem}. Recall that $A$ is the union of the $(-2)$-curves on $X$, and $U$ is its complementary

\begin{theorem}
1) The schematic fibers of the morphism $\varphi$ are the points of $U$ and the fundamental cycles. As a consequence, $\varphi$ is birational, and we have $\varphi_{\ast} \O_X=\O_{X_s}$.

2) for all $n\in \Z$, $i>0$, we have $\R^i\varphi_{\ast}\O(nK_X)=0$.

3) The surface $X_s$ is normal. The singular points of $X_s$ are the images of the fundamental cycles; they are rational double points.

4) If we set $\O(K_{X_s}):=\varphi_{\ast}\O(K_X)$, then $\O(K_{X_s})$ is locally free of rank $1$, and for every integer $n$ we have canonical isomorphisms
\[
\O(nK_{X_s})=\varphi_{\ast}\O(nK_X),~\O(nK_{X})=\varphi^{\ast}\O(nK_{X_s})
\]
\end{theorem}

In other words, $\varphi$ is an isomorphism from $U$ on its image $U_s$, and each connected component of $A$ is sent to a point which is a rational double point on $X_s$. We obtain all singular points of $X_s$ in this way. The map $\varphi$ is a minimal resolution of the singularities of $X_s$, and the Dynkin type of $X$ is the dual graph to this resolution. Thus the singularity type of $X_s$ is exactly the Dynkin type of $X$.

The type of a singularity $x$ of $X_s$ is the type of the connected component corresponding to $x$ in the intersection graph of $(-2)$-curves of $X$. Since all singularities are rational double points, their types fall in the $\A\D\E$ classification.

Note also that since $\O(K_{X_s})$ is locally free of rank $1$, this invertible sheaf corresponds to a Cartier divisor $K_{X_s}$.

We end by recalling the following result \cite[V Corollaire 2]{dem}. 

\begin{theorem}
\label{secsing}
Let $\FFF$ denote a locally free sheaf on $X_s$. Then, for any $i\geq 0$, we have
\[
H^i(X_s,\FFF)=H^i(X,\varphi^\ast\FFF)
\]

Moreover, we have Serre duality
\[
H^i(X_s,\FFF)=H^{2-i}(X_s,\O(K_{X_s})\otimes\FFF^\vee)^\vee
\]
\end{theorem}

From this result, we can describe the global sections of the Cartier divisors on $X_s$. Moreover we see that the sheaf $\O(K_{X_s})$ is dualizing; since it is locally free, the surface $X_s$ is Gorenstein.

\subsection{Divisor class groups and zeta functions}

Let $X$ denote a weak non ordinary del Pezzo surface defined over $k=\F_q$, and $X_s$ its anticanonical model. Then $\varphi$ and $X_s$ are defined over $k$ since the anticanonical divisor is. In the same way, the varieties $\Sing(X_s)$ of dimension zero, and $A$ of dimension $1$ are defined over $k$.

Recall that the geometric Picard group (the group of classes of Cartier divisors) $\Pic(X\otimes\overline{k})$ identifies to the free $\Z$-module generated by $E_0$ and $E_1,\ldots,E_r$, with $r=9-d$. It is equal to the group of classes of Weil divisors $\Cl(X\otimes\overline{k})$ since $X$ is smooth. Using this identification, we will no longer mention the dependance on $X$ of some objects such as the root modules.

We first describe the groups $\Cl(X_s\otimes\overline{k})$ and $\Pic(X_s\otimes\overline{k})$. Note that since $X_s$ is normal, the groups of Cartier divisors and of invertible sheaves coincide.

The restriction to $U\otimes\overline{k}$ of Weil divisors $X\otimes\overline{k}$ is surjective \cite[Proposition II.6.5]{hart}; its kernel consists of divisors whose support is contained in the complementary of $U\otimes\overline{k}$, i.e. in $A\otimes\overline{k}$. This last group is $\RR$, since it is generated by the irreducible components of $A\otimes\overline{k}$, and these are exactly the $(-2)$-curves.

No principal divisor on $X\otimes\overline{k}$ has support contained in $A\otimes\overline{k}$ \cite[p 225]{lip}. Thus $\RR$ remains the kernel of the restriction of classes of Weil divisors from $\Cl(X\otimes\overline{k})$ to $\Cl(U\otimes\overline{k})$. 

The morphism $\varphi$ induces an isomorphism from $U$ to $U_s$, and we have $\Cl(U_s\otimes\overline{k})=\Cl(U\otimes\overline{k})$. Now $X_s\otimes\overline{k}\setminus U_s\otimes\overline{k}$ has codimension $2$ in $X_s\otimes\overline{k}$, and we deduce \cite[Proposition II.6.5]{hart} that $\Cl(U_s\otimes\overline{k})=\Cl(X_s\otimes\overline{k})$. We get

\begin{equation}
\label{cl}
\xymatrix{
0 \ar@{->}[r] & \RR \ar@{->}[r] & \Cl(X\otimes\overline{k})=\Pic(X\otimes\overline{k}) \ar@{->}[r] & \Cl(X_s\otimes\overline{k}) \ar@{->}[r] & 0 \\
}
\end{equation}

We come to the Picard group. The pull-back $\varphi^\ast:\Pic(X_s\otimes\overline{k}) \rightarrow \Pic(X\otimes\overline{k})$ gives rise to the exact sequence \cite[Proposition 1]{bri}

\begin{equation}
\label{pic}
\xymatrix{
0 \ar@{->}[r] & \Pic(X_s\otimes\overline{k}) \ar@{->}[r]^{\varphi^\ast} & \Pic(X\otimes\overline{k}) \ar@{->}[r]^{~~~~\theta} & \RR^\vee \ar@{->}[r] & \Br(X_s\otimes\overline{k}) \ar@{->}[r]^{\varphi^\ast} & \Br(X\otimes\overline{k})\\
}
\end{equation}
where we have set $\RR^\vee:=\Hom(\RR,\Z)$, and $\theta$ comes from the intersection product: for any $D\in \Pic(X\otimes\overline{k})$, $\theta(D):\RR\rightarrow \Z$ is defined by $\theta(D)(R)=D\cdot R$.

In other words, the group $\Pic(X_s\otimes\overline{k})$ identifies to the following subgroup of $\Pic(X\otimes\overline{k})$
\[
\begin{array}{rcl}
\Pic(X_s\otimes\overline{k}) & = & \{D\in \Pic(X\otimes\overline{k}),~\forall R\in \RR,~D\cdot R=0\}\\
& = & \{D\in \Pic(X\otimes\overline{k}),~\forall R\in \R_{\irr},~D\cdot R=0\}\\
\end{array}
\]

Since $X\otimes\overline{k}$ is a rational surface, its Brauer group is trivial. We deduce the equality $\Coker \theta=\Br(X_s\otimes\overline{k})$, and we have the following explicit description of this last group \cite[Proposition 4]{bri}. The root module $\RR$ is a subgroup of $K_X^\perp$, and $\Br(X_s\otimes\overline{k})$ is the torsion subgroup of the quotient
\[
\Coker \theta=\Br(X_s\otimes\overline{k})=\left(K_X^\perp/\RR\right)_{\tors}
\]
It depends on the geometric type defined above (not only on the Dynkin type in general: for instance there are two orbits for the Dynkin type $4\A_1$ when $d=2$, one gives a trivial cokernel, the other an order $2$ cokernel), and the different cases are described in \cite[Theorem 6]{bri}. Note that $\theta$ is often surjective (this is always true when $d\geq 5$).

We turn to the study of the zeta function. We first determine the zeta function of the union $A$ of the $(-2)$-curves in $X$.

Since the set $\R_{\irr}$ is stable under the action of $\Gamma$, the action of $\sigma^\ast$ on $\Pic(X\otimes\overline{k})$ restricts to an action on $\RR$. Moreover this action preserves the intersection product, and it induces an automorphism on the singular graph; as a consequence it permutes the $(-2)$-curves.  The matrix of the action of $\sigma^\ast$ over $\RR$ with respect to the basis $\R_{\irr}$ is a permutation matrix.

We can express the zeta function of $A$ in terms of the characteristic polynomial of this action

\begin{lemma}
\label{zetarac}
We have the equality
\[
Z(A,T)=Z(\Sing(X_s),T)\det(\I-qT\sigma^\ast|\RR)^{-1}
\]
\end{lemma}

\begin{proof}

Write $A$ as the disjoint union of its connected components $A_x$, where for any $x\in \Sing(X_s)(\overline{\F}_{q})$, $A_x$ is the fiber (seen as a set) $\varphi^{-1}(\{x\})$.

Let $n\geq 1$ an integer; since $\varphi$ is defined over $k$, we have $A_{x^{\sigma^n}}=A_x^{\sigma^n}$. If we have $x\notin \Sing(X_s)(\F_{q^n})$, then $A_x^{\sigma^n}\cap A_x=\emptyset$, and we deduce that $A_x(\F_{q^n})=\emptyset$. 

Now assume $x\in \Sing(X_s)(\F_{q^n})$, and let us denote by $C_{x1},\ldots,C_{xk}$ the absolutely irreducible components of $A_x$, whuch form the Coxeter-Dynkin graph associated to $x$. 

For any two curves $C_{xi}$ and $C_{xj}$ defined over $\F_{q^n}$, there is a unique (since the graph has no cycle) chain with minimal length in the graph between them. As the extremities of this chain are fixed by $\sigma^n$, the whole chain is fixed. We deduce from this fact that the subgraph of the $C_{xi}$, $1\leq i\leq k$ defined over $\F_{q^n}$ is connected; let us denote $N_{nx}$ the number of its vertices. As the $(-2)$-curves have normal crossings, we deduce

\begin{eqnarray*}
\# A_x(\F_{q^n}) & = & \sum_{i=1}^{N_{nx}} \# C_{xi}(\F_{q^n}) - \sum_{1\leq i<j\leq N_n} \# (C_{xi}\cap C_{xj})(\F_{q^n})\\
& = & N_{nx}(1+q^n)-\sum_{1\leq i<j\leq N_{nx}} C_{xi}\cdot C_{xj}\\
& = & 1+N_{nx}q^n\\
\end{eqnarray*}

where the last equality comes from the fact that the sum of the intersection products is the number of edges in the subgraph whose vertices are the $C_{xi}$ defined over $\F_{q^n}$. Since this graph is connected, has $N_{nx}$ vertices and type $\A,\D,\E$, it contains exactly $N_{nx}-1$ edges.

Finally, if none of the absolutely irreducible components of $A_x$ is defined over $\F_{q^n}$, then $\sigma^n$ acts without fixed points over the Coxeter-Dynkin graph associated to $x$. The only graphs admitting such an action are the $\A_{2k}$, and the action is the symmetry around the center of the graph, of order $2$. In this case, the unique point defined over $\F_{q^n}$ is the intersection of the effective roots $C_{xk}$ and $C_{xk+1}$.

Summing up, we have
\[
\# A_x(\F_{q^n})=\left\{ 
\begin{array}{rl}
1+ N_{nx}q^n & \textrm{~if~} x\in \Sing(X_s)(\F_{q^n})\\
0 & {\rm else}\\
\end{array}
\right.
\]

Thus, for any $n\geq 1$ we have
\[
\# A(\F_{q^n})=\# \Sing(X_s)(\F_{q^n})+N_nq^n
\]
where $N_n$ is the number of $(-2)$-curves on $X$ defined over $\F_{q^n}$. 

Replacing this expression in the usual expansion of zeta functions, we get

\[
Z(A,T)=\exp\left(\sum_{n\geq 1} \# A(\F_{q^n}) \frac{T^n}{n}\right)=Z(\Sing(X_s),T)\exp\left(\sum_{n\geq 1}N_n \frac{(qT)^n}{n}\right)
\]

Since the matrix of $\sigma^{\ast n}$ with respect to the basis $\R_{\irr}(X)$ is a permutation matrix, its trace is exactly the number of elements in the basis which remain invariant under the action of $\sigma^{\ast n}$. We deduce that $N_n=\Tr(\sigma^{\ast n}|\RR)$, and the result comes from the expansion of the characteristic polynomial of an endomorphism in terms of the traces of its iterates.

\end{proof}

We are ready to prove Theorem \ref{zetasing}.

\begin{proof}
 
From the above results, the following equalities hold for any $n\geq 1$
\[
\begin{array}{rcll}
\# X_s(\F_{q^n}) & = & \# U_s(\F_{q^n}) + \#\Sing(X_s)(\F_{q^n}) & {\rm excision~for~} X_s \\
 & = & \# U(\F_{q^n})+ \#\Sing(X_s)(\F_{q^n}) & {\rm isomorphim~} U\simeq U_s \\
 & = & \# U(\F_{q^n})+\# A(\F_{q^n})-q^n \Tr(\sigma^{\ast n}|\RR) & {\rm Lemma~} \ref{zetarac}\\
 & = & \# X(\F_{q^n})-q^n \Tr(\sigma^{\ast n}|\RR) & {\rm excision~for~} X \\
\end{array}
\]

Now from Proposition \ref{zetafaible}, we have $\#X(\F_{q^n})=q^{2n}+q^n\Tr(\sigma^{\ast n})+1$, and
\[
\# X_s(\F_{q^n})=q^{2n}+q^n\left(\Tr(\sigma^{\ast n})-\Tr(\sigma^{\ast n}|\RR)\right)+1
\]

%
%
%
%

If we pass to zeta functions as in the proof of Proposition \ref{zetafaible}, we get
\[
Z(X_s,T)^{-1}=(1-T)(1-q^2T)\det\left(\I-qT\sigma^\ast |\Pic(X\otimes\overline{k})\right)\det(\I-qT\sigma^\ast|\RR)^{-1}
\]

Tensoring the exact sequence of $\Z$-modules (\ref{pic}) with $\C$, the torsion disappears and we get the exact sequence of vector spaces
\begin{equation*}
\xymatrix{
0 \ar@{->}[r] & \Pic(X_s\otimes\overline{k})\otimes \C \ar@{->}[r] & \Pic(X\otimes\overline{k})\otimes \C \ar@{->}[r] & \RR^\vee\otimes \C \ar@{->}[r] & 0 \\
}
\end{equation*}
Finally, the action of $\sigma^\ast$ on $\RR^\vee\otimes \C$ is adjoint to that of $\sigma^{\ast-1}$ over $\RR\otimes \C$, and they share the same characteristic polynomial. Finally the matrix of $\sigma^\ast$ in the basis $\R_{\irr}$ of the the $\Z$-module $\RR$ is a permutation matrix; thus the actions of $\sigma^{\ast-1}$ and $\sigma^{\ast}$ also share the same characteristic polynomial. We get the equality
\[
\det\left(\I-qT\sigma^\ast |\Pic(X\otimes\overline{k})\right)=\det(\I-qT\sigma^\ast|\RR)\det\left(\I-qT\sigma^\ast |\Pic(X_s\otimes\overline{k})\right)
\]

This ends the proof.
\end{proof}

\section{Construction of weak del Pezzo surfaces of degree at least five}
\label{sec3}

The aim of this section is to show that for any finite field $\F_q$ and arithmetic type belonging to a degree $d\geq 5$, there exists a weak del Pezzo surface of this type defined over $\F_q$. In order to do this, we give explicit constructions.

These surfaces are birational to the projective plane $\P^2$ \cite[Lemma 9.3]{ct}, and our constructions are sequences of blowups and contractions. The fundamental remark is that the configurations of points to be blown-up exist over any finite field. For instance, we will often need three collinear points, but this is possible over any finite field since a line defined over $\F_q$ contains $q+1$ rational points.

We will often use the graphs of negative curve from Definition \ref{negacurves} in order to represent the geometric type. 

\subsection{Degrees seven and eight}

First note that in this case, there is only one geometric type (of Dynkin type $\A_1$), and one arithmetic type.

From Proposition \ref{class1}, the only degree $8$ weak, non ordinary del Pezzo surface is the Hirzebruch surface $F_2$; the only degree $8$ singular del Pezzo surface is its image by the contraction of its negative section. We construct these surfaces from degree $7$ surfaces below.

Again from Proposition \ref{class1}, all the degree $7$ del Pezzo surfaces over an algebraically closed field are obtained from the projective plane $\P^2$ by two successive blowups at points $p_1,p_2$. Such a surface is not ordinary if and only if the center of the second blowup is a point $p_2\succ p_1$ of the exceptional divisor of the first one. In this case, the effective root is the strict transform of this exceptional divisor, whose class in the Picard group is $E_1-E_2$. We get the following graph of negative curves

\begin{center}
\begin{tikzpicture}[baseline=0]
\coordinate (R1) at (0,0);
\coordinate (E1) at (2,0);
\coordinate (E2) at (4,0);
\draw (R1) -- (E1) -- (E2);
\draw (R1) node[above] {\tiny$E_1-E_2$} circle (2pt) ;
\draw[fill] (E1) node[above] {\tiny$E_2$} circle (2pt) ;
\draw[fill] (E2) node[above] {\tiny$E_0-E_1-E_2$} circle (2pt) ;

\end{tikzpicture}
\end{center}

In order to get a non ordinary del Pezzo surface of degree $7$ defined over the finite field $\F_q$, we have to choose rational points $p_1,p_2$ for the centers of the blowups.

The Hirzebruch surface can be obtained from the degree $7$ non ordinary del Pezzo surface by contracting the strict transform of the line $(p_1p_2)$ (the unique line in $\P^2$ passing trough $p_1$ and whose strict transform after the first blowup passes trough $p_2$).

\subsection{Degree six}

Recall that the Weyl group is generated by the reflections $s_{E_1-E_2}$, $s_{E_2-E_3}$ and $s_{E_0-E_1-E_2-E_3}$. It is isomorphic to the dihedral group of order $12$.

The possible geometric types are given in \cite[Section 8.4.2]{dolga} or \cite[Proposition 8.3]{ct}. Except in one case, we follow the choices made in this last article for the set of roots representing the corresponding orbit. As a consequence, most of the graphs of negative curve are already described in the above mentioned article, and we do not recall them. 

We refer to the corresponding table in the appendix for the different arithmetic types.

\begin{itemize}
	\item[$\A_1$(4)] to obtain the split surface (arithmetic type $1$) we blowup three non collinear points rational points $p_1\prec p_2$ and $p_3$ (in other words, the strict transform of the line $(p_1p_3)$ after the first blowup does not pass through $p_2$). 

The stabilizer of the root $E_1-E_2$ is the subgroup of order $2$ generated by the reflection $s_{E_0-E_1-E_2-E_3}$. Its elements fall into two conjugacy classes.

In order to construct a surface of arithmetic type $2$, we will contract a line on a weak degree $5$ del Pezzo surface $X(\Sigma)$ of geometric type $\A_1$, obtained by blowing-up four points $p_1,\ldots,p_4$ in $\P^2$, the first three being collinear (see the degree $5$ surface of geometric type $\A_1$ below).  

If we contract the exceptional curve $E_0-E_3-E_4$ on such a surface, the vertices $E_3$ and $E_4$ disappear, and we obtain a degree $6$ del Pezzo surface with the graph of negative curves

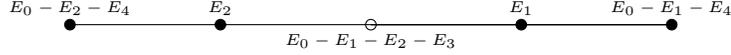
\begin{center}
\begin{tikzpicture}[baseline=0]
\coordinate (R1) at (0,0);
\coordinate (E1) at (2,0);
\coordinate (E2) at (4,0);
\coordinate (E3) at (-2,0);
\coordinate (E4) at (-4,0);
\draw (E4) -- (E2) -- (R1);
\draw (R1) -- (E1) -- (E2);
\draw (R1) node[below] {\tiny$E_0-E_1-E_2-E_3$} circle (2pt) ;
\draw[fill] (E1) node[above] {\tiny$E_1$} circle (2pt) ;
\draw[fill] (E2) node[above] {\tiny$E_0-E_1-E_4$} circle (2pt) ;
\draw[fill] (E4) node[above] {\tiny$E_0-E_2-E_4$} circle (2pt) ;
\draw[fill] (E3) node[above] {\tiny$E_2$} circle (2pt) ;
\end{tikzpicture}
\end{center}

Assume that the first two points $p_1,p_2\in \P^2(\F_{q^2})$ are conjugate over $\F_{q}$, and $p_3,p_4$ are defined over $\F_q$. The exceptional curve $E_0-E_3-E_4$ is defined over $\F_q$. Contracting it, we get a surface $X$ defined over $\F_q$ of degree $6$, of the geometric type under consideration, which is not split: this is a surface of arithmetic type $2$.

	\item[$\A_1$(3)] we blow-up three collinear points in $\P^2$; we get the three exceptional curves $E_1,E_2,E_3$, and the effective root $E_0-E_1-E_2-E_3$ which corresponds to the strict transform of the line through the $p_i$. 
	
The stabilizer of the root in the Weyl group $W(\E_6)$ is the subgroup generated by the reflections $s_{E_2-E_1}$ and $s_{E_2-E_3}$, which is isomorphic to the symmetric group $\mathfrak{S}_3$.

There are three conjugacy classes in this group, corresponding to the three arithmetic types $3,4,5$. We construct a surface of each type in the following way
\begin{itemize}
	\item[3 :] the points $p_1,p_2,p_3$ are defined over $\F_q$;
	\item[4 :] the points $p_1,p_2=p_1^\sigma$ are defined over $\F_{q^2}$, and $p_3$ over $\F_q$;
	\item[5 :] the points $p_1,p_2=p_1^\sigma,p_3=p_1^{\sigma^2}$ are defined over $\F_{q^3}$.
\end{itemize}

	\item[$2\A_1$] we blowup three collinear points $p_1\prec p_2$ and $p_3$ (the point $p_2$ is the intersection of the strict transform of the line $(p_1p_3)$ with the exceptional curve of the first blowup). The exceptional curves on the resulting surface are $E_2,E_3$, and the $(-2)$-curves correspond to the classes $E_1-E_2$ and $E_0-E_1-E_2-E_3$. 
	
The stabilizer of the set $\{E_1-E_2,E_0-E_1-E_2-E_3\}$ in the Weyl group is trivial, and there is exactly one arithmetic type, with number $6$.

	\item[$\A_2$] we blowup three non collinear points $p_1\prec p_2 \prec p_3$ (the point $p_3$ is not the intersection of the strict transform of $(p_1p_2)$ with the exceptional divisor of the second blowup). The exceptional curves correspond to the classes $E_0-E_1-E_2$, $E_3$, and the $(-2)$-curves to $E_1-E_2$, $E_2-E_3$.

The stabilizer of the set $\{E_1-E_2,E_2-E_3\}$ in the Weyl group is generated by the reflection $s_{E_0-E_1-E_2-E_3}$, it has order $2$. As a consequence, there are two arithmetic types, one split that can be constructed by choosing rational points above, the other non-split.

In order to construct the non-split one, we start with a degree $5$ del Pezzo surface of geometric type $\A_2$ and arithmetic type $7$ (see below). If we contract the exceptional curve with class $E_0-E_1-E_2$ (which is defined over $\F_q$), we get a degree $6$ del Pezzo surface of arithmetic type $8$.

	\item[$\A_1\A_2$] we blowup three collinear points $p_1\prec p_2 \prec p_3$, where the point $p_3$ is the intersection of the strict transform of $(p_1p_2)$ with the exceptional divisor of the second blowup. We get an exceptional curve of class $E_3$, and $(-2)$-curves corresponding to the classes $E_1-E_2$, $E_2-E_3$, $E_0-E_1-E_2-E_3$.

Once again, the stabilizer is trivial and the only arithmetic type (number 9) is split.
\end{itemize}

\subsection{Degree five}

Here the Weyl group is generated by the reflections $s_{E_1-E_2}$, $s_{E_2-E_3}$, $s_{E_3-E_4}$ and $s_{E_0-E_1-E_2-E_3}$. It is isomorphic to the symmetric group $\mathfrak{S}_5$.

In this section, we shall not systematically follow the choices made in \cite[Proposition 8.5]{ct} for the classes of $(-2)$-curves representing each geometric type (see also \cite[Section 8.5.1]{dolga} for a list of these types). The reason is that there are many ways to construct a del Pezzo surface of a given geometric type by playing on the configuration of the points we blow up. Here we adopt the configuration that seems best adapted to the description of the arithmetic types (in some sense the most ``symmetric '' one), and it need not coincide with the one described in the above article. We rewrite the graphs when they differ from \cite{ct}.

We refer to the corresponding table in the Appendix for the different arithmetic types.

\begin{itemize}
	\item[$\A_1$]  We blowup four points $p_1,\ldots,p_4$ in $\P^2$ with the first three collinear. The only $(-2)$-curve is the strict transform of the line through these points, its class is $E_0-E_1-E_2-E_3$; moreover there exist seven exceptional curves, and the graph of negative curves is

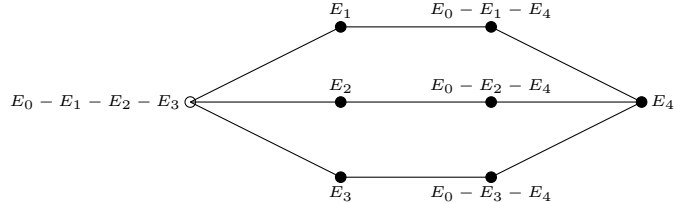
\begin{tikzpicture}[baseline=0]
\coordinate (R1) at (0,0);
\coordinate (E1) at (2,1);
\coordinate (E2) at (2,0);
\coordinate (E3) at (2,-1);
\coordinate (E4) at (4,1);
\coordinate (E5) at (4,0);
\coordinate (E6) at (4,-1);
\coordinate (E7) at (6,0);
\draw (R1) -- (E1) -- (E4) -- (E7);
\draw (R1) -- (E2) -- (E5) -- (E7);
\draw (R1) -- (E3) -- (E6) -- (E7);
\draw (R1) node[left] {\tiny$E_0-E_1-E_2-E_3$} circle (2pt) ;
\draw[fill] (E1) node[above] {\tiny$E_1$} circle (2pt) ;
\draw[fill] (E2) node[above] {\tiny$E_2$} circle (2pt) ;
\draw[fill] (E3) node[below] {\tiny$E_3$} circle (2pt) ;
\draw[fill] (E4) node[above] {\tiny$E_0-E_1-E_4$} circle (2pt) ;
\draw[fill] (E5) node[above] {\tiny$E_0-E_2-E_4$} circle (2pt) ;
\draw[fill] (E6) node[below] {\tiny$E_0-E_3-E_4$} circle (2pt) ;
\draw[fill] (E7) node[right] {\tiny$E_4$} circle (2pt) ;
\end{tikzpicture}

The stabilizer of the root $E_0-E_1-E_2-E_3$ is generated by the reflections $s_{E_1-E_2}$, $s_{E_2-E_3}$, and is isomorphic to the symmetric group $\mathfrak{S}_3$. There are three conjugacy classes, giving rise to three arithmetic types.

The corresponding surfaces are easily constructed, playing on the definition fields of the points $p_1,p_2,p_3$ (we always choose $p_4\in \P^2(\F_q)$)
\begin{itemize}
	\item[1 :] the three points are in $\P^2(\F_q)$;
	\item[2 :] there are two conjugate points in $\P^2(\F_{q^2})$, and one in $\P^2(\F_q)$;
	\item[3 :] the three points are conjugate in $\P^2(\F_{q^3})$.
\end{itemize}

  \item[$2\A_1$]  Here we follow \cite[Proposition 8.5 2.]{ct}; we blowup $p_1\prec p_2,p_3\prec p_4$ such that the strict transform of the line $(p_1p_3)$ does not meet the exceptional divisor of the first blowup at $p_2$ nor of the third at $p_4$. We get two disjoint $(-2)$-curves with classes $E_1-E_2$ and $E_3-E_4$, and five exceptional curves. 
The stabilizer of the set of negative curves is the order two subgroup of $W(\A_4)$ generated by $s_{E_1-E_3}\circ s_{E_2-E_4}$.

We get the two arithmetic types in the following way
\begin{itemize}
	\item[4 :] if $p_1\prec p_2,p_3\prec p_4$ are defined over $\F_q$, the surface is split;
	\item[5 :] we choose couples of conjugate points $p_1\prec p_2,~p_3=p_1^\sigma\prec p_4=p_2^\sigma$ defined over $\F_{q^2}$.
\end{itemize}
  \item[$\A_2$]  We blow up $p_1,p_3,p_4$ collinear, then $p_2\succ p_1$ which does not lie on the strict transform of the line $(p_1p_3)$; we obtain the following graph of negative curves
	
	
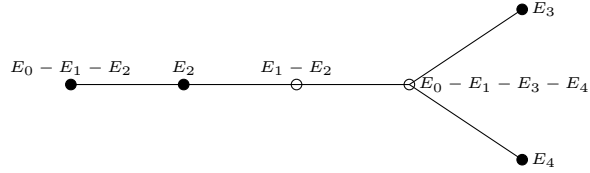
\begin{tikzpicture}[baseline=0]
\coordinate (E1) at (0,0);
\coordinate (E2) at (1.5,0);
\coordinate (R1) at (3,0);
\coordinate (R2) at (4.5,0);
\coordinate (E3) at (6,1);
\coordinate (E4) at (6,-1);
\draw (E1) -- (E2) -- (R1) -- (R2) -- (E3);
\draw (R2) -- (E4);
\draw (R1) node[above] {\tiny$E_1-E_2$} circle (2pt) ;
\draw (R2) node[right] {\tiny$E_0-E_1-E_3-E_4$} circle (2pt) ;
\draw[fill] (E1) node[above] {\tiny$E_0-E_1-E_2$} circle (2pt) ;
\draw[fill] (E2) node[above] {\tiny$E_2$} circle (2pt) ;
\draw[fill] (E3) node[right] {\tiny$E_3$} circle (2pt) ;
\draw[fill] (E4) node[right] {\tiny$E_4$} circle (2pt) ;
\end{tikzpicture}

The stabilizer of the set of $(-2)$-curves is the subgroup of the Weyl group generated by $s_{E_3-E_4}$; it has order $2$ and we get two arithmetic types. These are easily described
\begin{itemize}
	\item[6 :] if the $p_i$ are defined over $\F_q$, the surface is split;
	\item[7 :] if we choose $p_4=p_3^\sigma$ defined over $\F_{q^2}$, it is not.
\end{itemize}

  \item[$\A_1\A_2$] We blowup $p_1\prec p_2\prec p_3$ collinear and $p_4\in \P^2$ outside the line $(p_1p_2)$.
	

The stabilizer is trivial, and the only arithmetic type is split.

  \item[$\A_3$] We blowup $p_1\prec p_2\prec p_3 \prec p_4$ where $p_1,p_2,p_3$ are not collinear. The stabilizer is trivial, and the only arithmetic type is split.

  \item[$\A_4$] We blowup $p_1\prec p_2\prec p_3\prec p_4$ where $p_1,p_2,p_3$ are collinear. The stabilizer is trivial, and the only arithmetic type is split.

\end{itemize}

\section{Construction of singular del Pezzo surfaces of degree four.}
\label{sec4}

We adopt a different point of view in this section. Actually not all weak degree four del Pezzo surfaces are birational to the projective plane over a finite field, and the preceding constructions, which use sequences of blowups and contractions, no longer suffice to describe all types.

We shall mostly use a well known property of these surfaces: their anticanonical model is the base locus of a pencil of quadrics in projective space $\P^4$. The classification of these objects is different in characteristic two, and we assume that the characteristic of our base field is odd in this section.

We shall show that singular del Pezzo surfaces of every arithmetic type exist over any finite field with odd cardinality.

\subsection{Pencils of quadrics, their Segre symbols, and geometric types}

In this section, we work over the field $k$, an algebraic closure of a finite field of odd characteristic.

We recall \cite[Section 8.6.1]{dolga}. The results there are described over the field of complex numbers, and our aim is to show that they remain valid over any algebraically closed field of odd characteristic.

The anticanonical model of a weak del Pezzo surface of degree $4$ is the base locus of a pencil of quadrics in projective space $\P^4$ \cite[Theorem 3.5]{kollar}. We denote by $Q_0$ and $Q_\infty$ two distinct quadrics of this pencil. This gives us two quadratic forms over the vector space $k^5$, and two symmetric $5\times 5$ matrices; we denote the characteristic polynomial of the pencil corresponding to this basis by $F(\lambda,\mu):=\det(\lambda Q_0+\mu Q_\infty)$.

A theorem of Kronecker \cite[Theorem 3.1]{water} tells us that the vector space $k^5$ can be written as an orthogonal direct sum of subspaces of two types: 
\begin{itemize}
	\item a \emph{non singular} space, say of dimension $d$, over which the determinant has degree $d$;
	\item some \emph{basic singular} spaces; here the equations of the quadrics take a very special form with respect to a well chosen basis.
\end{itemize}
A direct calculation from such equations shows that the base locus of a pencil of quadrics has one-dimensional singular locus as long as the space contains a basic singular space. Thus the space $k^5$ must be non-singular for the pair of quadratic forms. Note that this is no longer true in characteristic two, even for ordinary degree $4$ del Pezzo surfaces \cite{dd2}, since a quadratic form over an odd dimensional vector space must be degenerate.

We assume in the following that the quadric $Q_\infty$ is non-singular; thus the polynomial $P(t):=F(1,t)$ has degree $5$. This is a harmless assumption since we assume $k$ algebraically closed.

We follow Waterhouse \cite[Section 1]{water} and associate to our pair of quadratic forms a quadratic form $\Phi$ on a $k[T]$-module of finite length. If $\varphi_0$, $\varphi_\infty$ denote the bilinear forms on $k^5$ associated to $Q_0$ and $Q_\infty$, then $\varphi_\infty$ is non degenerate, and there is an unique endomorphism $u$ of $k^5$ such that $\varphi_\infty(u(x),y)=\varphi_0(x,y)$ for any vectors $x,y$.  Moreover $u$ is symmetric with respect to $\varphi_\infty$.

The vector space $k^5$, endowed with the action of $u$, becomes a $k[T]$-module of finite length that we denote by $M$. The non degenerate form $\varphi_\infty$ defines an isomorphism $\Phi$ of $k[T]$-modules between $M$ and its dual $M^\ast:=\Hom_{k[T]}(M,k(T)/k[T])$.

Consider the factorisation of $P$ over $k$, $P(T)=\prod_{i=1}^t (T-\theta_i)^{m_i}$; we get an isomorphism of $k[T]$-modules
\[
M=\oplus_{i=1}^t\oplus_{j=1}^{e_i} \left(k[T]/((T-\theta_i)^{m_{ij}})\right)^{n_{ij}}
\]
As a consequence, we can associate to $M$ the following \emph{Segre symbol}
\begin{equation}
\label{segresymbol}
[(\underbrace{m_{11}\ldots m_{11}}_{n_{11} ~\times}\ldots\underbrace{m_{1e_1}\ldots m_{1e_1}}_{n_{1e_1} \times})\ldots(\underbrace{m_{t1}\ldots m_{t1}}_{n_{t1} \times}\ldots\underbrace{m_{te_t}\ldots m_{te_t}}_{n_{te_t} \times})]
\end{equation}
Note that we remove the parentheses around the $i$th term when it contains exactly one integer, ie when we have $e_i=n_{i1}=1$.

As in \cite[Section 2]{water}, one can write the restrictions of the quadratic forms to a submodule of the form $k[T]/(T-\theta)^{m}$ as (here $\res$ denotes the residue in the sense of Laurent series)
\[
Q_0(x)=\res(T(T-\theta)^{-m} x^2),~Q_\infty(x)=\res((T-\theta)^{-m} x^2)
\]
since every element in $\left(k[T]/(T-\theta)^{m}\right)^\times$ is a square.

In the basis $((T-\theta)^{m-1},\ldots,T-\theta,1)$, the expressions of these forms are exactly those given in \cite[Equation (8.21)]{dolga}. Now the analytic expressions for the rational double points of type $\A_n$ or $\D_n$ (the only ones occuring for degree $4$ singular del Pezzo surfaces) are exactly the same in characteristic $0$ or odd \cite{artin}, and we deduce that the following one to one correspondance between the geometric types and (some of the) Segre symbols \cite[Table 8.6]{dolga} remains true in our setting
$$
\begin{array}{l|l|l||}
\text{\bf Type} & \text{\bf Number} &\text{\bf Segre} \\
\text{\bf of singularity} & \text{\bf of lines} & \text{\bf symbol} \\
\hline
\text{\bf Ordinary}&16&[11111] \\
\hline
\mathbf{A}_1&12&[2111] \\
\hline
2\mathbf{A}_1&9&[221] \\
\hline
2\mathbf{A}_1&8&[(11)111] \\
\hline
\mathbf{A}_2&8&[311] \\
\hline
3\mathbf{A}_1&6&[(11)21] \\
\hline
\mathbf{A}_1\mathbf{A}_2&6&[32] \\
\hline
\mathbf{A}_3&5&[41] \\
\hline
\end{array}
\begin{array}{|l|l|l}
\text{\bf Type} & \text{\bf Number} &\text{\bf Segre} \\
\text{\bf of singularity} & \text{\bf of lines} & \text{\bf symbol} \\
\hline
\mathbf{A}_3&4&[(21)11] \\
\hline
4\mathbf{A}_1&4&[(11)(11)1] \\
\hline
2\mathbf{A}_1\mathbf{A}_2&4&[3(11)] \\
\hline
\mathbf{A}_1\mathbf{A}_3&3&[(21)2] \\
\hline
\mathbf{A}_4&3&[5] \\
\hline
\mathbf{D}_4&2&[(31)1] \\
\hline
2\mathbf{A}_1\mathbf{A}_3&2&[(21)(11)] \\
\hline
\mathbf{D}_5&1&[(41)]\\
\hline
\end{array}
$$
Some Segre symbols do not give rise to singular degree $4$ del Pezzo surfaces. For instance the symbol $[(11111)]$ gives rise to two collinear quadratic forms over $k^5$, and the base locus of the pencil is a quadric in $\P^4$.

We shall often use the following remark along this section

\begin{remark}
\label{rankquad}
These symbols carry the following informations: the number of different terms (ie the number $t$ in (\ref{segresymbol}),  each couple of parentheses counts for one term) is equal to the number of singular quadrics in the pencil, and for such a quadric, its co-rank is the number of terms in the corresponding parenthesis. Actually we remark that the possible ranks for the singular quadrics in the pencil are $3$ or $4$.
\end{remark}

In the following, we shall often use the Segre symbol instead of the geometric type for our contructions. Moreover, we always choose the same set of $(-2)$-curves as in \cite[Proposition 6.1]{ct} to represent the orbit.

\subsection{Morphisms to the projective line}

We continue to work over an algebraically closed field of odd characteristic in this section. We keep on with the notations of the preceding section.

A convenient way to classify the arithmetic types for an ordinary del Pezzo surface is to visualize the Galois action on a graph consisting of its conic bundles \cite[pages 8 and 9]{KST}. A non ordinary del Pezzo surface no longer has conic bundles, due to the presence of $(-2)$-curves; the aim of this section is to define a new graph for each geometric type, which is very close to the original one, whose vertices represent 
\begin{itemize}
	\item the fibers of rational maps or morphisms from $X$ to the projective line,
	\item the classes of $(-2)$-curves.
\end{itemize}
We will use these graphs in the next sections in order to control the arithmetic types.

Let $\varphi:X\rightarrow \P^1$ denote a surjective morphism. Since $X$ is a rational surface, the generic fiber of $\varphi$ is isomorphic to $\P^1$. Moreover all its fibers are linearly equivalent. We deduce from the adjunction formula that their class $C$ satisfies $C^{\cdot 2}=0$ and $C\cdot K_X=-2$; moreover we have $\varphi=\varphi_{|C|}$.  

Recall \cite[Theorem 2]{BBFL}: there exist exactly ten classes in $\Pic(X)$ satisfying the above equations (note that they form \emph{complementary couples} in the sense that we have $C_i+C_i'=-K_X$ for each $i$)
\[
C_i=E_0-E_i,~C_i'=2E_0-\sum_{j=1}^5 E_j+E_i,~1\leq i\leq 5
\]

\begin{definition}
\label{conicbundles}
We set $\CC:=\cup_{i=1}^5\{C_i,C_i'\}$ in the following. We represent these classes in the following graph, in which the couples are materialized by vertical dashed lines.

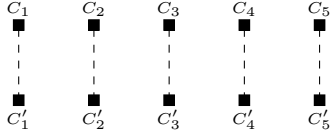
\[
\begin{tikzpicture}[baseline=0]
\coordinate (C1) at (0,0.5);
\coordinate (C2) at (1,0.5);
\coordinate (C3) at (2,0.5);
\coordinate (C4) at (3,0.5);
\coordinate (C5) at (4,0.5);
\coordinate (C1p) at (0,-0.5);
\coordinate (C2p) at (1,-0.5);
\coordinate (C3p) at (2,-0.5);
\coordinate (C4p) at (3,-0.5);
\coordinate (C5p) at (4,-0.5);
\draw[dashed] (C1p) -- (C1);
\draw[dashed] (C2p) -- (C2);
\draw[dashed] (C3p) -- (C3);
\draw[dashed] (C4p) -- (C4);
\draw[dashed] (C5p) -- (C5);
\draw[fill] (C1) node[above] {\tiny$C_1$} +(-2pt,-2pt) rectangle +(2pt,2pt);
\draw[fill] (C2) node[above] {\tiny$C_2$} +(-2pt,-2pt) rectangle +(2pt,2pt);
\draw[fill] (C3) node[above] {\tiny$C_3$} +(-2pt,-2pt) rectangle +(2pt,2pt) ;
\draw[fill] (C4) node[above] {\tiny$C_4$} +(-2pt,-2pt) rectangle +(2pt,2pt) ;
\draw[fill] (C5) node[above] {\tiny$C_5$} +(-2pt,-2pt) rectangle +(2pt,2pt) ;
\draw[fill] (C1p) node[below] {\tiny$C'_1$} +(-2pt,-2pt) rectangle +(2pt,2pt) ;
\draw[fill] (C2p) node[below] {\tiny$C'_2$} +(-2pt,-2pt) rectangle +(2pt,2pt) ;
\draw[fill] (C3p) node[below] {\tiny$C'_3$} +(-2pt,-2pt) rectangle +(2pt,2pt) ;
\draw[fill] (C4p) node[below] {\tiny$C'_4$} +(-2pt,-2pt) rectangle +(2pt,2pt) ;
\draw[fill] (C5p) node[below] {\tiny$C'_5$} +(-2pt,-2pt) rectangle +(2pt,2pt) ;
\end{tikzpicture}
\]
\end{definition}

In the case of singular del Pezzo surfaces, the elements of $\CC$ no longer correspond to conic bundles. But they still correspond to rational maps from $X$ to the projective line, as we shall see below. We will adapt the graph according to the geometric type, by adding vertices corresponding to $(-2)$-curves, and an edge between a class in $\CC$ and such a curve when they intersect positively.

From the description of the elements in $\CC$ on one hand, and of the $(-2)$-curves (recall that their classes have the form $E_i-E_j$ or $E_0-E_i-E_j-E_k$) on the other, we see that we must have $C\cdot R\in \{-1,0,1\}$. 

Let $C\in \CC$ be such that we have $C\cdot R=-1$ for some $(-2)$-curve $R$; we deduce that $C-R\in \CC$. Thus $R$ is a fixed component of the linear system $|C|$, and $\varphi_{|C|}$ is a rational map, but not a morphism.

We deduce that there exists at most
\[
N_X:=\# \{ C\in \CC,~\forall R\in \R_{irr}(X),~C\cdot R\geq 0\}
\]
morphisms from $X$ to $\P^1$.

Our strategy is to construct exactly $N_X$ such morphisms by considering the singular quadrics in the pencil defining the anticanonical model $X_s$. Recall that $P$ denotes the characteristic polynomial of the pencil, and that $\theta_1,\ldots,\theta_t$ are its (pairwise distinct) roots.

Let $\theta_i$ denote a simple root of $P$, if any; the $i$th part of the Segre symbol (\ref{segresymbol}) is a $1$. From remark \ref{rankquad}, the quadric $Q_i:=Q_0-\theta_iQ_\infty$ is a singular quadric in $\P^4$ of rank $4$. Thus it is a cone whose vertex is a point $v_i$ and base a non singular quadric $q_i:=Q_i\cap H_i$ for some hyperplane $H_i$ not containing $v_i$. Moreover the vertex $v_i$ is not contained in $X_s$ since $\theta_i$ is a simple root of $P$.

The projection $\P^4\rightarrow H_i$ from $v_i$ restricts to a double covering $X_s\rightarrow q_i$. Now $q_i$ is non singular, thus it is isomorphic to $\P^1\times \P^1$, and the projections on the factors give rise to two morphisms from $X_s$ to $\P^1$. Composing them with the anticanonical morphism gives two morphisms $\varphi_{i1},\varphi_{i2}:X\rightarrow \P^1$.

Note that we have morphisms from $X_s$ to $\P^1$, thus their fibers are Cartier divisors on $X_s$, and their classes in $\Pic(X)$ satisfy $C\cdot R=0$ for all $(-2)$-curves $R$ on $X$.

Finally, the fibers of the above morphisms from $X_s$ to $\P^1$ are of the form $\Pi\cap Q_\infty$, where $\Pi$ is some plane containing one of the lines of $q_i$ and passing through $v_i$. These are plane conics. For any point $p\in X_s$, the hyperplane $H$ of $\P^4$ tangent to $Q_i$ at $p$ intersects $Q_i$ along the union of two planes $\Pi_1$ and $\Pi_2$, and we have $H\cap X_s= H\cap Q_i\cap Q_\infty=(\Pi_1\cap Q_\infty)\cup (\Pi_2\cap Q_\infty)$. We deduce that the sum of the classes of the fibers of these two morphisms is exactly the class of the anticanonical divisor: they are complementary.

Now assume that $\theta_i$ is a multiple root of $P$, and that $Q_i$ has rank $4$ (this means that the $i$th part of the Segre symbol (\ref{segresymbol}) satisfies $e_i=n_{i1}=1$ and $m_{i1}=m_i>1$). The same construction as in the preceding case yields two morphisms from $X_s\setminus \{v_i\}$ to $\P^1$ (note that the vertex $v_i$ of $Q_i$ is a singular point of $X_s$ here). 

Consider the blowup $\Bl_{v_i}(\P^4)\rightarrow \P^4$; then $\Bl_{v_i}(X_s)$ is the strict transform of $X_s$ \cite[Corollary II.7.15]{hart}, and the two morphisms above become morphisms from $\Bl_{v_i}(X_s)$ to $\P^1$. The fiber of $v_i$ in the blowup $\Bl_{v_i}(X_s)\rightarrow X_s$ contains only $(-2)$-curves, and we must have a morphism $X\rightarrow \Bl_{v_i}(X_s)$ since $X$ is the minimal desingularization of $X_s$. Composing, we get two morphisms $\varphi_{i1},\varphi_{i2}:X\rightarrow \P^1$.

The last case to be considered is when $\theta_i$ is a multiple root of $P$, and $Q_i$ has rank $3$; here the $i$th part of the Segre symbol (\ref{segresymbol}) contains exactly two terms between parentheses. In this case $Q_i$ is a cone with vertex $\ell_i\simeq \P^1$ and base a smooth plane conic $c_i$. The projection with vertex $\ell_i$ from $\P^4\setminus \ell_i$ to this plane induces a morphism from $X_s\setminus \ell_i\cap X_s$ to $c_i\simeq \P^1$ (a rational map $X_s \dashrightarrow \P^1$).

The blowup $\Bl_{\ell_i}(\P^4)\rightarrow \P^4$ restricts to $\Bl_{\ell_i\cap X_s}(X_s)\rightarrow X_s$ as above, and we get a morphism $\Bl_{\ell_i\cap X_s}(X_s)\rightarrow \P^1$ from the above rational map. Once again, the fibers above the points in $\ell_i\cap X_s$ in the morphism $\Bl_{\ell_i\cap X_s}(X_s) \rightarrow X_s$ only contain $(-2)$-curves, and we get a morphism $X\rightarrow \Bl_{\ell_i\cap X_s}(X_s)$ from the minimal desingularization, from which we get a morphism $\varphi_{i1} : X\rightarrow \P^1$.

Summing up, we have constructed $2a+2b+c$ morphisms from $X$ to $\P^1$, where $a$ is the number of simple roots of $P$, $b$ the number of multiple roots corresponding to a rank $4$ quadric in the pencil, and $c$ the number of rank $3$ quadrics in the pencil. 

The following table contains the graphs we mentioned at the beginning of the section. It is sufficient to prove the next proposition for each geometric type.

Note that for each type we use an empty square to denote the classes in $\CC$ that intersect negatively a $(-2)$-curve, and a full square for the other ones; we also denote by $r_i$ the curve with class $E_i-E_{i+1}$, and by $r_{ijk}$ the curve with class $E_0-E_i-E_j-E_k$

\begin{center}
\label{Tablegraphs4}

\begin{longtable}{|l|c|c|c|l|}
\caption{Conics-$(-2)$-curves graphs}\\
\hline
\text{Geom.~Type} & \text{Conics}-$(-2)$\text{ Curves~Graph} & $N$ & $(a,b,c)$ & $(-2)$-\text{curves} \\
\hline
\hline
\endhead
\hline
\endfoot
$\A_1$ $[2111]$ & 
\begin{tikzpicture}[baseline=0]
\coordinate (C1) at (0,0.5);
\coordinate (C2) at (1,0.5);
\coordinate (C3) at (2,0.5);
\coordinate (C4) at (3,0.5);
\coordinate (C5) at (4,0.5);
\coordinate (C1p) at (0,-0.5);
\coordinate (C2p) at (1,-0.5);
\coordinate (C3p) at (2,-0.5);
\coordinate (C4p) at (3,-0.5);
\coordinate (C5p) at (4,-0.5);
\coordinate (R) at ($(C4)!.5!(C5p)$);
\draw[dashed] (C1p) -- (C1);
\draw[dashed] (C2p) -- (C2);
\draw[dashed] (C3p) -- (C3);
\draw[dashed] (C4p) -- (C4);
\draw[dashed] (C5p) -- (C5);
\draw (C4) -- (R) -- (C5p);
\draw[fill] (C1) node[above] {\tiny$C_1$} +(-2pt,-2pt) rectangle +(2pt,2pt) ;
\draw[fill] (C2) node[above] {\tiny$C_2$} +(-2pt,-2pt) rectangle +(2pt,2pt) ;
\draw[fill] (C3) node[above] {\tiny$C_3$} +(-2pt,-2pt) rectangle +(2pt,2pt) ;
\draw[fill] (C4) node[above] {\tiny$C_4$} +(-2pt,-2pt) rectangle +(2pt,2pt) ;
\draw (C5) node[above] {\tiny$C_5$} +(-2pt,-2pt) rectangle +(2pt,2pt) ;
\draw[fill] (C1p) node[below] {\tiny$C'_1$} +(-2pt,-2pt) rectangle +(2pt,2pt) ;
\draw[fill] (C2p) node[below] {\tiny$C'_2$} +(-2pt,-2pt) rectangle +(2pt,2pt) ;
\draw[fill] (C3p) node[below] {\tiny$C'_3$} +(-2pt,-2pt) rectangle +(2pt,2pt) ;
\draw (C4p) node[below] {\tiny$C_4'$} +(-2pt,-2pt) rectangle +(2pt,2pt) ;
\draw[fill] (C5p) node[below] {\tiny$C_5'$} +(-2pt,-2pt) rectangle +(2pt,2pt) ;
\draw[fill=white] (R) node[above] {\tiny$R$} circle (2pt) ;
\end{tikzpicture}

 & $8$ & $(3,1,0)$ & $R = r_{45}$ \\
\hline
$2\A_1(9)$ $[221]$ &  

\begin{tikzpicture}[baseline=0]
\coordinate (C1) at (0,0.5);
\coordinate (C2) at (1,0.5);
\coordinate (C3) at (2,0.5);
\coordinate (C4) at (3,0.5);
\coordinate (C5) at (4,0.5);
\coordinate (C1p) at (0,-0.5);
\coordinate (C2p) at (1,-0.5);
\coordinate (C3p) at (2,-0.5);
\coordinate (C4p) at (3,-0.5);
\coordinate (C5p) at (4,-0.5);
\coordinate (R1) at ($(C2)!.5!(C3p)$);
\coordinate (R2) at ($(C4)!.5!(C5p)$);
\draw[dashed] (C1p) -- (C1);
\draw[dashed] (C2p) -- (C2);
\draw[dashed] (C3p) -- (C3);
\draw[dashed] (C4p) -- (C4);
\draw[dashed] (C5p) -- (C5);
\draw (C2) -- (R1) -- (C3p);
\draw (C4) -- (R2) -- (C5p);
\draw[fill] (C1) node[above] {\tiny$C_1$} +(-2pt,-2pt) rectangle +(2pt,2pt) ;
\draw[fill] (C2) node[above] {\tiny$C_2$} +(-2pt,-2pt) rectangle +(2pt,2pt) ;
\draw (C3) node[above] {\tiny$C_3$} +(-2pt,-2pt) rectangle +(2pt,2pt) ;
\draw[fill] (C4) node[above] {\tiny$C_4$} +(-2pt,-2pt) rectangle +(2pt,2pt) ;
\draw (C5) node[above] {\tiny$C_5$} +(-2pt,-2pt) rectangle +(2pt,2pt) ;
\draw[fill] (C1p) node[below] {\tiny$C'_1$} +(-2pt,-2pt) rectangle +(2pt,2pt)  ;
\draw (C2p) node[below] {\tiny$C_2'$} +(-2pt,-2pt) rectangle +(2pt,2pt) ;
\draw[fill] (C3p) node[below] {\tiny$C'_3$} +(-2pt,-2pt) rectangle +(2pt,2pt)  ;
\draw (C4p) node[below] {\tiny$C_4'$} +(-2pt,-2pt) rectangle +(2pt,2pt) ;
\draw[fill] (C5p) node[below] {\tiny$C'_5$} +(-2pt,-2pt) rectangle +(2pt,2pt)  ;
\draw[fill=white] (R1) circle (2pt) node[below] {\tiny$R_1$} ;
\draw[fill=white] (R2) circle (2pt) node[below] {\tiny$R_2$} ;
\end{tikzpicture}

& $6$ & $(1,2,0)$ & $\begin{array}{l} R_1=r_{23}\\ R_2=r_{45} \\ \end{array}$ \\

\hline
$2\A_1(8)$ $[111(11)]$ &

\begin{tikzpicture}[baseline=0]
\coordinate (C1) at (0,0.5);
\coordinate (C2) at (1,0.5);
\coordinate (C3) at (2,0.5);
\coordinate (C4) at (3,0.5);
\coordinate (C5) at (4,0.5);
\coordinate (C1p) at (0,-0.5);
\coordinate (C2p) at (1,-0.5);
\coordinate (C3p) at (2,-0.5);
\coordinate (C4p) at (3,-0.5);
\coordinate (C5p) at (4,-0.5);
\coordinate (R1) at ($(C4)!.5!(C5)$);
\coordinate (R2) at ($(C4)!.5!(C5p)$);
\draw[dashed] (C1p) -- (C1);
\draw[dashed][dashed] (C2p) -- (C2);
\draw[dashed] (C3p) -- (C3);
\draw[dashed] (C4p) -- (C4);
\draw[dashed] (C5p) -- (C5);
\draw (C4) -- (R1) -- (C5);
\draw (C4) -- (R2) -- (C5p);
\draw[fill] (C1) node[above] {\tiny$C_1$} +(-2pt,-2pt) rectangle +(2pt,2pt) ;
\draw[fill] (C2) node[above] {\tiny$C_2$} +(-2pt,-2pt) rectangle +(2pt,2pt) ;
\draw[fill] (C3) node[above] {\tiny$C_3$} +(-2pt,-2pt) rectangle +(2pt,2pt) ;
\draw[fill] (C4) node[above] {\tiny$C_4$} +(-2pt,-2pt) rectangle +(2pt,2pt) ;
\draw (C5) node[above] {\tiny$C_5$} +(-2pt,-2pt) rectangle +(2pt,2pt) ;
\draw[fill] (C1p) node[below] {\tiny$C'_1$} +(-2pt,-2pt) rectangle +(2pt,2pt)  ;
\draw[fill] (C2p) node[below] {\tiny$C'_2$} +(-2pt,-2pt) rectangle +(2pt,2pt)  ;
\draw[fill] (C3p) node[below] {\tiny$C'_3$} +(-2pt,-2pt) rectangle +(2pt,2pt)  ;
\draw (C4p) node[below] {\tiny$C_4'$} +(-2pt,-2pt) rectangle +(2pt,2pt) ;
\draw (C5p) node[below] {\tiny$C_5'$} +(-2pt,-2pt) rectangle +(2pt,2pt) ;
\draw[fill=white] (R1) circle (2pt) node[above] {\tiny$R_1$} ;
\draw[fill=white] (R2) circle (2pt) node[below] {\tiny$R_2$} ;
\end{tikzpicture}

& $7$ & $(3,0,1)$ & $\begin{array}{l} R_1=r_{123}\\ R_2=r_{45} \\ \end{array}$ \\

\hline
$\A_2$ $[311]$ & 

\begin{tikzpicture}[baseline=0]
\coordinate (C1) at (0,0.5);
\coordinate (C2) at (1,0.5);
\coordinate (C3) at (2,0.5);
\coordinate (C4) at (3,0.5);
\coordinate (C5) at (4,0.5);
\coordinate (C1p) at (0,-0.5);
\coordinate (C2p) at (1,-0.5);
\coordinate (C3p) at (2,-0.5);
\coordinate (C4p) at (3,-0.5);
\coordinate (C5p) at (4,-0.5);
\coordinate (R1) at ($(C3)!.5!(C4p)$);
\coordinate (R2) at ($(C4)!.5!(C5p)$);
\draw[dashed] (C1p) -- (C1);
\draw[dashed] (C2p) -- (C2);
\draw[dashed] (C3p) -- (C3);
\draw[dashed] (C4p) -- (C4);
\draw[dashed] (C5p) -- (C5);
\draw (C3) -- (R1) -- (C4p); 
\draw (C4) -- (R2) -- (C5p);
\draw[fill] (C1) node[above] {\tiny$C_1$} +(-2pt,-2pt) rectangle +(2pt,2pt) ;
\draw[fill] (C2) node[above] {\tiny$C_2$} +(-2pt,-2pt) rectangle +(2pt,2pt) ;
\draw[fill] (C3) node[above] {\tiny$C_3$} +(-2pt,-2pt) rectangle +(2pt,2pt) ;
\draw (C4) node[above] {\tiny$C_4$} +(-2pt,-2pt) rectangle +(2pt,2pt) ;
\draw (C5) node[above] {\tiny$C_5$} +(-2pt,-2pt) rectangle +(2pt,2pt) ;
\draw[fill] (C1p) node[below] {\tiny$C'_1$} +(-2pt,-2pt) rectangle +(2pt,2pt)  ;
\draw[fill] (C2p) node[below] {\tiny$C'_2$} +(-2pt,-2pt) rectangle +(2pt,2pt)  ;
\draw (C3p) node[below] {\tiny$C_3'$} +(-2pt,-2pt) rectangle +(2pt,2pt) ;
\draw (C4p) node[below] {\tiny$C_4'$} +(-2pt,-2pt) rectangle +(2pt,2pt) ;
\draw[fill] (C5p) node[below] {\tiny$C_5'$} +(-2pt,-2pt) rectangle +(2pt,2pt) ;
\draw[fill=white] (R1) circle (2pt) node[below] {\tiny$R_1$} ;
\draw[fill=white] (R2) circle (2pt) node[below] {\tiny$R_2$} ;
\end{tikzpicture}

& $6$ & $(2,1,0)$ & $\begin{array}{l} R_1=r_{34}\\ R_2=r_{45} \\ \end{array}$ \\

\hline
$3\A_1$ $[(11)21]$ & 

\begin{tikzpicture}[baseline=0]
\coordinate (C1) at (0,0.5);
\coordinate (C2) at (1,0.5);
\coordinate (C3) at (2,0.5);
\coordinate (C4) at (3,0.5);
\coordinate (C5) at (4,0.5);
\coordinate (C1p) at (0,-0.5);
\coordinate (C2p) at (1,-0.5);
\coordinate (C3p) at (2,-0.5);
\coordinate (C4p) at (3,-0.5);
\coordinate (C5p) at (4,-0.5);
\coordinate (R1) at ($(C2)!.5!(C3p)$);
\coordinate (R2) at ($(C4)!.5!(C5p)$);
\coordinate (R3) at ($(C4)!.5!(C5)$);
\draw[dashed] (C1p) -- (C1);
\draw[dashed] (C2p) -- (C2);
\draw[dashed] (C3p) -- (C3);
\draw[dashed] (C4p) -- (C4);
\draw[dashed] (C5p) -- (C5);
\draw (C2) -- (R1) -- (C3p);
\draw (C4) -- (R3) -- (C5);
\draw (C4) -- (R2) -- (C5p);
\draw[fill] (C1) node[above] {\tiny$C_1$} +(-2pt,-2pt) rectangle +(2pt,2pt) ;
\draw[fill] (C2) node[above] {\tiny$C_2$} +(-2pt,-2pt) rectangle +(2pt,2pt) ;
\draw (C3) node[above] {\tiny$C_3$} +(-2pt,-2pt) rectangle +(2pt,2pt) ;
\draw[fill] (C4) node[above] {\tiny$C_4$} +(-2pt,-2pt) rectangle +(2pt,2pt) ;
\draw (C5) node[above] {\tiny$C_5$} +(-2pt,-2pt) rectangle +(2pt,2pt) ;
\draw[fill] (C1p) node[below] {\tiny$C'_1$} +(-2pt,-2pt) rectangle +(2pt,2pt)  ;
\draw (C2p) node[below] {\tiny$C_2'$} +(-2pt,-2pt) rectangle +(2pt,2pt) ;
\draw[fill] (C3p) node[below] {\tiny$C_3'$} +(-2pt,-2pt) rectangle +(2pt,2pt) ;
\draw (C4p) node[below] {\tiny$C_4'$} +(-2pt,-2pt) rectangle +(2pt,2pt) ;
\draw (C5p) node[below] {\tiny$C_5'$} +(-2pt,-2pt) rectangle +(2pt,2pt) ;
\draw[fill=white] (R1) circle (2pt) node[below] {\tiny$R_1$} ;
\draw[fill=white] (R2) circle (2pt) node[below] {\tiny$R_2$} ;
\draw[fill=white] (R3) circle (2pt) node[above] {\tiny$R_3$} ;
\end{tikzpicture}

& $5$ & $(1,1,1)$ & $\begin{array}{l} R_1=r_{23}\\ R_2=r_{45}\\ R_3=r_{123} \\ \end{array}$ \\

\hline

$\A_1\A_2$ $[32]$ & 

\begin{tikzpicture}[baseline=0]
\coordinate (C1) at (0,0.5);
\coordinate (C2) at (1,0.5);
\coordinate (C3) at (2,0.5);
\coordinate (C4) at (3,0.5);
\coordinate (C5) at (4,0.5);
\coordinate (C1p) at (0,-0.5);
\coordinate (C2p) at (1,-0.5);
\coordinate (C3p) at (2,-0.5);
\coordinate (C4p) at (3,-0.5);
\coordinate (C5p) at (4,-0.5);
\coordinate (R1) at ($(C1)!.5!(C2p)$);
\coordinate (R2) at ($(C3)!.5!(C4p)$);
\coordinate (R3) at ($(C4)!.5!(C5p)$);
\draw[dashed] (C1p) -- (C1);
\draw[dashed] (C2p) -- (C2);
\draw[dashed] (C3p) -- (C3);
\draw[dashed] (C4p) -- (C4);
\draw[dashed] (C5p) -- (C5);
\draw (C1) -- (R1) -- (C2p);
\draw (C3) -- (R2) -- (C4p);
\draw (C4) -- (R3) -- (C5p);
\draw[fill] (C1) node[above] {\tiny$C_1$} +(-2pt,-2pt) rectangle +(2pt,2pt) ;
\draw (C2) node[above] {\tiny$C_2$} +(-2pt,-2pt) rectangle +(2pt,2pt) ;
\draw[fill] (C3) node[above] {\tiny$C_3$} +(-2pt,-2pt) rectangle +(2pt,2pt) ;
\draw (C4) node[above] {\tiny$C_4$} +(-2pt,-2pt) rectangle +(2pt,2pt) ;
\draw (C5) node[above] {\tiny$C_5$} +(-2pt,-2pt) rectangle +(2pt,2pt) ;
\draw (C1p) node[below] {\tiny$C'_1$} +(-2pt,-2pt) rectangle +(2pt,2pt)  ;
\draw[fill] (C2p) node[below] {\tiny$C'_2$} +(-2pt,-2pt) rectangle +(2pt,2pt)  ;
\draw (C3p) node[below] {\tiny$C'_3$} +(-2pt,-2pt) rectangle +(2pt,2pt)  ;
\draw (C4p) node[below] {\tiny$C'_4$} +(-2pt,-2pt) rectangle +(2pt,2pt)  ;
\draw[fill] (C5p) node[below] {\tiny$C'_5$} +(-2pt,-2pt) rectangle +(2pt,2pt)  ;
\draw[fill=white] (R1) circle (2pt) node[below] {\tiny$R_1$} ;
\draw[fill=white] (R2) circle (2pt) node[below] {\tiny$R_2$} ;
\draw[fill=white] (R3) circle (2pt) node[below] {\tiny$R_3$} ;
\end{tikzpicture}

& $4$ & $(0,2,0)$ & $\begin{array}{l} R_1=r_{12}\\ R_2=r_{34} \\ R_3=r_{45} \\ \end{array}$ \\

\hline

$\A_3(5)$ $[41]$ & 

\begin{tikzpicture}[baseline=0]
\coordinate (C1) at (0,0.5);
\coordinate (C2) at (1,0.5);
\coordinate (C3) at (2,0.5);
\coordinate (C4) at (3,0.5);
\coordinate (C5) at (4,0.5);
\coordinate (C1p) at (0,-0.5);
\coordinate (C2p) at (1,-0.5);
\coordinate (C3p) at (2,-0.5);
\coordinate (C4p) at (3,-0.5);
\coordinate (C5p) at (4,-0.5);
\coordinate (R1) at ($(C2)!.5!(C3p)$);
\coordinate (R2) at ($(C3)!.5!(C4p)$);
\coordinate (R3) at ($(C4)!.5!(C5p)$);
\draw[dashed] (C1p) -- (C1);
\draw[dashed] (C2p) -- (C2);
\draw[dashed] (C3p) -- (C3);
\draw[dashed] (C4p) -- (C4);
\draw[dashed] (C5p) -- (C5);

\draw (C2) -- (R1) -- (C3p);
\draw (C3) -- (R2) -- (C4p);
\draw (C4) -- (R3) -- (C5p);
\draw[fill] (C1) node[above] {\tiny$C_1$} +(-2pt,-2pt) rectangle +(2pt,2pt) ;
\draw[fill] (C2) node[above] {\tiny$C_2$} +(-2pt,-2pt) rectangle +(2pt,2pt) ;
\draw (C3) node[above] {\tiny$C_3$} +(-2pt,-2pt) rectangle +(2pt,2pt) ;
\draw (C4) node[above] {\tiny$C_4$} +(-2pt,-2pt) rectangle +(2pt,2pt) ;
\draw (C5) node[above] {\tiny$C_5$} +(-2pt,-2pt) rectangle +(2pt,2pt) ;
\draw (C1p) node[below] {\tiny$C'_1$} +(-2pt,-2pt) rectangle +(2pt,2pt)  ;
\draw (C2p) node[below] {\tiny$C'_2$} +(-2pt,-2pt) rectangle +(2pt,2pt)  ;
\draw (C3p) node[below] {\tiny$C'_3$} +(-2pt,-2pt) rectangle +(2pt,2pt)  ;
\draw (C4p) node[below] {\tiny$C'_4$} +(-2pt,-2pt) rectangle +(2pt,2pt)  ;
\draw[fill] (C5p) node[below] {\tiny$C'_5$} +(-2pt,-2pt) rectangle +(2pt,2pt)  ;
\draw[fill=white] (R1) circle (2pt) node[below] {\tiny$R_1$} ;
\draw[fill=white] (R2) circle (2pt) node[below] {\tiny$R_2$} ;
\draw[fill=white] (R3) circle (2pt) node[below] {\tiny$R_3$} ;
\end{tikzpicture}

& $4$ & $(1,1,0)$ & $\begin{array}{l} R_1=r_{23}\\ R_2=r_{34} \\ R_3=r_{45} \\ \end{array}$ \\

\hline

$\A_3(4)$ $[(21)11]$ &

\begin{tikzpicture}[baseline=0]
\coordinate (C1) at (0,0.5);
\coordinate (C2) at (1,0.5);
\coordinate (C3) at (2,0.5);
\coordinate (C4) at (3,0.5);
\coordinate (C5) at (4,0.5);
\coordinate (C1p) at (0,-0.5);
\coordinate (C2p) at (1,-0.5);
\coordinate (C3p) at (2,-0.5);
\coordinate (C4p) at (3,-0.5);
\coordinate (C5p) at (4,-0.5);
\coordinate (R1) at ($(C4)!.5!(C5)$);
\coordinate (R2) at ($(C3)!.5!(C4p)$);
\coordinate (R3) at ($(C4)!.5!(C5p)$);
\draw[dashed] (C1p) -- (C1);
\draw[dashed] (C2p) -- (C2);
\draw[dashed] (C3p) -- (C3);
\draw[dashed] (C4p) -- (C4);
\draw[dashed] (C5p) -- (C5);
\draw (C4) -- (C5);
\draw (C3) -- (R2) -- (C4p);
\draw (C4) -- (R3) -- (C5p);
\draw[fill] (C1) node[above] {\tiny$C_1$} +(-2pt,-2pt) rectangle +(2pt,2pt) ;
\draw[fill] (C2) node[above] {\tiny$C_2$} +(-2pt,-2pt) rectangle +(2pt,2pt) ;
\draw[fill] (C3) node[above] {\tiny$C_3$} +(-2pt,-2pt) rectangle +(2pt,2pt) ;
\draw (C4) node[above] {\tiny$C_4$} +(-2pt,-2pt) rectangle +(2pt,2pt) ;
\draw (C5) node[above] {\tiny$C_5$} +(-2pt,-2pt) rectangle +(2pt,2pt) ;
\draw[fill] (C1p) node[below] {\tiny$C'_1$} +(-2pt,-2pt) rectangle +(2pt,2pt)  ;
\draw[fill] (C2p) node[below] {\tiny$C'_2$} +(-2pt,-2pt) rectangle +(2pt,2pt)  ;
\draw (C3p) node[below] {\tiny$C'_3$} +(-2pt,-2pt) rectangle +(2pt,2pt)  ;
\draw (C4p) node[below] {\tiny$C'_4$} +(-2pt,-2pt) rectangle +(2pt,2pt)  ;
\draw (C5p) node[below] {\tiny$C'_5$} +(-2pt,-2pt) rectangle +(2pt,2pt)  ;
\draw[fill=white] (R1) circle (2pt) node[above] {\tiny$R_3$} ;
\draw[fill=white] (R2) circle (2pt) node[below] {\tiny$R_2$} ;
\draw[fill=white] (R3) circle (2pt) node[below] {\tiny$R_1$} ;
\end{tikzpicture}

& $5$ & $(2,0,1)$ & $\begin{array}{l} R_1=r_{45}\\ R_2=r_{34} \\ R_3=r_{123} \\ \end{array}$ \\

\hline

$4\A_1$ $[(11)(11)1]$ & 

\begin{tikzpicture}[baseline=0]
\coordinate (C1) at (0,0.5);
\coordinate (C2) at (1,0.5);
\coordinate (C3) at (2,0.5);
\coordinate (C4) at (3,0.5);
\coordinate (C5) at (4,0.5);
\coordinate (C1p) at (0,-0.5);
\coordinate (C2p) at (1,-0.5);
\coordinate (C3p) at (2,-0.5);
\coordinate (C4p) at (3,-0.5);
\coordinate (C5p) at (4,-0.5);
\coordinate (R1) at ($(C4)!.5!(C5)$);
\coordinate (R2) at ($(C1)!.5!(C2p)$);
\coordinate (R3) at ($(C1)!.5!(C2)$);
\coordinate (R4) at ($(C4)!.5!(C5p)$);
\draw[dashed] (C1p) -- (C1);
\draw[dashed] (C2p) -- (C2);
\draw[dashed] (C3p) -- (C3);
\draw[dashed] (C4p) -- (C4);
\draw[dashed] (C5p) -- (C5);
\draw (C1) -- (R2) -- (C2p);
\draw (C1) -- (C2);
\draw (C4) -- (R4) -- (C5p);
\draw (C4) -- (C5);
\draw[fill] (C1) node[above] {\tiny$C_1$} +(-2pt,-2pt) rectangle +(2pt,2pt) ;
\draw (C2) node[above] {\tiny$C_2$} +(-2pt,-2pt) rectangle +(2pt,2pt) ;
\draw[fill] (C3) node[above] {\tiny$C_3$} +(-2pt,-2pt) rectangle +(2pt,2pt) ;
\draw[fill] (C4) node[above] {\tiny$C_4$} +(-2pt,-2pt) rectangle +(2pt,2pt) ;
\draw (C5) node[above] {\tiny$C_5$} +(-2pt,-2pt) rectangle +(2pt,2pt) ;
\draw (C1p) node[below] {\tiny$C'_1$} +(-2pt,-2pt) rectangle +(2pt,2pt)  ;
\draw (C2p) node[below] {\tiny$C'_2$} +(-2pt,-2pt) rectangle +(2pt,2pt)  ;
\draw[fill] (C3p) node[below] {\tiny$C'_3$} +(-2pt,-2pt) rectangle +(2pt,2pt)  ;
\draw (C4p) node[below] {\tiny$C'_4$} +(-2pt,-2pt) rectangle +(2pt,2pt)  ;
\draw (C5p) node[below] {\tiny$C'_5$} +(-2pt,-2pt) rectangle +(2pt,2pt)  ;
\draw[fill=white] (R1) circle (2pt) node[above] {\tiny$R_4$} ;
\draw[fill=white] (R2) circle (2pt) node[below] {\tiny$R_1$} ;
\draw[fill=white] (R3) circle (2pt) node[above] {\tiny$R_2$} ;
\draw[fill=white] (R4) circle (2pt) node[below] {\tiny$R_3$} ;
\end{tikzpicture}

& $4$ & $(1,0,2)$ & $\begin{array}{l} R_1=r_{12}\\ R_2=r_{345} \\ R_3=r_{45} \\ R_4=r_{123} \\ \end{array}$ \\

\hline

$2\A_1\A_2$ $[3(11)]$ &

\begin{tikzpicture}[baseline=0]
\coordinate (C1) at (0,0.5);
\coordinate (C2) at (1,0.5);
\coordinate (C3) at (2,0.5);
\coordinate (C4) at (3,0.5);
\coordinate (C5) at (4,0.5);
\coordinate (C1p) at (0,-0.5);
\coordinate (C2p) at (1,-0.5);
\coordinate (C3p) at (2,-0.5);
\coordinate (C4p) at (3,-0.5);
\coordinate (C5p) at (4,-0.5);
\coordinate (R1) at ($(C1)!.5!(C2p)$);
\coordinate (R2) at ($(C2)!.5!(C3p)$);
\coordinate (R3) at ($(C4)!.5!(C5)$);
\coordinate (R4) at ($(C4)!.5!(C5p)$);
\draw[dashed] (C1p) -- (C1);
\draw[dashed] (C2p) -- (C2);
\draw[dashed] (C3p) -- (C3);
\draw[dashed] (C4p) -- (C4);
\draw[dashed] (C5p) -- (C5);
\draw (C1) -- (R1) -- (C2p);
\draw (C2) -- (R2) -- (C3p);
\draw (C4) -- (R4) -- (C5p);
\draw (C4) -- (C5);
\draw[fill] (C1) node[above] {\tiny$C_1$} +(-2pt,-2pt) rectangle +(2pt,2pt) ;
\draw (C2) node[above] {\tiny$C_2$} +(-2pt,-2pt) rectangle +(2pt,2pt) ;
\draw (C3) node[above] {\tiny$C_3$} +(-2pt,-2pt) rectangle +(2pt,2pt) ;
\draw[fill] (C4) node[above] {\tiny$C_4$} +(-2pt,-2pt) rectangle +(2pt,2pt) ;
\draw (C5) node[above] {\tiny$C_5$} +(-2pt,-2pt) rectangle +(2pt,2pt) ;
\draw (C1p) node[below] {\tiny$C'_1$} +(-2pt,-2pt) rectangle +(2pt,2pt)  ;
\draw (C2p) node[below] {\tiny$C'_2$} +(-2pt,-2pt) rectangle +(2pt,2pt)  ;
\draw[fill] (C3p) node[below] {\tiny$C'_3$} +(-2pt,-2pt) rectangle +(2pt,2pt)  ;
\draw (C4p) node[below] {\tiny$C'_4$} +(-2pt,-2pt) rectangle +(2pt,2pt)  ;
\draw (C5p) node[below] {\tiny$C'_5$} +(-2pt,-2pt) rectangle +(2pt,2pt)  ;
\draw[fill=white] (R1) circle (2pt) node[below] {\tiny$R_1$} ;
\draw[fill=white] (R2) circle (2pt) node[below] {\tiny$R_2$} ;
\draw[fill=white] (R3) circle (2pt) node[above] {\tiny$R_4$} ;
\draw[fill=white] (R4) circle (2pt) node[below] {\tiny$R_3$} ;
\end{tikzpicture}

& $3$ & $(0,1,1)$ & $\begin{array}{l} R_1=r_{12}\\ R_2=r_{23} \\ R_3=r_{45} \\ R_4=r_{123} \\ \end{array}$ \\

\hline

$\A_1\A_3$ $[(21)2]$ &

\begin{tikzpicture}[baseline=0]
\coordinate (C1) at (0,0.5);
\coordinate (C2) at (1,0.5);
\coordinate (C3) at (2,0.5);
\coordinate (C4) at (3,0.5);
\coordinate (C5) at (4,0.5);
\coordinate (C1p) at (0,-0.5);
\coordinate (C2p) at (1,-0.5);
\coordinate (C3p) at (2,-0.5);
\coordinate (C4p) at (3,-0.5);
\coordinate (C5p) at (4,-0.5);
\coordinate (R1) at ($(C1)!.5!(C2p)$);
\coordinate (R2) at ($(C3)!.5!(C4p)$);
\coordinate (R3) at ($(C4)!.5!(C5)$);
\coordinate (R4) at ($(C4)!.5!(C5p)$);
\draw[dashed] (C1p) -- (C1);
\draw[dashed] (C2p) -- (C2);
\draw[dashed] (C3p) -- (C3);
\draw[dashed] (C4p) -- (C4);
\draw[dashed] (C5p) -- (C5);
\draw (C1) -- (R1) -- (C2p);
\draw  (C3) -- (R2) -- (C4p);
\draw (C4) -- (R4) -- (C5p);
\draw (C4) -- (C5);
\draw[fill] (C1) node[above] {\tiny$C_1$} +(-2pt,-2pt) rectangle +(2pt,2pt) ;
\draw (C2) node[above] {\tiny$C_2$} +(-2pt,-2pt) rectangle +(2pt,2pt) ;
\draw[fill] (C3) node[above] {\tiny$C_3$} +(-2pt,-2pt) rectangle +(2pt,2pt) ;
\draw (C4) node[above] {\tiny$C_4$} +(-2pt,-2pt) rectangle +(2pt,2pt) ;
\draw (C5) node[above] {\tiny$C_5$} +(-2pt,-2pt) rectangle +(2pt,2pt) ;
\draw (C1p) node[below] {\tiny$C'_1$} +(-2pt,-2pt) rectangle +(2pt,2pt)  ;
\draw[fill] (C2p) node[below] {\tiny$C'_2$} +(-2pt,-2pt) rectangle +(2pt,2pt)  ;
\draw (C3p) node[below] {\tiny$C'_3$} +(-2pt,-2pt) rectangle +(2pt,2pt)  ;
\draw (C4p) node[below] {\tiny$C'_4$} +(-2pt,-2pt) rectangle +(2pt,2pt)  ;
\draw (C5p) node[below] {\tiny$C'_5$} +(-2pt,-2pt) rectangle +(2pt,2pt)  ;
\draw[fill=white] (R1) circle (2pt) node[below] {\tiny$R_1$} ;
\draw[fill=white] (R2) circle (2pt) node[below] {\tiny$R_2$} ;
\draw[fill=white] (R3) circle (2pt) node[above] {\tiny$R_4$} ;
\draw[fill=white] (R4) circle (2pt) node[below] {\tiny$R_3$} ;
\end{tikzpicture}

& $3$ & $(0,1,1)$ & $\begin{array}{l} R_1=r_{12}\\ R_2=r_{34} \\ R_3= r_{45} \\ R_4=r_{123} \\ \end{array}$ \\

\hline

$\A_4$ $[5]$ &
\begin{tikzpicture}[baseline=0]
\coordinate (C1) at (0,0.5);
\coordinate (C2) at (1,0.5);
\coordinate (C3) at (2,0.5);
\coordinate (C4) at (3,0.5);
\coordinate (C5) at (4,0.5);
\coordinate (C1p) at (0,-0.5);
\coordinate (C2p) at (1,-0.5);
\coordinate (C3p) at (2,-0.5);
\coordinate (C4p) at (3,-0.5);
\coordinate (C5p) at (4,-0.5);
\coordinate (R1) at ($(C1)!.5!(C2p)$);
\coordinate (R2) at ($(C2)!.5!(C3p)$);
\coordinate (R3) at ($(C3)!.5!(C4p)$);
\coordinate (R4) at ($(C4)!.5!(C5p)$);
\draw[dashed] (C1p) -- (C1);
\draw[dashed] (C2p) -- (C2);
\draw[dashed] (C3p) -- (C3);
\draw[dashed] (C4p) -- (C4);
\draw[dashed] (C5p) -- (C5);
\draw (C1) -- (R1) -- (C2p);
\draw (C2) -- (R2) -- (C3p);
\draw (C3) -- (R3) -- (C4p);
\draw (C4) -- (R4) -- (C5p);
\draw[fill] (C1) node[above] {\tiny$C_1$} +(-2pt,-2pt) rectangle +(2pt,2pt) ;
\draw (C2) node[above] {\tiny$C_2$} +(-2pt,-2pt) rectangle +(2pt,2pt) ;
\draw (C3) node[above] {\tiny$C_3$} +(-2pt,-2pt) rectangle +(2pt,2pt) ;
\draw (C4) node[above] {\tiny$C_4$} +(-2pt,-2pt) rectangle +(2pt,2pt) ;
\draw (C5) node[above] {\tiny$C_5$} +(-2pt,-2pt) rectangle +(2pt,2pt) ;
\draw (C1p) node[below] {\tiny$C'_1$} +(-2pt,-2pt) rectangle +(2pt,2pt)  ;
\draw (C2p) node[below] {\tiny$C'_2$} +(-2pt,-2pt) rectangle +(2pt,2pt)  ;
\draw (C3p) node[below] {\tiny$C'_3$} +(-2pt,-2pt) rectangle +(2pt,2pt)  ;
\draw (C4p) node[below] {\tiny$C'_4$} +(-2pt,-2pt) rectangle +(2pt,2pt)  ;
\draw[fill] (C5p) node[below] {\tiny$C'_5$} +(-2pt,-2pt) rectangle +(2pt,2pt)  ;
\draw[fill=white] (R1) circle (2pt) node[below] {\tiny$R_1$} ;
\draw[fill=white] (R2) circle (2pt) node[below] {\tiny$R_2$} ;
\draw[fill=white] (R3) circle (2pt) node[below] {\tiny$R_3$} ;
\draw[fill=white] (R4) circle (2pt) node[below] {\tiny$R_4$} ;
\end{tikzpicture}

& $2$ & $(0,1,0)$ & $\begin{array}{l} R_1=r_{12}\\ R_2=r_{23} \\ R_3=r_{34} \\ R_4=r_{45} \\ \end{array}$ \\

\hline

$\D_4$ $[(31)1]$ &

\begin{tikzpicture}[baseline=0]
\coordinate (C1) at (0,0.5);
\coordinate (C2) at (1,0.5);
\coordinate (C3) at (2,0.5);
\coordinate (C4) at (3,0.5);
\coordinate (C5) at (4,0.5);
\coordinate (C1p) at (0,-0.5);
\coordinate (C2p) at (1,-0.5);
\coordinate (C3p) at (2,-0.5);
\coordinate (C4p) at (3,-0.5);
\coordinate (C5p) at (4,-0.5);
\coordinate (R1) at ($(C4)!.5!(C5)$);
\coordinate (R2) at ($(C2)!.5!(C3p)$);
\coordinate (R3) at ($(C3)!.5!(C4p)$);
\coordinate (R4) at ($(C4)!.5!(C5p)$);
\draw[dashed] (C1p) -- (C1);
\draw[dashed] (C2p) -- (C2);
\draw[dashed] (C3p) -- (C3);
\draw[dashed] (C4p) -- (C4);
\draw[dashed] (C5p) -- (C5);
\draw (C2) -- (R2) -- (C3p);
\draw (C3) -- (R3) -- (C4p);
\draw (C4) -- (R4) -- (C5p);
\draw (C4) -- (C5);
\draw[fill] (C1) node[above] {\tiny$C_1$} +(-2pt,-2pt) rectangle +(2pt,2pt) ;
\draw[fill] (C2) node[above] {\tiny$C_2$} +(-2pt,-2pt) rectangle +(2pt,2pt) ;
\draw (C3) node[above] {\tiny$C_3$} +(-2pt,-2pt) rectangle +(2pt,2pt) ;
\draw (C4) node[above] {\tiny$C_4$} +(-2pt,-2pt) rectangle +(2pt,2pt) ;
\draw (C5) node[above] {\tiny$C_5$} +(-2pt,-2pt) rectangle +(2pt,2pt) ;
\draw[fill] (C1p) node[below] {\tiny$C'_1$} +(-2pt,-2pt) rectangle +(2pt,2pt)  ;
\draw (C2p) node[below] {\tiny$C'_2$} +(-2pt,-2pt) rectangle +(2pt,2pt)  ;
\draw (C3p) node[below] {\tiny$C'_3$} +(-2pt,-2pt) rectangle +(2pt,2pt)  ;
\draw (C4p) node[below] {\tiny$C'_4$} +(-2pt,-2pt) rectangle +(2pt,2pt)  ;
\draw (C5p) node[below] {\tiny$C'_5$} +(-2pt,-2pt) rectangle +(2pt,2pt)  ;
\draw[fill=white] (R1) circle (2pt) node[above] {\tiny$R_4$} ;
\draw[fill=white] (R2) circle (2pt) node[below] {\tiny$R_1$} ;
\draw[fill=white] (R3) circle (2pt) node[below] {\tiny$R_2$} ;
\draw[fill=white] (R4) circle (2pt) node[below] {\tiny$R_3$} ;
\end{tikzpicture}

& $3$ & $(1,0,1)$ & $\begin{array}{l} R_1=r_{23}\\ R_2=r_{34} \\ R_3= r_{45} \\ R_4=r_{123} \\ \end{array}$  \\

\hline

$2\A_1\A_3$ $[(21)(11)]$ &
\begin{tikzpicture}[baseline=0]
\coordinate (C1) at (0,0.5);
\coordinate (C2) at (1,0.5);
\coordinate (C3) at (2,0.5);
\coordinate (C4) at (3,0.5);
\coordinate (C5) at (4,0.5);
\coordinate (C1p) at (0,-0.5);
\coordinate (C2p) at (1,-0.5);
\coordinate (C3p) at (2,-0.5);
\coordinate (C4p) at (3,-0.5);
\coordinate (C5p) at (4,-0.5);
\coordinate (R1) at ($(C4)!.5!(C5)$);
\coordinate (R2) at ($(C1)!.5!(C2p)$);
\coordinate (R3) at ($(C3)!.5!(C4p)$);
\coordinate (R4) at ($(C4)!.5!(C5p)$);
\coordinate (R5) at ($(C1)!.5!(C2)$);
\draw[dashed] (C1p) -- (C1);
\draw[dashed] (C2p) -- (C2);
\draw[dashed] (C3p) -- (C3);
\draw[dashed] (C4p) -- (C4);
\draw[dashed] (C5p) -- (C5);
\draw (C1) -- (C2);
\draw (C1) -- (R2) -- (C2p);
\draw (C3) -- (R3) -- (C4p);
\draw (C4) -- (R4) -- (C5p);
\draw (C4) -- (C5);
\draw[fill] (C1) node[above] {\tiny$C_1$} +(-2pt,-2pt) rectangle +(2pt,2pt) ;
\draw (C2) node[above] {\tiny$C_2$} +(-2pt,-2pt) rectangle +(2pt,2pt) ;
\draw[fill] (C3) node[above] {\tiny$C_3$} +(-2pt,-2pt) rectangle +(2pt,2pt) ;
\draw (C4) node[above] {\tiny$C_4$} +(-2pt,-2pt) rectangle +(2pt,2pt) ;
\draw (C5) node[above] {\tiny$C_5$} +(-2pt,-2pt) rectangle +(2pt,2pt) ;
\draw (C1p) node[below] {\tiny$C'_1$} +(-2pt,-2pt) rectangle +(2pt,2pt)  ;
\draw (C2p) node[below] {\tiny$C'_2$} +(-2pt,-2pt) rectangle +(2pt,2pt)  ;
\draw (C3p) node[below] {\tiny$C'_3$} +(-2pt,-2pt) rectangle +(2pt,2pt)  ;
\draw (C4p) node[below] {\tiny$C'_4$} +(-2pt,-2pt) rectangle +(2pt,2pt)  ;
\draw (C5p) node[below] {\tiny$C'_5$} +(-2pt,-2pt) rectangle +(2pt,2pt)  ;
\draw[fill=white] (R1) circle (2pt) node[above] {\tiny$R_5$} ;
\draw[fill=white] (R2) circle (2pt) node[below] {\tiny$R_1$} ;
\draw[fill=white] (R3) circle (2pt) node[below] {\tiny$R_3$} ;
\draw[fill=white] (R4) circle (2pt) node[below] {\tiny$R_4$} ;
\draw[fill=white] (R5) circle (2pt) node[above] {\tiny$R_2$} ;
\end{tikzpicture}

& $2$ & $(0,0,2)$ & $\begin{array}{l} R_1=r_{12}\\ R_2=r_{345} \\ R_3=r_{34} \\ R_4=r_{45} \\ R_5=r_{123} \\ \end{array}$  \\

\hline

$\D_5$ $[(41)]$ &

\begin{tikzpicture}[baseline=0]
\coordinate (C1) at (0,0.5);
\coordinate (C2) at (1,0.5);
\coordinate (C3) at (2,0.5);
\coordinate (C4) at (3,0.5);
\coordinate (C5) at (4,0.5);
\coordinate (C1p) at (0,-0.5);
\coordinate (C2p) at (1,-0.5);
\coordinate (C3p) at (2,-0.5);
\coordinate (C4p) at (3,-0.5);
\coordinate (C5p) at (4,-0.5);
\coordinate (R1) at ($(C1)!.5!(C2p)$);
\coordinate (R2) at ($(C2)!.5!(C3p)$);
\coordinate (R3) at ($(C3)!.5!(C4p)$);
\coordinate (R4) at ($(C4)!.5!(C5p)$);
\coordinate (R5) at ($(C4)!.5!(C5)$);
\draw[dashed] (C1p) -- (C1);
\draw[dashed] (C2p) -- (C2);
\draw[dashed] (C3p) -- (C3);
\draw[dashed] (C4p) -- (C4);
\draw[dashed] (C5p) -- (C5);
\draw (C1) -- (R1) -- (C2p);
\draw (C2) -- (R2) -- (C3p);
\draw (C3) -- (R3) -- (C4p);
\draw (C4) -- (R4) -- (C5p);
\draw (C4) -- (C5);
\draw[fill] (C1) node[above] {\tiny$C_1$} +(-2pt,-2pt) rectangle +(2pt,2pt) ;
\draw (C2) node[above] {\tiny$C_2$} +(-2pt,-2pt) rectangle +(2pt,2pt) ;
\draw (C3) node[above] {\tiny$C_3$} +(-2pt,-2pt) rectangle +(2pt,2pt) ;
\draw (C4) node[above] {\tiny$C_4$} +(-2pt,-2pt) rectangle +(2pt,2pt) ;
\draw (C5) node[above] {\tiny$C_5$} +(-2pt,-2pt) rectangle +(2pt,2pt) ;
\draw (C1p) node[below] {\tiny$C'_1$} +(-2pt,-2pt) rectangle +(2pt,2pt)  ;
\draw (C2p) node[below] {\tiny$C'_2$} +(-2pt,-2pt) rectangle +(2pt,2pt)  ;
\draw (C3p) node[below] {\tiny$C'_3$} +(-2pt,-2pt) rectangle +(2pt,2pt)  ;
\draw (C4p) node[below] {\tiny$C'_4$} +(-2pt,-2pt) rectangle +(2pt,2pt)  ;
\draw (C5p) node[below] {\tiny$C'_5$} +(-2pt,-2pt) rectangle +(2pt,2pt)  ;
\draw[fill=white] (R1) circle (2pt) node[below] {\tiny$R_1$} ;
\draw[fill=white] (R2) circle (2pt) node[below] {\tiny$R_2$} ;
\draw[fill=white] (R3) circle (2pt) node[below] {\tiny$R_3$} ;
\draw[fill=white] (R4) circle (2pt) node[below] {\tiny$R_4$} ;
\draw[fill=white] (R5) circle (2pt) node[above] {\tiny$R_5$} ;
\end{tikzpicture}

& $1$ & $(0,0,1)$ & $\begin{array}{l} R_1=r_{12}\\ R_2=r_{23} \\ R_3=r_{34} \\ R_4=r_{45} \\ R_5=r_{123} \\ \end{array}$ \\

\end{longtable}
\end{center}

We summarize the results of this section in the following

\begin{proposition}
\label{classmorphip1}
Let $X$ denote a degree four weak del Pezzo surface, and $P$ a characteristic polynomial for the pencil of quadrics defining its anticanonical model in $\P^4$. Denote by $a_X$ the number of simple roots of $P$, $b_X$ (\emph{resp.} $c_X$) the number of multiple roots corresponding to a rank $4$ (\emph{resp.} rank $3$) quadric in the pencil.

If $N_X$ is the number of classes in $\CC$ not intersecting any $(-2)$-curve negatively, we have $N_X=2a_X+2b_X+c_X$, and this is the number of surjective morphisms from $X$ to the projective line. 

For each of these classes $C$, the morphism $\varphi_{|C|}:X\rightarrow \P^1$ has fibers linearly equivalent to $C$, and the type of map $\varphi_{|C|}$ defines on $X_s$ depends on the following numerical criterion 
\begin{itemize}
	\item[(a)] $2a_X$ of them satisfy $C\cdot R=0$ for all $(-2)$-curves $R\in \R_{irr}(X)$; for those ones $\varphi_{|C|}$ factors through the anticanonical morphism and yields a morphism from $X_s$ to $\P^1$; moreover these classes lie in $\Pic(X_s)$, and they form $a_X$ couples of complementary classes;
	\item[(b)] $2b_X+c_X$ satisfy $C\cdot R\geq 0$ for all $(-2)$-curves $R$ and $C\cdot R>0$ for at least one, and for those $\varphi_{|C|}$ does not define a morphism from $X_s$ to $\P^1$.
\end{itemize}
\end{proposition}

\begin{remark}
\label{morphismsordinary}
Note that in the ordinary case, we have $a_X=5$ and $b_X=c_X=0$. Then the fibers of the $10$ morphisms form the conic bundles, and we recover the original graph.
\end{remark}

\begin{remark}
\label{incidencefiberroots}
Another consequence of the discussion above is that, for a fixed class $C\in \CC$ which intersects all the $(-2)$-curves nonnegatively, we have two possibilities for the relative position of a $(-2)$-curve with class $R$ with the fibers of the morphism $\varphi_{|C|}$

\begin{itemize}
	\item[(i)] when we have $C\cdot R=1$, $R$ is transverse to the fibration, and $\varphi_{|C|}$ restricts to an isomorphism from $R$ to $\P^1$. The images in $X_s$ of its fibers all pass through the singular point which is the image of $R$;
	\item[(ii)] when we have $C\cdot R=0$, $R$ is contained in a fiber of the morphism $\varphi_{|C|}$.
\end{itemize}
\end{remark}

\subsection{Galois action on the singular quadrics, and arithmetic types}

In this section we work over the finite field $k=\F_q$, $q$ a power of an odd prime.

It remains to provide information on the arithmetic type of a degree four del Pezzo surface $X$ from arithmetic information on the rank four singular quadrics in the pencil defining $X_s$. We begin by using the results of the preceding section to describe the Galois action on the morphisms to the projective line from this information. Later, we will transpose this action to an action on the graph of conics and $(-2)$-curves, which will naturally lead to a (``part'' of a) conjugacy class in the Weyl group $W(\D_5)$, that we describe below, and that is given in the fourth column of the table for degree four surfaces in the Appendix. 

We begin by a definition (note that it is independant of the choice of the supplementary)
\begin{definition}
Let $Q$ denote a rank four quadric in $\P^4$, defined over $\F_{q^d}$; its \emph{restricted discriminant} is the discriminant of its base in $\F_{q^d}^\times/\F_{q^d}^{\times2}$. In other word, it is the discriminant of the restriction of quadratic form corresponding to $Q$ to a supplementary of its kernel.
\end{definition}

The Weyl group $W(\D_5)$ can be identified with the group $H_5 \rtimes \S_5$, where $H_5$ is the subgroup of $(\Z/2\Z)^5$ which is the kernel of the map $\varepsilon=(\varepsilon_1,\ldots,\varepsilon_5)\mapsto \sum_{i=1}^5 \varepsilon_i$, and $\S_5$ acts on $H_5$ by $\sigma\cdot\varepsilon=(\varepsilon_{\sigma^{-1}(1)},\ldots,\varepsilon_{\sigma^{-1}(5)})$.

Recall from \cite[Proposition 25]{carter} that the conjugacy classes in $W(\D_5)$ are even signed cycle-types; they form a partition of $5$ where each integer in the partition is overlined or not, and there is an even number of overlines. The conjugacy class of an element $(\varepsilon,\sigma)$ has partition corresponding to the conjugacy class of $\sigma$ in $\S_5$, and a number in this partition is overlined when the sum of the $\varepsilon_i$ where $i$ describes the support of the corresponding cycle is $1$. Note that if $\sigma$ is a $d$-cycle, the conjugacy class of an element of the form $(\varepsilon,\sigma)$ is $d$ when its $d$-th power is trivial, and $\overline{d}$ else.

Note that $W(\D_5)$ is a subgroup of $W(\B_5)=(\Z/2\Z)^5 \rtimes \S_5$, whose conjugacy classes can be represented by (even or odd) signed cycle-types.

Note also that an arithmetic type is a conjugacy class in some subgroup of $W(\D_5)$. As a consequence, different arithmetic types can be contained in the same conjugacy class in $W(\D_5)$, and we will have to be more precise in the last section.

With these notations (and the ones used in the preceding sections) at hand, we begin by describing the Galois action on the part of the graph coming from the morphisms associated to the simple roots of $P$

\begin{lemma} \label{simpleroots}
Let $X$ be a weak degree $4$ del Pezzo surface over $\F_q$ whose anti-canonical model is the intersection of two quadrics $Q_0$ and $Q_\infty$ in~$\P^4$, both defined over $\F_q$.
Let $G\in\F_q[T]$ be a simple irreducible divisor of degree~$d$ of the
characteristic polynomial~$P(T)$ and let~$\theta\in\F_{q^d}$ be a root of $G$. 

We denote by $\Delta(\theta)\in \F_{q^d}^\times$ the restricted discriminant of~$ Q=Q_0 - \theta Q_\infty$, ie the discriminant of its base $q$.

To $G$, one can associate $d$ couples of complementary classes in $\CC$, and the action of the Frobenius 
on those classes induces a permutation of signed sub-type~$d$ or~$\overline{d}$ depending on whether $\Delta(\theta)$ is or is not a square in~$\F_{q^d}^\times$ (note that this does not depend on the choice of the root $\theta$).
\end{lemma}

\begin{proof}
Let $\theta_1=\theta,\theta_2=\theta^q,\ldots, \theta_d=\theta^{q^{d-1}}\in \F_{q^d}$ denote the roots of $G$. From the results in the preceding section, we obtain $d$ rank $4$ quadrics $Q_i:=Q_0 - \theta_i Q_\infty$, and $2d$ morphisms $\varphi_{ij}$ from $X_s$ to $\P^1$, $1\leq i\leq d$, $1\leq j\leq 2$. It follows from Proposition \ref{classmorphip1} that the classes $f_{ij}$ of the fibers of these morphisms form $d$ couples of complementary classes in $\CC$.

The action of Frobenius sends $Q_i$ on $Q_{i+1}$ (the indices are read modulo $d$), and the couple $\{f_{i1},f_{i2}\}$ to the couple $\{f_{i+1,1},f_{i+1,2}\}$. Thus the permutation associated to the Frobenius (seen as an element in $W(\D_5)$) is the cycle $(1\ldots d)$ of length $d$.

The base $q_1$ of the quadric $Q_1$ is defined over $\F_{q^d}$; it is well known that it is either hyperbolic (isomorphic to $\P^1\times \P^1$ over $\F_{q^d}$) or elliptic (it becomes isomorphic to $\P^1\times \P^1$ only after a quadratic extension of the base field) depending on whether its discriminant is or is not a square in~$\F_{q^d}$.

In the hyperbolic case (then any of the $q_i$ is hyperbolic), the $d$-th power of the Frobenius stabilizes each one of the two $\P^1$, and the fibers $f_{11}$ and $f_{12}$; the Galois action has order $d$, and type $d$. Else the $d$-th power of the Frobenius exchanges the two $\P^1$ and the fibers $f_{11}$, $f_{12}$; the Galois action has order $2d$, and type $\overline{d}$.
\end{proof}

We now describe the Galois action on the morphisms associated to the rank $4$ quadrics in the pencil coming from multiple roots of its characteristic polynomial. 

\begin{lemma}
\label{multipleroots4}
Let $\theta\in \overline{\F}_q$ be a multiple root of $P$, say of multiplicity $m$, such that the associated quadric $Q:=Q_0-\theta Q_\infty$ has rank $4$. Denote by $G$ the minimal polynomial of $\theta$ over $\F_q$, by $d$ its degree, and by $\Delta(\theta)$ the restricted discriminant of $Q$.
  
The $d$ vertices of $Q$ and its conjugates $Q^\sigma,\ldots,Q^{\sigma^{d-1}}$ are singular points on $X_s$, and these singularities have Dynkin type $d\A_{m-1}$. To these correspond $2d$ morphisms from $X$ to $\P^1$, whose fibers $f_1,f_2,\ldots,f_1^{\sigma^{d-1}},f_2^{\sigma^{d-1}}$ form $2d$ classes in $\CC$, and the following cases occur
\begin{itemize}
	\item[(a)] $d=1$, $2\leq m \leq 4$: the Galois action fixes $f_1$ and $f_2$ if $\Delta(\theta)$ is a square in $\F_q^\times$, and exchanges them otherwise;
	\item[(b)] $d=2$, $m=2$ and the Galois action acts on the fibers as the bitransposition $(f_1f_1^{\sigma})(f_2f_2^{\sigma})$ if $\Delta(\theta)$ is a square in $\F_{q^2}^\times$, and as the four cycle $(f_1f_1^{\sigma}f_2f_2^{\sigma})$ else.
\end{itemize}
\end{lemma}

\begin{proof}
Since $P$ has degree $5$ and $G^m$ is a divisor, we must have $md\leq 5$, which leaves us with the listed possibilities, and $d=1$, $m=5$. But we see from the Segre symbols that this case corresponds to a singularity of Dynkin type $\A_4$. For this geometric type, there is only one arithmetic type, and the Galois action is not relevant: it will be sufficient to construct any del Pezzo surface over $\F_q$ with this singularity.

The assertion on the Dynkin type of the singularities comes from the Segre symbols: the number $m$ appearing as a part of it corresponds to an $\A_{m-1}$ singularity. The assertion on the fibers comes from Proposition \ref{classmorphip1}.

Finally the arithmetic assertions on the Galois action are easily deduced from the proof of the preceding lemma.
\end{proof}

\subsection{Quadratic module associated to a pencil of quadrics}

We continue to work over the finite field $k=\F_q$.

The anticanonical model of a degree four del Pezzo surface is the base locus of a pencil of quadrics in $\P^4$, thus it is defined by the vanishing of a pair of quadratic forms. We follow \cite{water} in this section, and define a quadratic module (here a $k[T]$-module of finite length endowed with a non degenerate bilinear form) from the surface. We give a normal form for such a module and we deduce arithmetic information on the singular quadrics of the pencil from this normal form.

As usual we denote by $\{Q_0,Q_\infty\}$ a basis for the pencil of quadratic forms defining $X$. We assume $Q_\infty$ is non-degenerate. Note that this requires the existence of a non-degenerate quadric of the pencil which is defined over $k$. We will have to drop this assumption farther in some very particular cases when $q=3$. 

If $\varphi_0$, $\varphi_\infty$ denote the bilinear forms on $k^5$ associated respectively to $Q_0$ and $Q_\infty$, then $\varphi_\infty$ is non degenerate, and there is an unique endomorphism $u$ of $k^5$ such that $\varphi_\infty(u(x),y)=\varphi_0(x,y)$ for any vectors $x,y$.  Moreover $u$ is symmetric with respect to $\varphi_\infty$.

The vector space $V=k^5$, endowed with the action of $u$ becomes a $k[T]$-module of finite length that we denote by $V_u$, and $\varphi_\infty$ is a nondegenerate symmetric $k[T]$-bilinear form on $V_u$.

\begin{definition}
We call the pair $(V_u,\varphi_\infty)$ a \emph{quadratic $k[T]$-module} in the following.
\end{definition}

Since $k[T]$ is a principal ideal domain, the module $V_u$ decomposes as a direct sum of primary cyclic components:
\begin{align}\label{eq_cyclic_primary_decomp}
V_u
&=
\bigoplus_{i = 1}^t
\left(\bigoplus_{j=1}^{n_i} k[T].x_{ij}\right),
&
&\ann\left(x_{ij}\right) = P_i^{m_{ij}}k[T],
\end{align}
where the $x_{ij}$ are elements of $V_u$, the polynomials $P_1,\ldots,P_t\in k[T]$ are irreducible and pairwise distinct, and the exponents $m_{i1},\ldots m_{it_i}$ are (not necessarily distinct) positive integers. Note that when $k$ is algebraically closed we recover the Segre symbol (\ref{segresymbol}).

Waterhouse has proved that the cyclic primary components of the decomposition \eqref{eq_cyclic_primary_decomp} can be chosen in such a way that they are two-by-two $\varphi_\infty$-orthogonal.

Moreover, it turns out that each cyclic component can be explicitly described. To this end, we need to introduce some notations. Let $F\in k[T]$ be a non constant,
unitary polynomial and let us consider the quotient algebra $k[T]/(F)$. We put $t = T\bmod F$, we choose some $\delta\in k[T]/(F)$ and we define $\lambda_F$,~$\Phi_F$, $\delta\cdot\lambda_F$ and~$\delta\cdot\Phi_F$ as
\begin{itemize}
\item the linear form $\lambda_F : k[T]/(F) \to k$ defined as the last vector of the dual basis of $(1,t,\ldots,t^{\deg(F)-1})$;
\item the symmetric bilinear form~$\Phi_F : k[T]/(F) \times k[T]/(F) \to k$ defined by $\Phi_F(x,y) = \lambda_F(xy)$;
\item the linear form defined by~$\delta\cdot\lambda_F(x) = \lambda_F(\delta x)$ and 
\item the bilinear form defined by $\delta\cdot\Phi_F(x,y) = \lambda_F(\delta x y)$.
\end{itemize}
Then the linear form $\lambda_F$ is a {\em dualizing form} and the algebra $k[T]/(F)$ with its dualizing form is called a {\em Frobenius
algebra}; by definition, this means that the bilinear form~$\Phi_F$ is non degenerate or equivalently that for
every $\lambda\in\Hom_k\left(k[T]/(F),k\right)$, there exists $\alpha\in k[T]/(F)$
such that $\lambda(x) = \alpha\cdot\lambda_F(x) = \lambda_F(\alpha x)$.

\begin{definition}
The \emph{quadratic module $[\![F,\delta]\!]$} is the $k[T]$-module $k[T]/(F)$ endowed with the bilinear form $\varphi = \delta \cdot \Phi_F$.
\end{definition}

This is an example of quadratic module with cyclic underlying $k[T]$-module. Note that the pair of quadratic forms on the corresponding $k$-vector space is given by $Q_\infty(x)= \lambda_F(\delta x^2)$ and $Q_0(x)=\lambda_F(t\delta x^2)$.

In fact this is a generic example.

\begin{proposition}
Let $(V_u,\varphi)$ be $k[T]$-quadratic module. Suppose that the $k[T]$-module $V_u$ is cyclic,
let $x\in V_u$ be such that $V_u = k[T]\cdot x$, and let~$F\in k[T]$ be a generator of its annihilator. Then there
exists $\delta\in k[T]/(F)$ such that for all $p,q$ in $k[T]/(F)$ we have $\varphi(p\cdot x,q\cdot x) = \delta \cdot \Phi_F(p,q)$.
\end{proposition}

\begin{proof}
By definition of $F$, the map $\iota$ defined by $p \mapsto p\cdot  x$ is an isomorphism of $k[T]$-modules from $k[T]/(F)$
to $V_u = k[T]\cdot x$. The map $\psi$ defined by $\psi(p,q) = \varphi(p\cdot x,q\cdot x)$ for every $p,q\in k[T]/(F)$ is bilinear symmetric. Since $u$
is $\varphi$-symmetric, the multiplication by $t$ endomorphism is $\psi$-symmetric, and we have $\psi(tp,q)=\psi(p,tq)$. In particular,
we get $\psi(p,q) = \psi(1,pq)$. But we know that there exists $\delta\in k[T]/(F)$ such that $\psi(1,q) = \delta\cdot\lambda_F(q)$ for
every $q\in k[T]/(F)$. We deduce that
$$
\varphi(p\cdot x,q\cdot x) = \psi(p,q) = \psi(1,pq) = \delta\cdot\lambda_F(pq) = \delta\cdot\Phi_F(p,q)
$$
and the result follows.
\end{proof}

We can use the quadratic modules we have just defined to give a normal form for all quadratic modules. From \cite[Theorem 1.1 and Section 2]{water}, we have

\begin{theorem}[Waterhouse]
Let $(V,\varphi)$ be a non-degenerate $k[T]$-quadratic module of finite length.

Let~$\left(P_i^{e_{ij}}\right)_{i,j}$, $1\leq i\leq r$, $1\leq j\leq n_{ij}$, be the elementary divisors of~$u$. Then there exists $\delta_{ij} \in (k[T]/(P_i^{e_{ij}}))^\times$ such that the pair $(V,\varphi)$ is isometric to
the orthogonal sum $\bigoplus_{i,j} [\![P_i^{e_{ij}},\delta_{ij}]\!]$.
\end{theorem}

Thus we can associate to the anticanonical model of any degree $4$ del Pezzo surface (with at least one non-singular quadric defined over $\F_q$ in the pencil) a $k[T]$-quadratic module of the form $\bigoplus_{i} [\![F_i,\delta_{i}]\!]$. 

Conversely, to such a module we associate the del Pezzo surface defined by the vanishing of the following two quadratic forms defined for $x=(x_i)$ by
\begin{equation}
\label{eqquadmod}
Q_\infty(x)= \sum_i\lambda_{F_i}(\delta_i x_i^2),~Q_0(x)=\sum_i\lambda_{F_i}(T\delta_i x_i^2)
\end{equation}

\begin{remark}
\label{constantmultiple}
Note that this association is in no way an application: we have arbitrarily chosen the basis $\{Q_0,Q_\infty\}$, and one could replace some $\delta_i$ by $x^2\delta_i$ without changing the resulting del Pezzo surface. Also replacing all $\delta_i$ by $a\delta_i$ for some $a\in \F_q^\times$ gives the quadratic forms $aQ_\infty,aQ_0$ and does not change the del Pezzo surface. 

For this reason, we will frequently be able to put additional restrictions on the $\delta_i$ in the last section, in order to ease the computation of the restricted discriminants of the singular quadrics of the pencil.

Moreover we have chosen to use the elementary divisors of the $k[T]$-module in the above presentation, but we could have chosen the invariant factors instead, or any other possible decomposition in cyclic submodules.
\end{remark}

We have seen that one of the key tools to compute the arithmetic type of a Del Pezzo surface of degree~$4$ is to control
the discriminants of the bases of the rank $4$ singular quadrics in the pencil. From the Frobenius algebras point of view, we have

\begin{lemma}\label{lem_DiscriminantQuadriquesSingulieres}
Let~$F \in k[T]$ be a unitary polynomial and let~$\delta\in k[T]/(F)$.

We denote by $N_F:k[T]/(F)\rightarrow k$ the norm.

\begin{enumerate}
\item\label{item_Dis} The discriminant of the bilinear form~$\delta\cdot\Phi_F$ equals~
\[
(-1)^{\frac{\deg(F)(\deg(F)-1)}{2}}N_F(\delta)
\]
\item\label{item_RestrictedDis} Let~$\theta\in k$ be a root of~$F$. Then
the bilinear form~$(\theta-T)\delta\cdot\Phi_F$ is degenerate of rank $\deg F-1$ and its restricted discriminant
equals
\[
(-1)^{\frac{\deg(F)\left(\deg(F)-1\right)}{2}} N_F(\delta)\delta(\theta)
\]
\end{enumerate}
\end{lemma}

\begin{proof}
$(\ref{item_Dis})$ Put~$n=\deg(F)$. Let~$\Ccal = (1,\ldots,t^{n-1})$ be the canonical basis of~$k[T]/(F)$, 
and let~$\Dcal=(f_1,\ldots,f_n)$ be its dual basis with respect to $\lambda_F$, i.e. such that $\lambda_F(t^{i-1}f_j) = \delta_{ij}$, $1\leq i,j\leq n$. The coordinates of the elements of~$\Dcal$ in the basis $\Ccal$ are given by the columns of the matrix:
$$
P
=
\begin{pmatrix}
a_{1}   & \cdots                 & a_{n-1}&1\\
\vdots &\reflectbox{$\ddots$}& 1 &\\
a_{n-1} &\reflectbox{$\ddots$} &&\\
1      &       &&
\end{pmatrix}
$$
where the~$a_i$'s are the coefficients of~$F = T^n + a_{n-1}T^{n-1} + \cdots + a_1T+a_0$. Moreover, for any~$x\in k[T]/(F)$,
one has~$x = \sum_{i=1}^n \lambda_F(x f_i) e_i$.
Let us compare the matrices of the bilinear form~$\delta\cdot\phi_F$ and of the linear map~$m_\delta$ (multiplication by~$\delta$) in the
basis~$\Ccal$; by definition they are equal to
\begin{align*}
&\Mat\left(\delta\cdot\Phi_F, \Ccal\right)
=
\left(\lambda_F(\delta e_i e_j)\right)_{1\leq i,j\leq n}
&
&\text{and}
&
\Mat\left(m_\delta, \Ccal\right)
=
\left(\lambda_F(\delta f_i e_j)\right)_{1\leq i,j\leq n}.
\end{align*}
Thus they are related by the formula~$\Mat\left(m_\delta, \Ccal\right) ={}^t P\Mat\left(\delta\cdot\Phi_F, \Ccal\right)$ and the result follows from the equality of the determinants.

$(\ref{item_RestrictedDis})$ Let $e$ denote the multiplicity of the root $\theta$ for $F$, and~set~$F = (T-\theta)^e G$; 
there exist~$U,V\in k[T]$ satisfying~$U(T-\theta)^e + V G = 1$. Then for any $x\in k[T]/(F)$, we have
\begin{eqnarray*}
\lambda_F(x) & = & \lambda_F(xU(T-\theta)^e)+\lambda_F(xVG) \\
             & = & \lambda_F((xU \bmod G)(T-\theta)^e)+\lambda_F((xV \bmod (T-\theta)^e)G) \\
						& = & \lambda_G(xU)+\lambda_{(T-\theta)^e}(xV) \\
\end{eqnarray*}
where the last equality holds since $G$ and $(T-\theta)^e$ are unitary. We deduce the following orthogonal decomposition
\begin{equation}
\label{orthdec}
(\theta - T)\delta\cdot \Phi_F
=
(\theta - T)V\delta\cdot \Phi_{(T-\theta)^e} \oplus (\theta - T)U\delta\cdot \Phi_G.
\end{equation}
In order to compute the restricted discriminant of the first component, we note that $(\theta - T)V\delta\cdot \lambda_{(T-\theta)^e}((T-\theta)^a)=-\lambda_{(T-\theta)^e}((T-\theta)^{a+1}\delta V)$ is equal to $0$ as long as $a\geq e-1$, and to $-\delta(\theta)V(\theta)$ when $a=e-1$. We get
$$
\Mat\left((\theta - T)V\delta\cdot \Phi_{(T-\theta)^e}, \left((T-\theta)^{e-1},\ldots,1\right)\right)
=
\begin{pmatrix}
& & & 0\\
& & \reflectbox{$\ddots$}& -V(\theta)\delta(\theta)\\
& \reflectbox{$\ddots$} &\reflectbox{$\ddots$} &\vdots\\
0 & -V(\theta)\delta(\theta)& \cdots& \star\\
\end{pmatrix}.
$$
This form has rank~$e-1$; its kernel is generated by $(T-\theta)^{e-1}$, and its restricted discriminant equals:
$$
(-1)^{\frac{(e-1)(e-2)}{2}} \times (-1)^{e-1} \times V(\theta)^{e-1}\delta(\theta)^{e-1}
=
(-1)^{\frac{e(e-1)}{2}} \frac{\delta(\theta)^{e-1}}{G(\theta)^{e-1}}
$$
because~$V(\theta)G(\theta) = 1$.

For the second component of the decomposition (\ref{orthdec}), putting~$d=\deg(G)$, and using $(\ref{item_Dis})$, its discriminant
equals
$$
(-1)^{d\left(d-1\right)}N_G\left((\theta - T)U\delta\right)
=
(-1)^{\frac{d\left(d-1\right)}{2}+ed}
\frac{N_G(\delta)}{ N_G(\theta -T)^{e-1}}
$$
since~$U(\theta -T) \times (-1)^e(\theta -T)^{e-1} = 1 \bmod{G}$. 

We have $N_G(\theta -T)=G(\theta)$, and the product of these two discriminants gives the result as $N_F(\delta)=N_G(\delta)N_{(T-\theta)^e}(\delta)=N_G(\delta)N_{T-\theta}(\delta)^e=N_G(\delta)\delta(\theta)^e$.
\end{proof}

\begin{remark}
At this point, we make the following observation: when the underlying $k[T]$-module is cyclic, all singular quadrics in the correspondong pencil (defined by equations (\ref{eqquadmod})) have corank one.

The converse is true: when the underlying $k[T]$-module is not cyclic, it admits at least two invariant factors, and for all roots of the one of smallest degree the above result shows that the corresponding singular quadric has corank equal to the number of invariant factors. 
\end{remark}

We end this section with a criterion that determines the field of definition of certain singular points of the del Pezzo surface. We will use it in the next section in order to determine the arithmetic types when the pencil contains a rank $3$ quadric.  

\begin{lemma}
\label{rationalornot}
Let $X$ be a degree four del Pezzo surface over $\F_q$ such that for some $\theta\in \F_q$, the non-degenerate quadratic $k[T]$-module associated to $X$ can be written as an orthogonal sum $[\![T-\theta,\delta_1]\!]\oplus[\![T-\theta,\delta_2]\!]\oplus M$, where the annihilator $P$ of the $k[T]$-module $M$ satisfies $P(\theta)\neq 0$.

Then the quadric of the pencil with equation $Q_0-\theta Q_\infty=0$ has rank $3$, and if $\ell\simeq \P^1$ is its vertex, the singular del Pezzo variety $X_s$ meets $\ell$ at two distinct points, which are defined over $\F_q$ if $-\delta_1\delta_2$ is a square in $\F_q^\times$, and conjugate over $\F_q$ else.
\end{lemma}

\begin{proof}
We choose a basis $(e_1,\ldots,e_5)$ of $k^5$ which is adapted to the orthogonal decomposition. In this basis, from the expressions (\ref{eqquadmod}) of the forms $Q_0,Q_\infty$, their matrices can be written in block diagonal form as respectively
\[
\Mat(Q_\infty)=\begin{pmatrix}
\delta_1 & 0 &  0\\
0& \delta_2 &  0\\
0 &0 &A_\infty\\
\end{pmatrix}.
,~\Mat(Q_0)=\begin{pmatrix}
\theta\delta_1 &0  & 0\\
0 & \theta\delta_2& 0\\
0 &0&A_0\\
\end{pmatrix}.
\]
Since the annihilator $P$ of the $k[T]$-module $M$ satisfies $P(\theta)\neq 0$, the matrix $A_0-\theta A_\infty$ is invertible, and the quadric with equation $Q_0-\theta Q_\infty$ has rank $3$.

Its vertex is the line $(x_1:x_2:0:0:0)$ in $\P^4$, the intersection of which with the quadric $Q_\infty=0$ is defined by the equation $\delta_1x_1^2+\delta_2x_2^2=0$.
Now we have $\delta_1\delta_2\neq 0$ since the form $Q_\infty$ is assumed non degenerate, and we get two solutions since the characteristic is odd. The last assertion is classical.

\end{proof}

\subsection{Construction of degree four del Pezzo surfaces of any arithmetic type}

We are ready to show Theorem \ref{alltypes} \ref{alltypesdeg4}. Our strategy is the following : in order to construct a surface of a given arithmetic type, we construct a quadratic module as a sum of cyclic modules in normal form. We proceed geometric type by geometric type, and for each one we do the following
\begin{itemize}
	\item[1.] the Segre symbol gives the factorisation of the characteristic polynomial over $\overline{\F}_q$, and we deduce from it a decomposition of a quadratic $k[T]$-module associated to $X$. We choose it cyclic when this is possible, but in general we choose the decomposition that have seemed us best adapted to the computation of the restricted discriminants from lemma \ref{lem_DiscriminantQuadriquesSingulieres}, and of the information needed in lemma \ref{rationalornot}. We sometimes put additional restrictions on the $\delta_i$ in order to ease this computation.
	\item[2.] we describe all possible Galois actions on the graph of conics and $(-2)$-curves given in Table \ref{Tablegraphs4}; they give the possible arithmetic types. For each one, we use lemmas \ref{simpleroots}, \ref{multipleroots4} and \ref{rationalornot} to describe the factorization type of the characteristic polynomial over $\F_q$, the restricted discriminants of rank $4$ singular quadrics and the squareness of certain invariants; 
	\item[3.] we choose some $\delta_i$ in the cyclic quadratic submodules such that the invariants associated to the singular quadrics in the pencil satisfy the conditions listed at the second point above.  
\end{itemize}

The existence of a polynomial with this factorization, and of some $\delta_i$ with the claimed properties, is sufficient to ensure that the del Pezzo surface associated to the quadratic $k[T]$-module we have constructed has the arithmetic type we need.

\begin{remark}
As already mentioned in the introduction, we have used the mathematical software {\tt magma} to construct a couple of quadrics in $\P^4$ for each type of degree four singular del Pezzo surfaces over a given finite field. There is a slight difference between the procedure used in this program and the description of the quadratic modules presented below: here we decompose the modules using a biggest possible cyclic module. In the program we use another decomposition, with more terms. There is a Chinese remainder theorem for quadratic modules --not presented here-- that gives an equivalence between the two descriptions.
\end{remark}

\subsubsection{Ordinary geometric type}

The factorization of $P$ over $\overline{\F}_q$ is $P(T)=\prod_{i=1}^5(T-\theta_i)$ with distinct $\theta_i$, $1\leq i\leq 5$.

All singular quadrics in the pencil have rank four, and we can associate to the del Pezzo surface $X$ a cyclic quadratic module $[\![P,\delta]\!]$. We will construct $\delta$ such that $N_P(\delta)$ is a square in $\F_q^\times$; then from lemma \ref{lem_DiscriminantQuadriquesSingulieres}, the restricted discriminants of the four singular quadrics are the $\Delta(\theta_i)=N_P(\delta)\delta(\theta_i)\in \F_q(\theta_i)^\times$ and we have $\Delta(\theta_i)\equiv\delta(\theta_i)\bmod \F_q(\theta_i)^{\times2}$. 

Applying lemma \ref{simpleroots}, we see that the signed types representing the conjugacy classes in $W(\D_5)$ (which are the arithmetic types of ordinary degree four del Pezzo surfaces) give
\begin{itemize}
	\item the degrees of the irreducible factors of $P$ in $\F_q[T]$ just by removing the bars;
	\item the reduced discriminants $\Delta(\theta_i) \bmod \F_q(\theta_i)^{\times2}$, ie the classes $\delta(\theta_i) \bmod \F_q(\theta_i)^{\times2}$ from our convention on $N_P(\delta)$;
\end{itemize}

For each geometric type, we present the correspondance between signed types and properties of $P$ and $\delta$ in a table. In the last column, we write $\{\underbrace{\square,\cdots,\square}_{d ~{\rm times}}\}$ when the corresponding $d$ roots $\theta_i,\ldots,\theta_{i+d-1}$ are conjugate over $\F_q$ and $\delta(\theta_i)$ is a square in $\F_{q^d}^\times$, and $\{\underbrace{\boxtimes,\cdots,\boxtimes}_{d ~{\rm times}}\}$ when $\delta(\theta_i)$ is not a square in $\F_{q^d}^\times$.

\begin{longtable}{l|l}
\hline
\text{Signed type}&$(\delta(\theta_1),\ldots,\delta(\theta_5))$\\
\hline
\hline
\endhead
\hline
\endfoot
$11111$&$(\square,\square,\square,\square,\square)$\\
\hline
$111\overline{1}\overline{1}$&$(\square,\square,\square,\boxtimes,\boxtimes)$\\
\hline
$1\overline{1}\overline{1}\overline{1}\overline{1}$&$(\square,\boxtimes,\boxtimes,\boxtimes,\boxtimes)$\\
\hline
$2111$&$(\{\square,\square\},\square,\square,\square)$\\
\hline
$21\overline{1}\overline{1}$&$(\{\square,\square\},\square,\boxtimes,\boxtimes)$\\
\hline
$221$&$(\{\square,\square\},\{\square,\square\},\square)$\\
\hline
$311$&$(\{\square,\square,\square\},\square,\square)$\\
\hline
$2\overline{2}\overline{1}$&$(\{\square,\square\},\{\boxtimes,\boxtimes\},\boxtimes)$\\
\hline
$41$&$(\{\square,\square,\square,\square\},\square)$\\
\hline
$\overline{2}11\overline{1}$&$(\{\boxtimes,\boxtimes\},\square,\square,\boxtimes)$\\
\hline
$\overline{2}\overline{1}\overline{1}\overline{1}$&$(\{\boxtimes,\boxtimes\},\boxtimes,\boxtimes,\boxtimes)$\\
\hline
$\overline{2}\overline{2}1$&$(\{\boxtimes,\boxtimes\},\{\boxtimes,\boxtimes\},\square)$\\
\hline
$5$&$(\{\square,\square,\square,\square,\square\})$\\
\hline
$32$&$(\{\square,\square,\square\},\{\square,\square\})$\\
\hline
$3\overline{1}\overline{1}$&$(\{\square,\square,\square\},\boxtimes,\boxtimes)$\\
\hline
$\overline{3}1\overline{1}$&$(\{\boxtimes,\boxtimes,\boxtimes\},\square,\boxtimes)$\\
\hline
$\overline{4}\overline{1}$&$(\{\boxtimes,\boxtimes,\boxtimes,\boxtimes\},\boxtimes)$\\
\hline
$\overline{3}\overline{2}$&$(\{\boxtimes,\boxtimes,\boxtimes\},\{\boxtimes,\boxtimes\})$\\
\end{longtable}

Note that a polynomial $P$ and an element $\delta$ with the desired properties exist as long as we have $q\geq 5$. But when we have $q=3$, one cannot construct a polynomial with five roots over $\F_q$, and some types do not exist \cite{trepa}.

\subsubsection{Geometric type $\A_1$}

The factorization of $P$ over $\overline{\F}_q$ is $P(T)=\prod_{i=1}^3(T-\theta_i)(T-\theta_4)^2$ with distinct $\theta_i$, $1\leq i\leq 4$, and $\theta_4\in \F_q$.

All singular quadrics in the pencil have rank four, and we can associate to the del Pezzo surface $X$ a cyclic quadratic module $[\![P,\delta]\!]$. For each arithmetic type, we will choose some $\delta$ such that $N_P(\delta)$ is a square in $\F_q^\times$; then from lemma \ref{lem_DiscriminantQuadriquesSingulieres}, the restricted discriminants of the four singular quadrics are the $\Delta(\theta_i)=N_P(\delta)\delta(\theta_i)\in \F_q(\theta_i)^\times$ and we have $\Delta(\theta_i)\equiv\delta(\theta_i)\bmod \F_q(\theta_i)^{\times2}$.  

An element of $\Stab(\R_{\irr})$ must act on the last two couples of complementary conics as the identity (of signed type $11$) or the bitransposition $(C_4C_5')(C_5C_4')$ (of signed type $2$); note that both have even signed types. Since it is an element of $W(\D_5)$, it must act on the first three couples as an element of even signed type, ie as an element of $H_3 \rtimes \S_3$, where $H_3$ is the subgroup of $\Z/2\Z^3$ which is the kernel of the map $\varepsilon=(\varepsilon_1,\ldots,\varepsilon_3)\mapsto \sum_{i=1}^3 \varepsilon_i$. We get the following ten conjugacy classes in $\Stab(\R_{\irr})$
\begin{align*}
&111\cdot 11
&
&1\overline{1}\overline{1} \cdot 11
&
&21 \; 11
&
&\overline{2}\overline{1} \cdot 11
&
&3 \cdot 11
&
&111\cdot 2
&
&1\overline{1}\overline{1}\cdot 2
&
&21 \cdot 2
&
&\overline{2}\overline{1} \cdot 2
&
&3\cdot 2.
\end{align*}

Write $P(T)=F(T)(T-\theta_4)^2$; applying lemma \ref{simpleroots}, we see that the first part of the above signed types give
\begin{itemize}
	\item the degrees of the irreducible factors of $F$ in $\F_q[T]$ just by removing the bars;
	\item for $1\leq i \leq 3$, the reduced discriminants $\Delta(\theta_i) \bmod \F_q(\theta_i)^{\times2}$, ie the classes $\delta(\theta_i) \bmod \F_q(\theta_i)^{\times2}$ from our convention on $N_P(\delta)$;
\end{itemize}
Then, from lemma \ref{multipleroots4}, the classes $C_4$ and $C_5'$, which correspond to the fibers of the morphisms defined by the last singular quadric $Q_0-\theta_4 Q_\infty$, are fixed by Frobenius when $\delta(\theta_4)$ is a square in $\F_q^\times$, and exchanged else.  This gives us the second part of the signed type. We obtain the following table

\begin{longtable}{l|l|l|l}
\hline
$N^\circ$&\text{Cl. ~Stab.}&\text{Cl.~Weyl}&$(\delta(\theta_1),\ldots,\delta(\theta_4))$\\
\hline
\hline
\endhead
\hline
\endfoot
$1$&$111\cdot11$&$11111$&$(\square,\square,\square,\square)$\\
\hline
2&$111\cdot2$&$2111$&$(\square,\square,\square,\boxtimes)$\\
\hline
3&$21\cdot11$&$2111$&$(\{\square,\square\},\square,\square)$\\
\hline
4&$1\overline{1}\overline{1}\cdot11$&$111\overline{1}\overline{1}$&$(\boxtimes,\boxtimes,\square,\square)$\\
\hline
5&$21\cdot2$&$221$&$(\{\square,\square\},\square,\boxtimes)$\\
\hline
6&$1\overline{1}\overline{1}\cdot2$&$21\overline{1}\overline{1}$&$(\boxtimes,\boxtimes,\square,\boxtimes)$\\
\hline
7&$3\cdot11$&$311$&$(\{\square,\square,\square\},\square)$\\
\hline
8&$\overline{2}\overline{1}\cdot2$&$2\overline{2}\overline{1}$&$(\{\boxtimes,\boxtimes\},\boxtimes,\boxtimes)$\\
\hline
9&$\overline{2}\overline{1}\cdot11$&$\overline{2}11\overline{1}$&$(\{\boxtimes,\boxtimes\},\boxtimes,\square)$\\
\hline
10&$3\cdot2$&$32$&$(\{\square,\square,\square\},\boxtimes)$\\
\end{longtable}

Note that we can always construct a polynomial with the claimed factorization, except when $q=3$ in cases 1, 2, 4 and 6 since then $P$ needs to have four distinct roots in $\F_q$. Elements $\delta\in k[T]/(P)$ with the required properties always exist. Finally note that we have $N_P(\delta)\in \F_q^{\times2}$ in all cases, as required.

It remains to construct surfaces of types 1,2,4 and 6 over $\F_3$. We mimic the above construction: we consider the quadratic forms with the following block diagonal matrices
\[
\Mat(Q_0)=\begin{pmatrix}
 \delta_\infty & & & &\\
 & 1 & & & \\
 & & 1 & & \\
 &  & & 0 & 0   \\
& & & 0 & \delta_{0} \\
\end{pmatrix},~
\Mat(Q_\infty)=\begin{pmatrix}
 0 & & & &\\
 & 1 & & & \\
 & & -1 & & \\
 &  & & 0 & \delta_{0}   \\
& & & \delta_{0} & 0 \\
\end{pmatrix}
\]

The four quadrics of the pencil which are defined over $\F_3$ (namely $Q_t=Q_0+tQ_\infty$, $t\in \F_3$ and $Q_\infty$) are singular of rank four, and the vertex of $Q_0$ --here $(0:0:0:1:0)$-- is the unique singular point of the base of the pencil. We deduce that $X_s$ has geometric type $\A_1$. 

The restricted discriminants are, modulo squares
\[
\Delta(Q_0)\equiv \delta_0\delta_\infty,~\Delta(Q_1)\equiv \Delta(Q_2)\equiv \delta_\infty,~\Delta(Q_\infty)\equiv 1
\]
and we get the types $1,2,4$ and $6$ respectively choosing $(\delta_0,\delta_\infty)$ of the form $(\square,\square)$, $(\boxtimes,\square)$, $(\boxtimes,\boxtimes)$ and $(\square,\boxtimes)$.

\subsubsection{Geometric type $2\A_1$ with $9$ lines}

The factorization of $P$ over $\overline{\F}_q$ is $P(T)=(T-\theta_1)(T-\theta_2)^2(T-\theta_3)^2$ with distinct $\theta_i$, $1\leq i\leq 3$, and $\theta_1\in \F_q$.

All singular quadrics in the pencil have rank four, and we can associate to the del Pezzo surface $X$ a cyclic quadratic module $[\![P,\delta]\!]$. We will always choose some $\delta$ such that $N_P(\delta)$ (or equivalently $\delta(\theta_1)$) is a square in $\F_q^\times$; then we have the congruence $\Delta(\theta_i)\equiv\delta(\theta_i)\bmod \F_q(\theta_i)^{\times2}$ for the restricted discriminants.  

An element of $\Stab(\R_{\irr})$ can
\begin{itemize}
	\item fix the $(-2)$-curves, and the sets $\{C_2,C_2',C_3,C_3'\}$ and $\{C_4,C_4',C_5,C_5'\}$; in this case it acts on the last four couples of complementary conics as
	\begin{itemize}
	\item the identity, of signed type $11\cdot 11$ or 
	\item a bitransposition such as $(C_2C_3')(C_2C_3')$, of signed type $2\cdot 11$ or
	\item four transpositions such as $(C_2C_3')(C_2C_3')(C_4C_5')(C_4C_5')$, of signed type $2\cdot 2$ ;
\end{itemize}
	\item exchange the $(-2)$-curves, and the sets $\{C_2,C_2',C_3,C_3'\}$ and $\{C_4,C_4',C_5,C_5'\}$; in this case it acts on the last four couples of complementary conics as
	
	\begin{itemize}
		\item four transpositions such as $(C_2C_4)(C_2'C_4')(C_3C_5)(C_3'C_5')$, of signed type we note here $\{2\cdot 2\}$ or
		\item a permutation such as $(C_3C_5C'_2C'_4)(C'_3C'_5C_2C_4)$, of signed type $4$.
	\end{itemize}
\end{itemize}
Note that all have even signed types, and the action must be trivial on the remaining couple of conics. We get the following five conjugacy classes in $\Stab(\R_{\irr})$
\begin{align*}
&1\cdot 11\cdot 11,
&
&1\cdot 2 \cdot 11,
&
&1\cdot 2\cdot 2,
&
&&1\cdot \{2\cdot 2\},
&
&1\cdot 4.
\end{align*}

Applying lemma \ref{multipleroots4}, we get the following table

$$
\begin{array}[t]{l|l|l|l}
\hline
N^\circ&\text{Cl. ~Stab.}&\text{Cl. ~Weyl}&(\delta(\theta_1),\delta(\theta_2),\delta(\theta_3))\\
\hline
11&1\cdot 11\cdot 11&11111&(\square,\square,\square)\\
\hline
12&1\cdot 2 \cdot 11&2111&(\square,\boxtimes,\square)\\
\hline
13&1\cdot 2\cdot 2&221&(\square,\boxtimes,\boxtimes)\\
\hline
14&1\cdot \{2\cdot 2\}&221&(\square,\{\square,\square\})\\
\hline
15&1 \cdot 4&41&(\square,\{\boxtimes,\boxtimes\})\\
\hline
\end{array}
$$

\subsubsection{Geometric type $2\A_1$ with $8$ lines}

The factorization of $P$ over $\overline{\F}_q$ is $P(T)=\prod_{i=1}^3(T-\theta_i)(T-\theta_4)^2$ with distinct $\theta_i$, $1\leq i\leq 4$, and $\theta_4\in \F_q$.

There is a rank three singular quadric in the pencil, corresponding to the root $\theta_4$, and we can associate to the del Pezzo surface $X$ a quadratic module of the form $[\![F,\delta]\!]\oplus [\![T-\theta_4,\delta_1]\!]\oplus[\![T-\theta_4,\delta_2]\!]$, where $F(T)=\prod_{i=1}^3(T-\theta_i)$. 

With the help of lemma \ref{lem_DiscriminantQuadriquesSingulieres}, and using a block diagonal form as in the proof of lemma \ref{rationalornot}, we get that the rank four singular quadrics in the pencil have restricted discriminants 
\[
\Delta(\theta_i)=(\theta_i-\theta_4)^2\delta_1\delta_2(-1)^3N_F(\delta)\delta(\theta_i)\equiv -\delta_1\delta_2N_F(\delta)\delta(\theta_i)\bmod \F_q(\theta_i)^{\times2}
\]
for $1\leq i\leq 3$. We shall construct quadratic modules satisfying $\delta_2=-1$ and $\delta_1N_F(\delta)\in \F_q^{\times2}$, so that $\Delta(\theta_i) \equiv \delta(\theta_i)\bmod \F_q(\theta_i)^{\times2}$ for $1\leq i\leq 3$.

An element of $\Stab(\R_{\irr})$ must act on the last two couples of complementary conics as the identity or as $(C_5C_5')$ of signed type $1\overline{1}$. In the first case it acts on the first three couples as an even signed type permutation, ie as an element of $H_3 \rtimes \S_3$, and in the second as an odd signed type permutation, ie as an element of $\Z/2\Z^3 \rtimes \S_3\setminus H_3 \rtimes \S_3 $. We get the following conjugacy classes in $\Stab(\R_{\irr})$

\[
111\cdot 11,
~
1\overline{1}\overline{1}\cdot 11,
~
21\cdot11,
~
\overline{2}\overline{1}\cdot11,
~
3\cdot11,
~
11\overline{1}\cdot1\overline{1},
~
\overline{1}\overline{1}\overline{1}\cdot1\overline{1},
~
2\overline{1}\cdot1\overline{1},
~
\overline{2}1\cdot1\overline{1},
~
\overline{3}\cdot1\overline{1}.
\]

The second part of the signed type is $11$ exactly when the singular points are defined over $\F_q$; lemma \ref{rationalornot}, applied with our convention $\delta_2=-1$, tells us that this happens exactly when $\delta_1$ is a square in $\F_q^\times$.

It remains to use lemma \ref{simpleroots} to determine the factorization pattern of $F$ and the $\delta(\theta_i)\bmod \F_q(\theta_i)^{\times2}$ from the signed type of the action on the first three couples. We get the following table
$$
\begin{array}[c]{l|l|l|l}
\hline
N^\circ&\text{Cl.~Stab.}&\text{Cl.~Weyl}&(\delta(\theta_1),\delta(\theta_2),\delta(\theta_3),\delta_{1})\\
\hline
16&111\cdot 11&11111&(\square,\square,\square,\square)\\
\hline
17&21\cdot11&2111&(\{\square,\square\},\square,\square)\\
\hline
18&1\overline{1}\overline{1}\cdot11&111\overline{1}\overline{1}&(\boxtimes,\boxtimes,\square,\square)\\
\hline
19&11\overline{1}\cdot1\overline{1}&111\overline{1}\overline{1}&(\boxtimes,\square,\square,\boxtimes)\\
\hline
20&2\overline{1}\cdot1\overline{1}&21\overline{1}\overline{1}&(\{\square,\square\},\boxtimes,\boxtimes)\\
\hline
21&\overline{1}\overline{1}\overline{1}\cdot1\overline{1}&1\overline{1}\overline{1}\overline{1}\overline{1}&(\boxtimes,\boxtimes,\boxtimes,\boxtimes)\\
\hline
22&3\cdot11&311&(\{\square,\square,\square\},\square)\\
\hline
23&\overline{2}1\cdot1\overline{1}&\overline{2}11\overline{1}&(\{\boxtimes,\boxtimes\},\square,\boxtimes)\\
\hline
24&\overline{2}\overline{1}\cdot11&\overline{2}11\overline{1}&(\{\boxtimes,\boxtimes\},\boxtimes,\square)\\
\hline
25&\overline{3}\cdot1\overline{1}&\overline{3}1\overline{1}&(\{\boxtimes,\boxtimes,\boxtimes\},\boxtimes)\\
\hline
\end{array}
$$

Note that we can always construct a polynomial over $\F_q$ with the claimed factorization, except when $q=3$ in cases 16, 18, 19 and 21. Elements $\delta\in k[T]/(F)$ with the required properties always exist. Finally note that we have $\delta_1N_F(\delta)\in \F_q^{\times2}$ in all cases.

It remains to construct surfaces of types 16, 18, 19 and 21 over $\F_3$. We mimic the above construction: we consider the quadratic forms with the following block diagonal matrices
\[
\Mat(Q_0)=\begin{pmatrix}
 \delta_\infty & & & &\\
 & \delta_1 & & & \\
 & & 1 & & \\
 &  & & 0 &    \\
& & &  & 0 \\
\end{pmatrix},~
\Mat(Q_\infty)=\begin{pmatrix}
 0 & & & &\\
 & \delta_1 & & & \\
 & & -1 & & \\
 &  & & \delta_{0} &    \\
& & &  & -1 \\
\end{pmatrix}
\]
 
The quadrics with equations $Q_\infty, Q_0+Q_\infty$ and $Q_0-Q_\infty$ have rank $4$, and the quadric $Q_0$ has rank $3$. The singular points of $X_s$ are the intersections of the vertex of $Q_0$ with $Q_\infty$, that is the points $(0:0:0:a:b)$ with $\delta_{0}a^2-b^2=0$. We get a surface of geometric type $2\A_1$ with $8$ lines as above, since the line joining the singularities is not contained in $X_s$. 

Computing the restricted discriminants, and applying lemma \ref{rationalornot}, we see that we get the types $16,18,19$ and $21$ respectively choosing $(\delta_0,\delta_1,\delta_\infty)$ of the form $(\square,\square,\square)$, $(\square,\boxtimes,\boxtimes)$, $(\boxtimes,\boxtimes,\boxtimes)$ and $(\boxtimes,\square,\square)$.

\subsubsection{Geometric type $\A_2$}

The factorization of $P$ over $\overline{\F}_q$ is $P(T)=(T-\theta_1)(T-\theta_2)(T-\theta_3)^3$ with distinct $\theta_i$, $1\leq i\leq 3$, and $\theta_3\in \F_q$.

All singular quadrics in the pencil have rank four, and we can associate to the del Pezzo surface $X$ a cyclic quadratic module $[\![P,\delta]\!]$. We will construct $\delta$ such that $N_P(\delta)$ is a square in $\F_q^\times$; then we have the congruence $\Delta(\theta_i)\equiv\delta(\theta_i)\bmod \F_q(\theta_i)^{\times2}$ for the restricted discriminants.  

An element of $\Stab(\R_{\irr})$ must act on the last three couples of complementary conics as the identity (of signed type $111$) or the permutation $(C_5C_3')(C_3C_5')(C_4C_4')$ (of signed type $2\overline{1}$). It will act on the first two couples as en even signed permutation in the first case, as an odd signed one in the second. We get the following five conjugacy classes in $\Stab(\R_{\irr})$
\begin{align*}
&11\cdot111,
&
&\overline{1}\overline{1}\cdot111,
&
&2\cdot111,
&
&1\overline{1}\cdot2\overline{1},
&
&\overline{2}\cdot2\overline{1}.
\end{align*}
Note that the signed type is completely determined by its first part, and we just have to describe the action on it via lemma \ref{simpleroots}. We get the following table, where we have chosen $\delta(\theta_3)$ so that the convention $N_P(\delta)=\delta(\theta_1)\delta(\theta_2)\delta(\theta_3)^3\in\F_q^{\times2}$ holds
$$
\begin{array}[t]{l|l|l|l}
\hline
N^\circ&\text{Cl. Stab.}&\text{Cl. Weyl}&(\delta(\theta_1),\delta(\theta_2),\delta(\theta_3))\\
\hline
26&11\cdot111&11111&(\square,\square,\square)\\
\hline
27&2\cdot111&2111&(\{\square,\square\},\square)\\
\hline
28&\overline{1}\overline{1}\cdot111&111\overline{1}\overline{1}&(\boxtimes,\boxtimes,\square)\\
\hline
29&1\overline{1}\cdot2\overline{1}&21\overline{1}\overline{1}&(\boxtimes,\square,\boxtimes)\\
\hline
30&\overline{2}\cdot2\overline{1}&2\overline{2}\overline{1}&(\{\boxtimes,\boxtimes\},\boxtimes)\\
\hline
\end{array}
$$ 

\subsubsection{Geometric type $3\A_1$}

The factorization of $P$ over $\F_q$ is $P(T)=(T-\theta_1)(T-\theta_2)^2(T-\theta_3)^2$ with distinct $\theta_i$, $1\leq i\leq 3$. We assume that the rank three conic in the pencil corresponds to $\theta_3$. We can write the quadratic $k[T]$-module $[\![F,\delta]\!]\oplus [\![T-\theta_3,\delta_1]\!]\oplus [\![T-\theta_3,\delta_2]\!]$ where $F(T):=(T-\theta_1)(T-\theta_2)^2$. 

With the help of lemma \ref{lem_DiscriminantQuadriquesSingulieres}, and using a block diagonal form as in the proof of lemma \ref{rationalornot}, we get that the rank four singular quadrics in the pencil have restricted discriminants 
\[
\Delta(\theta_i)=(\theta_i-\theta_3)^2\delta_1\delta_2(-1)^3N_F(\delta)\delta(\theta_i)\equiv -\delta_1\delta_2\delta(\theta_1)\delta(\theta_i)\bmod \F_q^{\times2}
,~1\leq i\leq 2
\]

 We shall construct quadratic modules satisfying $\delta_2=-1$ and $\delta_1\delta(\theta_1)\in \F_q^{\times2}$, so that $\Delta(\theta_i) \equiv \delta(\theta_i)\bmod \F_q^{\times2}$ for $1\leq i\leq 2$.

An element of $\Stab(\R_{\irr})$ must act on the last two couples of complementary conics as the identity (of signed type $11$) or the permutation $(C_5C_5')$ (of signed type $1\overline{1}$). On the second and third couples, it acts as the identity (of signed type $11$) or as the bitransposition $(C_2C_3')(C_3C_2')$ of signed type $2$. We then have to ``complete'' the action by fixing $C_1$ and $C_1'$ if $C_5$ and $C_5'$ are fixed, and by transposing them else. We get the following four conjugacy classes in $\Stab(\R_{\irr})$
\begin{align*}
&1\cdot 11 \cdot 11,
&
&\overline{1}\cdot 11 \cdot 1\overline{1},
&
&1\cdot 2 \cdot 11,
&
&\overline{1}\cdot 2 \cdot 1\overline{1}.
\end{align*}

As above, we use lemma \ref{multipleroots4} (a) to read the restricted discriminant $\Delta(\theta_2)$ from the action on the second and third couples, and lemma \ref{rationalornot} to get $\delta_1$ from the action on the last couple. We get
$$
\begin{array}[c]{l|l|l|l}
\hline
N^\circ&\text{Cl.~Stab.}&\text{Cl.~Weyl}&(\delta(\theta_1),\delta(\theta_2),\delta_{1})\\
\hline
31&1\cdot 11 \cdot 11&11111&(\square,\square,\square)\\
\hline
32&1\cdot 2 \cdot 11&2111&(\square,\boxtimes,\square)\\
\hline
33&\overline{1}\cdot 11 \cdot 1\overline{1}&111\overline{1}\overline{1}&(\boxtimes,\square,\boxtimes)\\
\hline
34&\overline{1}\cdot 2 \cdot 1\overline{1}&21\overline{1}\overline{1}&(\boxtimes,\boxtimes,\boxtimes)\\
\hline
\end{array}
$$

\subsubsection{Geometric type $\A_1\A_2$}

The factorization of $P$ over $\F_q$ is $P(T)=(T-\theta_1)^2(T-\theta_2)^3$ with distinct $\theta_i$, $1\leq i\leq 2$. There is no rank three conic in the pencil, and we can write the quadratic $k[T]$-module in cyclic form $[\![P,\delta]\!]$. 

Lemma \ref{lem_DiscriminantQuadriquesSingulieres} tells us that the restricted discriminants are $\Delta(\theta_i)=N_P(\delta)\delta(\theta_i)=\delta(\theta_1)^2\delta(\theta_2)^3\delta(\theta_i)\equiv \delta(\theta_2)\delta(\theta_i)\bmod \F_q^{\times2}$
for $1\leq i\leq 2$. We shall construct quadratic modules satisfying $\delta(\theta_2)\in \F_q^{\times2}$, so that $\Delta(\theta_i) \equiv \delta(\theta_i)\bmod \F_q^{\times2}$ for $1\leq i\leq 2$.

An element of $\Stab(\R_{\irr})$ must act on the first two couples of complementary conics as the identity (of signed type $11$) or a bitransposition (of signed type $2$). On the last three couples, it acts as the identity in order to have en even signed type (recall that the non trivial action on these three couples has signed type $2\overline{1}$). We get two conjugacy classes in $\Stab(\R_{\irr})$, namely $11\cdot111$ and $2\cdot111$, and using lemma \ref{multipleroots4} (a), we get the table
$$
\begin{array}[t]{l|l|l|l}
\hline
N^\circ&\text{Cl.~Stab}&\text{Cl.~Weyl}&\left(\delta(\theta_1),\delta(\theta_2)\right)\\
\hline
35&11\cdot111&11111&(\square,\square)\\
\hline
36&2\cdot111&2111&(\boxtimes,\square)\\
\hline
\end{array}
$$

\subsubsection{Geometric type $\A_3$ with five lines}

The factorization of $P$ over $\F_q$ is $P(T)=(T-\theta_1)(T-\theta_2)^4$ with distinct $\theta_i$, $1\leq i\leq 2$. There is no rank three conic in the pencil, and we can write the quadratic $k[T]$-module in cyclic form $[\![P,\delta]\!]$. 

Lemma \ref{lem_DiscriminantQuadriquesSingulieres} tells us that the restricted discriminants are $\Delta(\theta_i)=N_P(\delta)\delta(\theta_i)\equiv \delta(\theta_1)\delta(\theta_i)\bmod \F_q^{\times2}$
for $1\leq i\leq 2$. We shall construct quadratic modules satisfying $\delta(\theta_1)\in \F_q^{\times2}$, so that $\Delta(\theta_i) \equiv \delta(\theta_i)\bmod \F_q^{\times2}$ for $1\leq i\leq 2$.

An element of $\Stab(\R_{\irr})$ must act on the last four couples of complementary conics as the identity (of signed type $1111$) or as $(C_2C_5')(C_2'C_5)(C_3C_4')(C_3'C_4)$ (of signed type $22$). Thus it must act as the identity on the first couple in order to have en even signed type. We get two conjugacy classes in $\Stab(\R_{\irr})$, namely $1\cdot1111$ and $1\cdot22$, and using lemma \ref{multipleroots4} (a), we get the table
$$
\begin{array}[t]{l|l|l|l}
\hline
N^\circ&\text{Cl.~Stab.}&\text{Cl.~Weyl}&(\delta(\theta_1),\delta(\theta_2))\\
\hline
37&1\cdot1111&11111&(\square,\square)\\
\hline
38&1\cdot22&221&(\square,\boxtimes)\\
\hline
\end{array}
$$

\subsubsection{Geometric type $\A_3$ with four lines}

The factorization of $P$ over $\overline{\F}_q$ is $P(T)=(T-\theta_1)(T-\theta_2)(T-\theta_3)^3$ with distinct $\theta_i$, $1\leq i\leq 3$. We can write the quadratic $k[T]$-module $[\![F,\delta]\!]\oplus [\![T-\theta_3,\delta_1]\!]\oplus [\![(T-\theta_3)^2,\delta_2]\!]$ where $F(T):=(T-\theta_1)(T-\theta_2)$. 

With the help of lemma \ref{lem_DiscriminantQuadriquesSingulieres}, and using a block diagonal form as in the proof of lemma \ref{rationalornot}, we get that the rank four singular quadrics in the pencil have restricted discriminants 
\[
\Delta(\theta_i)=(\theta_i-\theta_3)^3\delta_1\delta_2^2N_F(\delta)\delta(\theta_i)\equiv (\theta_i-\theta_3)\delta_1\delta(\theta_1)\delta(\theta_2)\delta(\theta_i)\bmod \F_q^{\times2}
,~1\leq i\leq 2
\]
We shall construct quadratic modules satisfying $\delta_1=(\theta_1-\theta_3)(\theta_2-\theta_3)$, so that $\Delta(\theta_i) \equiv (\theta_j-\theta_3)\delta(\theta_j)\bmod \F_q^{\times2}$ for $1\leq i\neq j\leq 2$.

An element of $\Stab(\R_{\irr})$ must act on the last three couples of complementary conics as the identity (of signed type $111$) or the transposition $(C_5C_5')$ (of signed type $11\overline{1}$). We deduce the possible actions on the first two couples, and the following five conjugacy classes in $\Stab(\R_{\irr})$
\begin{align*}
&11\cdot 111,
&
&\overline{1}\overline{1}\cdot 111,
&
&2\cdot 111,
&
&1\overline{1}\cdot 11\overline{1},
&
& \overline{2}\cdot 11\overline{1}.
\end{align*}
The type is completely determined by the action on the first two couples: it is sufficient to use lemma \ref{simpleroots} to obtain the table
$$
\begin{array}[c]{l|l|l|l}
\hline
N^\circ&\text{Cl.~Stab.}&\text{Cl.~Weyl}&((\theta_1-\theta_3)\delta(\theta_1),(\theta_2-\theta_3)\delta(\theta_2))\\
\hline
39&11\cdot111&11111&(\square,\square)\\
\hline
40&2\cdot111&2111&(\{\square,\square\})\\
\hline
41&\overline{11}\cdot111&111\overline{1}\overline{1}&(\boxtimes,\boxtimes)\\
\hline
42&1\overline{1}\cdot 11\overline{1}&111\overline{1}\overline{1}&(\square,\boxtimes)\\
\hline
43&\overline{2}\cdot 11\overline{1}&\overline{2}11\overline{1}&(\{\boxtimes,\boxtimes\})\\
\hline
\end{array}
$$

\subsubsection{Geometric type $4\A_1$}

The factorization of $P$ over $\overline{\F}_q$ is $P(T)=(T-\theta_1)(T-\theta_2)^2(T-\theta_3)^2$ with distinct $\theta_i$, $1\leq i\leq 3$. We can write the quadratic $k[T]$-module $[\![T-\theta_1,\delta]\!]\oplus [\![F,\delta_1]\!]\oplus [\![F,\delta_2]\!]$ where $F(T):=(T-\theta_2)(T-\theta_3)$. We will construct such a module satisfying the additional assumption $\delta_2=-1$.

An element in $\Stab(R_{\irr})$ acts on the four $(-2)$-curves, and we deduce from this its action on the couples $\{C_i,C_i'\}$, $i\in\{1,2,4,5\}$
\begin{itemize}
	\item[(a)] it fixes the four $(-2)$-curves, and the eight conics classes; then its signed type is $1111$;
	\item[(b)] it fixes two $(-2)$-curves (necessarily corresponding to the same rank $3$ quadric of the pencil, thus in the same ``component'' of the graph), and permutes the other ones; for instance the transposition $(R_3R_4)$ corresponds to $(C_5C_5')$, of signed type $111\overline{1}$;
	\item[(c)] it acts as the bitransposition $(R_1R_2)(R_3R_4)$ on $(-2)$-curves, and on conic classes as $(C_2C_2')(C_5C_5')$, its signed type is $11\overline{11}$
	\item[(d)] it acts as the bitransposition $(R_1R_3)(R_2R_4)$ on $(-2)$-curves, and on conic classes as $(C_1C_4)(C_2C_5)(C_1'C_4')(C_2'C_5')$ of signed type $22$;
	\item[(e)] it acts as the four-cycle $(R_1R_3R_2R_4)$ on $(-2)$-curves and on conic classes as $(C_1C_4)(C_1'C_4')(C_2C_5C_2'C_5')$ of signed type $2\overline{2}$;
 \end{itemize}
Then we complete with the correct action on $\{C_3,C_3'\}$ in order to get an even signed permutation. We get five conjugacy classes
\begin{align*}
&1\cdot 11 11,
&
& \overline{1}\cdot 111\overline{1},
&
& 1\cdot 11\overline{1}\overline{1},
&
& 1\cdot 22,
&
& \overline{1}\cdot 2\overline{2}.
\end{align*}

The two rank three conics can be defined over $\F_q$ (we have $\theta_2,\theta_3\in \F_q$), or over $\F_{q^2}$, and conjugate over $\F_q$. Applying lemma \ref{rationalornot} to $X$ in the first case, and to $X\otimes \F_{q^2}$ in the second, we get the following table under the assumption $\delta_2=-1$

$$
\begin{array}[c]{l|l|l|l}
\hline
N^\circ&\text{Cl.~Stab.}&\text{Cl.~Weyl}&(\delta_{1}(\theta_2),\delta_{1}(\theta_3))\\
\hline
44&1\cdot 1111&11111&(\square,\square)\\
\hline
45&1\cdot11\overline{1}\overline{1}&111\overline{1}\overline{1}&(\boxtimes,\boxtimes)\\
\hline
46&1\cdot22&221&(\{\square,\square\})\\
\hline
47&\overline{1}\cdot111\overline{1}&111\overline{1}\overline{1}&(\boxtimes,\square)\\
\hline
48&\overline{1}\cdot2\overline{2}&2\overline{2}\overline{1}&(\{\boxtimes,\boxtimes\})\\
\hline
\end{array}
$$

\subsubsection{Geometric type $2\A_1\A_2$}

The factorization of $P$ over $\F_q$ is $P(T)=(T-\theta_1)^3(T-\theta_2)^2$ with distinct $\theta_i$, $1\leq i\leq 2$. We can write the quadratic $k[T]$-module $[\![(T-\theta_1)^3,\delta]\!]\oplus [\![T-\theta_2,\delta_1]\!]\oplus [\![T-\theta_2,\delta_2]\!]$. We will construct such a module satisfying the additional assumptions $\delta=1$, $\delta_2=-1$.

We have already seen that an element in $\Stab(R_{\irr})$ can act on the first three couples with a signed type $111$ or $2\overline{1}$, and on the last two with a signed type $11$ or $1\overline{1}$. This leaves us with two conjugacy classes, namely $11\cdot111$ and $2\overline{1}\cdot1\overline{1}$, which are determined by the action on the last couple, ie on the two $\A_1$-singularities. Applying lemma \ref{rationalornot}, and from our assumption $\delta_2=-1$, we get the table 

$$
\begin{array}[c]{l|l|l|l}
\hline
N^\circ&\text{Cl.~Stab.}&\text{Cl.~Weyl}&\delta_{1}\\
\hline
49&111\cdot11&11111&\square\\
\hline
50&2\overline{1}\cdot1\overline{1}&21\overline{1}\overline{1}&\boxtimes\\
\hline
\end{array}
$$

\subsubsection{Geometric type $\A_1\A_3$}

The factorization of $P$ over $\F_q$ is $P(T)=(T-\theta_1)^2(T-\theta_2)^3$ with distinct $\theta_i$, $1\leq i\leq 2$. We can write the quadratic $k[T]$-module $[\![(T-\theta_1)^2,\delta]\!]\oplus [\![(T-\theta_2)^2,\delta_1]\!]\oplus [\![T-\theta_2,\delta_2]\!]$. We will construct such a module satisfying the additional assumptions $\delta_1=\delta_2=1$. The restricted discriminant of the rank four quadric in the pencil (corresponding to root $\theta_1$) is
\[
\Delta(\theta_1)=\delta(\theta_1)(\theta_1-\theta_2)^3\delta_1^2\delta_2\equiv \delta(\theta_1)(\theta_1-\theta_2) \bmod \F_q^{\times2}
\]

An element in $\Stab(R_{\irr})$ acts on the first two couples with a signed type $11$ or $2$, and on the last three with a signed type $111$ or $11\overline{1}$. This leaves us with two conjugacy classes, namely $111\cdot11$ and $2\cdot111$, and they are determined by the action on the first two couples. Applying lemma \ref{multipleroots4} (a), and from our assumptions, we get the table 

$$
\begin{array}[c]{l|l|l|l}
\hline
N^\circ&\text{Cl.~Stab}&\text{Cl.~Weyl}&(\theta_1-\theta_2)\delta(\theta_1)\\
\hline
51&111\cdot11&11111&\square\\
\hline
52&2\cdot111&2111&\boxtimes\\
\hline
\end{array}
$$

\subsubsection{Geometric type $\A_4$}

The factorization of $P$ over $\F_q$ is $P(T)=(T-\theta_1)^5$, and we can consider the  cyclic quadratic $k[T]$-module $[\![(T-\theta_1)^5,\delta]\!]$.

An element in $\Stab(R_{\irr})$ can only act on the graph trivially (the other automorphism of this graph is $(C_1C_5')(C_1'C_5)(C_2C_4')(C_2'C_4)(C_3C_3')$ which has odd signed type $22\overline{1}$). Thus for any choice of $\delta$ we get the arithmetic type $53$.


\subsubsection{Geometric type $\D_4$}

The factorization of $P$ over $\F_q$ is $P(T)=(T-\theta_1)(T-\theta_2)^4$ with distinct $\theta_i$, $1\leq i\leq 2$. We can write the quadratic $k[T]$-module $[\![T-\theta_1,\delta]\!]\oplus [\![(T-\theta_2)^3,\delta_1]\!]\oplus [\![T-\theta_2,\delta_2]\!]$. We will construct such a module satisfying the additional assumptions $\delta=\delta_2=1$. In this case the restricted discriminant of the rank four quadric in the pencil (corresponding to root $\theta_1$) is
\[
\Delta(\theta_1)=-(\theta_2-\theta_1)^4\delta_1^3\delta_2\equiv \delta_1 \bmod \F_q^{\times2}
\]

An element in $\Stab(R_{\irr})$ acts on the last four couples with a signed type $1111$ or $111\overline{1}$, and on the first one with a signed type $1$ or $\overline{1}$. This leaves us with two conjugacy classes, namely $1\cdot1111$ and $\overline{1}\cdot111\overline{1}$, and they are determined by the action on the first couple. Applying lemma \ref{simpleroots} to the value of $\Delta(\theta_1)$ above, we get the table 

$$
\begin{array}[c]{l|l|l|l}
\hline
N^\circ&\text{Cl.~Stab}&\text{Cl.~Weyl}&\delta_{1}\\
\hline
54&1\cdot1111&11111&\square\\
\hline
55&\overline{1}\cdot111\overline{1}&111\overline{1}\overline{1}&\boxtimes\\
\hline
\end{array}
$$

\subsubsection{Geometric type $2\A_1\A_3$}

The factorization of $P$ over $\F_q$ is $P(T)=(T-\theta_1)^2(T-\theta_2)^3$ with distinct $\theta_i$, $1\leq i\leq 2$. We can write the quadratic $k[T]$-module $[\![T-\theta_1,\delta_1]\!]\oplus [\![T-\theta_1,\delta_2]\!]\oplus [\![(T-\theta_2)^2,\eta_1]\!]\oplus [\![T-\theta_2,\eta_2]\!]$. We will construct such a module satisfying the additional assumptions $\delta_2=-1$ and $\eta_1=\eta_2=1$.

An element in $\Stab(R_{\irr})$ acts on the first two couples with a signed type $11$ or $1\overline{1}$, and on the last three with a signed type $111$ or $11\overline{1}$. This leaves us with two conjugacy classes, namely $11\cdot111$ and $1\overline{1}\cdot11\overline{1}$, and they are determined by the action on the first two couples, ie by the action on the two $\A_1$-singularities. Applying lemma \ref{rationalornot} with $\delta_2=-1$, we get the table 

$$
\begin{array}[c]{l|l|l|l}
\hline
N^\circ&\text{Cl.~Stab.}&\text{Cl.~Weyl}&\delta_{1}\\
\hline
56&11\cdot111&11111&\square\\
\hline
57&1\overline{1}\cdot11\overline{1}&111\overline{1}\overline{1}&\boxtimes\\
\hline
\end{array}
$$

\subsubsection{Geometric type $\D_5$}

The factorization of $P$ over $\F_q$ is $P(T)=(T-\theta)^5$. We can write the quadratic $k[T]$-module $[\![(T-\theta)^4,\delta_1]\!]\oplus [\![T-\theta,\delta_2]\!]$.

An element in $\Stab(R_{\irr})$ can only act on the graph trivially (the other automorphism of this graph is $(C_5C_5')$ which has odd signed type $1111\overline{1}$). Thus for any choice of $\delta_i$ we get the arithmetic type $58$.
%

\section{Construction of singular del Pezzo surfaces of degree three.}
\label{sec5}

The aim of this section is to prove the last two assertions of Theorem \ref{alltypes}. 

We blow up the degree four surfaces from the preceding section at well chosen rational points in order to construct degree three surfaces, but we also use some direct constructions by blowing up some well chosen configurations of points in the projective plane.

\subsection{Blowing up degree four surfaces}

Our first construction of degree three del Pezzo surfaces is by blowing up a rational point not lying on any of the negative curves of a degree four del Pezzo surface. In this way we get a surface of the same Dynkin type (note that for degree three, the geometric type coincides with the Dynkin type), whose zeta function is the one of the degree four surface, divided by $1-qT$. This (and the Galois action on the $(-2)$-curves) makes it very easy to control the arithmetic type of the new surface.

In order to do this, we count the number of rational points not lying on any negative curve on a degree four weak del Pezzo surface of a given arithmetic type. If a rational point lies on a negative curve, then the curve must be defined over $\F_q$, or the point is the intersection of two Galois conjugate curves (thus defined over $\F_{q^2}$). We deduce that the number we are looking for is (note that there is no $3$-cycle in any graph of negative curves, and such three curves cannot be concurrent)
\begin{eqnarray*}
N & = & \#X(\F_q)-(\# \N(\F_q)(q+1)-I_1)-I_2\\
& = & q^2-tq+1-(\# \N(\F_q)(q+1)-I_1)-I_2\\
\end{eqnarray*}
where $t$ is the trace of the action of the Frobenius operator on $\Pic(X\otimes\overline{\F}_q)$, $\N(\F_q)$ is the number of negative curves defined over $\F_q$, $I_1$ is the number of their intersection points (which is readily computed as an intersection number from the $\N(\F_q)$ curves above), and $I_2$ is the number of couples of conjugate negative curves intersecting themselves.

We compute these data and present them in table \ref{galnegcurves}, which is constructed as follows. For each geometric type of degree four del Pezzo surfaces, we consider the graph of negative curves given in \cite[Proposition 6.1]{ct}. We denote the curves in the same way, except for the $(-2)$-curves : recall that we denote by $r_i$ the curve with class $E_i-E_{i+1}$, and by $r_{ijk}$ the curve with class $E_0-E_i-E_j-E_k$.

For each arithmetic type in degree four, we give in the second column an element of the Weyl group in the conjugacy class corresponding to the Frobenius action. We present it as a composite of reflections, where we denote by $s_{ij}$ (\emph{resp.} $s_{ijk}$) the reflection around the root $E_i-E_j$ (\emph{resp.} $E_0-E_i-E_j-E_k$). In the third column we present the action of this element on the negative curves, as a product of cycles. As explained above, this is sufficient to determine the numbers $t,\N(\F_q),I_1,I_2$ and $N$; they are given in the following columns. 

The last column gives the type of the degree three surface obtained by blowing up a rational point not lying on any negative curve, if any.

\begin{center}
\begin{table}[h]
\label{galnegcurves}
{\scriptsize
$\begin{array}{|c|l|c|l|l|c|}
\hline
\textrm{Type} & \textrm{Frobenius in }W(\D_5)  & \textrm{Galois action on neg. curves} &  (t,\# \N(\F_q) ,I_1 ,I_2)& N & \textrm{Type} \\
\hline
\hline
1 & \A_1 ( 12 ) & Id &   (6,13 ,24 ,0) & q^2-7q+12 & 1 \\
2 &  s_{12}  & (\ell_{1}\ell_2)(\ell_{13}\ell_{23})(\ell_{14}\ell_{24}) &   (4,7,6,0) & q^2-3q & 2\\
3 &  s_{123}  & (\ell_{45}q)(\ell_1\ell_{23})(\ell_2\ell_{13})(\ell_3\ell_{12}) &   (4,5,0,0) & q^2-q-4 & 2 \\
4 & s_{12} \circ s_{345} & (\ell_{12}q)(\ell_1\ell_{2})(\ell_3\ell_{45})(\ell_5\ell_{34})(\ell_{14}\ell_{24})(\ell_{13}\ell_{23})  & (2 , 1 ,  4 , 0) & q^2+q+4 & 3\\
5 & s_{12} \circ s_{123} &  (\ell_{3}q)(\ell_2\ell_{13})(\ell_1\ell_{23})(\ell_5\ell_{34})(\ell_{12}\ell_{45})  & (2 , 3 ,  2 , 2) & q^2-q-2 & 3\\
6 & s_{12} \circ s_{345}\circ s_{123} & (\ell_1\ell_{13})(\ell_{3}q)(\ell_2\ell_{23})(\ell_5\ell_{34})(\ell_{12}\ell_{45})(\ell_{14}\ell_{24})  & (0 , 1 , 0 , 4) & q^2-q-4 & 4 \\
7 & s_{12} \circ s_{23}&  (\ell_1\ell_2\ell_3)(\ell_{13}\ell_{12}\ell_{23})(\ell_{14}\ell_{24}\ell_{34})  & (3 , 4 , 3  , 0) & q^2-q & 6 \\
8 & s_{13} \circ s_{23} \circ s_{145} \circ s_{123} & (\ell_{1}q\ell_{3}\ell_{13}) (\ell_{2}\ell_{23}\ell_{45}\ell_{12}) (\ell_{14}\ell_{5}\ell_{34}\ell_{24}) & (0 , 1 , 0 , 0) & q^2-q & 8 \\
9 &  s_{345}\circ s_{23} \circ s_{12} & (\ell_{23}q\ell_{12}\ell_{13}) (\ell_{2}\ell_{1}\ell_{45}\ell_{3}) (\ell_{14}\ell_{5}\ell_{34}\ell_{24})   & (2 , 1 , 0 , 0) & q^2+q & 9 \\
10 & s_{12} \circ s_{23} \circ s_{123}  & (\ell_{14}\ell_{24}\ell_{34})(\ell_{45}q)(\ell_{1}\ell_{13}\ell_{3}\ell_{23}\ell_{2}\ell_{12}) & (1 , 2 , 1 , 0) & q^2-q & 7\\
\hline
11 & 2\A_1 ( 9 ) & Id & (6 , 11 , 16 , 0) & q^2-5q+6 & 12\\
12 & s_{123} & (q\ell_{45})(\ell_1\ell_{23})(\ell_3\ell_{12})   & (4 , 5 , 4 , 0) &q^2-q-2 & 13\\
13 & s_{123} \circ s_{145} &  (q\ell_{1})(\ell_3\ell_{12})(\ell_5\ell_{14})(\ell_{45}\ell_{23})  & (2 , 3 , 2 , 2) & q^2-q-2 & 15\\
14 & s_{24} \circ s_{35} & (\ell_3\ell_{5})(\ell_{45}\ell_{23})(\ell_{14}\ell_{12})(r_2r_4) & (2 , 3 , 1 , 1) & q^2-q-2 & 14\\
15 & s_{145}\circ s_{24} \circ s_{35}& (q\ell_{23}\ell_1\ell_{45})(\ell_{3}\ell_{14}\ell_{12}\ell_{5})(r_2r_4)   & (2 , 1 , 0 , 0) & q^2+q & 19 \\
\hline
16 & 2\A_1 ( 8 ) & Id & (6 , 10 , 12 , 0) & q^2-4q+3 & 12 \\
17 & s_{12} & (\ell_{1}\ell_{2})(\ell_{14}\ell_{24}) &   (4 , 6 , 6 , 0) &q^2-2q+1 & 13\\
18 & s_{345}\circ s_{12} & (\ell_{1}\ell_{2})(\ell_{14}\ell_{24})(\ell_{3}\ell_{45})(\ell_{34}\ell_{5}) &   (2 , 2 , 0 , 0) & q^2-1 & 15 \\
19 & s_{15} \circ s_{234} &  (\ell_{1}\ell_{5})(\ell_{2}\ell_{34})(\ell_{3}\ell_{24})(\ell_{45}\ell_{14})(r_{123}r_4) &   (2 ,0 , 0  , 0) & q^2+q+1 & 14\\
20 & s_{134}\circ s_{245}\circ s_{25} & (\ell_{1}\ell_{34})(\ell_{2}\ell_{24})(\ell_{3}\ell_{14})(\ell_{45}\ell_{5})(r_{123}r_4)  &   (0 ,0 , 0  , 2) & q^2-1 & 16\\
21  & s_{145}\circ s_{234}\circ s_{15}  \circ s_{23} & (\ell_{1}\ell_{14})(\ell_{2}\ell_{24})(\ell_{3}\ell_{34})(\ell_{45}\ell_{5})(r_{123}r_4) & (-2 ,0 , 0 , 4) & q^2-2q-3 & 17\\
22 & s_{12}  \circ s_{23} & (\ell_{1}\ell_{2}\ell_{3})(\ell_{14}\ell_{24}\ell_{34})  & (3 , 4 , 3 , 0) & q^2-q & 18\\
23 & s_{134}\circ s_{15}  \circ s_{25} &  (\ell_{1}\ell_{5}\ell_{2}\ell_{34})(\ell_{3}\ell_{14}\ell_{45}\ell_{24})(r_{123}r_4)  & (2 , 0 , 0 , 0) & q^2+2q+1 & 19\\
24 &  s_{145}\circ s_{12}  \circ s_{23} & (\ell_{1}\ell_{2}\ell_{3}\ell_{45})(\ell_{5}\ell_{14}\ell_{24}\ell_{34})   & (2 , 2 , 0  , 0) &q^2-1 & 20\\
25 &  s_{124}\circ s_{15}  \circ s_{35} \circ s_{23} &  (\ell_{1}\ell_{5}\ell_{3}\ell_{14}\ell_{45}\ell_{34})(\ell_{2}\ell_{24})(r_{123}r_4)  & (1 , 0 , 0 , 1) & q^2+q & 21\\
\hline
26 & \A_2 ( 8 ) & Id & (6 , 10 , 13 , 0) & q^2-4q+4 & 22\\
27 & s_{12} & (\ell_{1}\ell_{2})(\ell_{13}\ell_{23}) &   (4 , 6 , 5 , 0) & q^2-2q & 23 \\
28 & s_{345}\circ s_{12} & (\ell_{1}\ell_{2})(\ell_{13}\ell_{23})(\ell_{5}\ell_{34})(\ell_{12}q) &  (2 , 2 , 1 , 0) & q^2 & 24\\
29 & s_{235}\circ s_{124}\circ s_{14} &  (\ell_{1}\ell_{12})(\ell_{2}q)(\ell_{5}\ell_{23})(\ell_{13}\ell_{34})(r_3r_4) &  (0 , 0 , 0 , 3) & q^2-2 & 25\\
30 & s_{135}\circ s_{124}\circ s_{14} \circ s_{24} & (\ell_{1}\ell_{12}\ell_{2}q)(\ell_{5}\ell_{13}\ell_{34}\ell_{23})(r_3r_4) &   (0 , 0 , 0 , 1) & q^2 & 28\\
\hline
31 & 3\A_1 ( 6 ) & Id & (6 , 9 , 10 , 0) & q^2 -3q+2 & 31 \\
32 & s_{145} & (\ell_{1}\ell_{45})(\ell_{5}\ell_{14}) &   (4 , 5 , 4 , 0) & q^2- q & 32 \\
33 & s_{234}\circ s_{15} & (\ell_{5}\ell_{1})(\ell_{3}\ell_{24})(\ell_{45}\ell_{14}) (r_{4}r_{123}) &   (2 , 1 , 0 , 0) &q^2 +q & 33\\
34 & s_{234}\circ s_{145}\circ s_{15} & (\ell_{1}\ell_{14})(\ell_{5}\ell_{45})(\ell_{3}\ell_{24}) (r_{4}r_{123}) &   (0 , 1 , 0 , 2) & q^2 -q-2 & 34\\
\hline
35 & \A_1\A_2 ( 6) & Id & (6 , 9 , 10 , 0) & q^2 -3q+2 & 37\\
36 & s_{345} & (\ell_{5}\ell_{34})(q\ell_{12}) &   (4 , 5 , 4 , 0) & q^2 -q & 38 \\
\hline
37 & \A_3  (5) & Id & (6 , 8 , 8 , 0) & q^2 -2q+1 & 40\\
38 & s_{125}\circ s_{134} & (\ell_{5}\ell_{12})(q\ell_{1})(r_{2}r_{4}) &   (2 , 2 , 1 , 1) & q^2 -1 & 42\\
\hline
39 & \A_3  (4) & Id & (6 , 7 , 6 , 0) & q^2 -q & 40\\
40 & s_{345} & (\ell_{5}\ell_{34}) &   (4 , 5 , 4 , 0) & q^2 -q & 41\\
41 & s_{345}\circ s_{12} & (\ell_{1}\ell_{2})(\ell_{5}\ell_{34}) &   (2 , 3 , 2 , 0) & q^2 -q& 43\\
42 & s_{234}\circ s_{15} & (\ell_{1}\ell_{5})(\ell_{2}\ell_{34})(r_{123}r_{4}) &   (2 , 1 ,  0 , 0) & q^2 +q& 42 \\
43 & s_{134}\circ s_{15} \circ s_{25} & (\ell_{1}\ell_{5}\ell_{2}\ell_{34})(r_{123}r_{4}) &   (2 , 1 , 0 , 0) & q^2 +q & 44 \\
\hline
44 & 4\A_1 ( 4) & Id & (6 , 8 , 8 , 0) & q^2 -2q+1 & 45\\
45 & s_{134}\circ s_{25} & (\ell_{2}\ell_{5})(\ell_{3}\ell_{14})(r_{1}r_{345})(r_{4}r_{123}) &   (2 , 0 , 0 , 0) & q^2 +2q+1 & 46 \\
46 & s_{14} \circ s_{25} & (\ell_{2}\ell_{5})(r_{1}r_{4})(r_{123}r_{345}) &   (2  , 2 , 0 , 0) & q^2-1 & 46\\
47 & s_{124}\circ s_{35} & (\ell_{2}\ell_{14})(\ell_{3}\ell_{5})(r_{4}r_{123}) &   (2 , 2 , 0 , 0) & q^2-1 & 47\\
48 & s_{145}\circ s_{14} \circ s_{25} \circ s_{35} & (\ell_{2}\ell_{14}\ell_{5}\ell_{3})(r_{1}r_{4}r_{345}r_{123}) &   (0 , 0 , 0 , 0) & q^2-1 & 49 \\
\hline
49 & 2\A_1\A_2  (4) & Id & (6 , 8 , 8 , 0) & q^2 -2q+1 & 50\\
50 & s_{134}\circ s_{245} \circ s_{25} & (\ell_{3}\ell_{14})(\ell_{5}\ell_{45})(r_{1}r_{2})(r_{123}r_{4}) &   (0 , 0 , 0 , 2) & q^2 -1 & 51\\
\hline
51 & \A_1\A_3 ( 3) & Id & (6 , 7 , 6 , 0) & q^2 -q & 52\\
52 & s_{345} & (\ell_{34}\ell_5) &   (4 , 5 , 4 , 0) & q^2 -q & 53\\
\hline
53 & \A_4 ( 3) & Id & (6 , 7 , 6 , 0) & q^2 -q & 60\\
\hline
54 & \D_4 ( 2 )& Id & (6 , 6 , 5 , 0) &q^2 & 62 \\
55 & s_{234} \circ s_{15} & (\ell_{5}\ell_{2})(r_{4}r_{123}) &   (2 , 2 , 1 , 0) &q^2 & 63 \\
\hline
56 & 2\A_1\A_3 ( 2 ) & Id & (6 , 7 , 6 , 0) & q^2 -q & 67\\
57 & s_{134} \circ s_{25} & (\ell_{5}\ell_{2})(r_{1}r_{345})(r_{123}r_{4}) &   (2 , 1 , 0 , 0) & q^2 +q & 68 \\
\hline
58 & \D_5 ( 1 ) & Id & (6 , 6 , 5 , 0) &q^2 & 72\\
\hline
\end{array}$
\caption{Galois action on negative curves in degree $4$}
}
\end{table}
\end{center}
 
Note that this is sufficient to prove the existence of a degree three del Pezzo surface of the type given in the last column over the finite field $\F_q$ as long as the number given in the last but one column is positive for the given $q$.

Assume that there exists a weak del Pezzo surface of degree three and arithmetic type $1$, defined over $\F_q$; the Galois action on the Picard group is trivial, and all its negative curves are defined over $\F_q$. It contains an exceptional curve that does not meet the $(-2)$-curve. Contracting this curve gives a degree three del Pezzo surface of arithmetic type $1$, and the image of the curve is a point that does not lie on any negative curve. But such a point does not exist when $q=3$ from the first line of the above table.

One shows in the same way that a weak del Pezzo surface of degree three and arithmetic type $12$ does not exist over $\F_3$. The only difference is that contracting an exceptional curve that does not meet the $(-2)$-curve gives a point outside the negative curves on a weak del Pezzo surface of degree four and arithmetic type $11$ or $16$; such a point does not exist when $q=3$.

\subsection{Other constructions}

There remains $77-48=29$ types to be constructed over any finite field with odd characteristic.

For those we give an alternative contruction: we often present them as blow ups of the projective plane at well-chosen rational points, and sometimes as blow ups of a degree four del Pezzo surface at some point lying on one or two negative curve(s).

First remark that we get all arithmetic types for the geometric type $\A_1$ over $\F_q$ by blowing up the points of a degree $6$ zero-dimensional subscheme of a smooth conic, everything being defined over $\F_q$, when $q$ is large enough. Note that the only $(-2)$-curve is the strict transform of the conic. Playing on the fields of definition of the points in the subscheme, we get all partitions of the integer $6$, and all types (note the stabilizer here is $\S_6$). In this way we construct a surface for each of the remaining types. We get type $5$ by blowing up two points defined over $\F_{q^3}$, type $10$ by blowing up one point defined over $\F_{q}$ and one over $\F_{q^5}$, and type $5$ by blowing up one point defined over $\F_{q^6}$. Since a smooth conic defined over $\F_q$ has $q^k+1$ points over $\F_{q^k}$, such configurations exist over any finite field (also of even characteristic).

We get all arithmetic types for the geometric type $\A_2$ over $\F_q$ by blowing up the points of two degree $3$ zero-dimensional subschemes lying on two different lines, but not containing their intersection point, everything being defined over $\F_q$ for $q$ large enough. Note that the $(-2)$-curves are the strict transforms of the two lines. As above, the fields of definitions of the points, and the lines, give the arithmetic type. We get types
\begin{itemize}
	\item[$26$] when we choose the two lines to be defined over $\F_q$, with three points defined over $\F_q$ on one, and three conjugate points defined over $\F_{q^3}$ on the other line.
		\item[$27$] when we choose the two lines to be defined over $\F_q$, with three conjugate points defined over $\F_{q^3}$ on the both lines.
			\item[$29$] when we choose the two lines to be defined over $\F_q$, with one point defined over $\F_q$ and two defined over $\F_{q^2}$ on one, and three conjugate points defined over $\F_{q^3}$ on the other line.
				\item[$30$] when we choose the two lines to be conjugate, defined over $\F_{q^2}$, with one point defined over $\F_{q^6}$ on one, and all its conjugates points.
\end{itemize}

We get type $35$ by blowing up three conjugate points defined over $\F_{q^3}$ in the projective plane, then three infinitely near conjugate points. Note that the $(-2)$-curves are the strict tranforms of the exceptional divisors of the first blow up.

In order to construct a surface of type $39$, we start from a degree four del Pezzo surface of arithmetic type $22$. Such a surface contains two exceptional curves defined over $\F_q$, that meet together and each one meets exactly one $(-2)$-curve. We blow up a point defined over $\F_q$ on one of these exceptional curve, different from their intersection point and not lying on a $(-2)$-curve. This is possible for any $q$, and we get the desired surface.

In order to get a degree three surface of geometric type $4\A_1$, we start with four lines in general position in the projective plane, their union being defined over $\F_q$. Then we blow up their six intersection points; the $(-2)$-curves are the strict transforms of the lines. Then we get the arithmetic types by playing on the fields of definition of the lines. For instance we get type $48$ by choosing three conjugates lines defined over $\F_{q^3}$ and the last one defined over $\F_q$.

If we start from a degree four surface of geometric type $2\A_1$ defined over $\F_q$, and we blow up a rational point which lies at the intersection of two exceptional curve (ie corresponding to a ``middle'' edge of the graph of negative curves), we get a degree three surface of geometric type $2\A_2$ over $\F_q$. In this way we get  types $54, 55, 56, 57, 58$ 
$59$ in degree three respectively from types $16, 17, 20, 21, 22$ et $25$ in degree four. Note that in any case these degree four surfaces contain two intersecting exceptional curves which are either defined over $\F_q$, or defined over $\F_{q^2}$ and conjugate over $\F_q$: their intersection point is always a rational point.

We get type $61$ by blowing up $p_1\prec p_2\prec p_3\prec p_4$ four infinitely near points with the first three not collinear, and a point of degree $2$ on a line passing through $p_1$ but whose strict transform by the first blow up does not contain $p_2$.

We get a degree three surface of type $64$ by blowing up three collinear points $p_1,p_2,p_3$ defined over $\F_{q^3}$ and conjugate over $\F_q$ then three conjugate points $p_4,p_5,p_6$ with $p_{i}\prec p_{i+3}$.

We get a degree three del Pezzo surface of geometric type $\A_1 2\A_2$ when we blow up the point $\ell_1\cap\ell_{14}$ on a degree four del Pezzo surface of type $3\A_1$. This point is defined over $\F_q$ when the degree four surface has one of the types $31$ or $34$, and we get a surface of type $65$ or $66$ in this way.

Most of the remaining constructions are completely geometric and can be found in \cite[pp 493-494]{dolga}. Namely we get type 
\begin{itemize}
	\item[69] by blowing up $p_1\prec p_2\prec p_3\prec p_4\prec p_5$ with $p_1, p_2, p_3$ not collinear and $p_6$ on the (smooth) conic defined by the first five ones;
	\item[70] by blowing up $p_1\prec p_2\prec p_3\prec p_4\prec p_5\prec p_6$ with $p_1, p_2, p_3$ not collinear;
	\item[73] by blowing up $p_1\prec p_2\prec p_3\prec p_4\prec p_5 \prec p_6$ with $p_1, p_2, p_3$ not collinear and $p_6$ on the conic defined by the first five ones;
	\item[77] by blowing up $p_1\prec p_2\prec p_3\prec p_4\prec p_5\prec p_6$ with $p_1, p_2, p_3$ collinear.
\end{itemize}

Type $71$ in degree three is obtained by blowing up a rational point lying on the exceptional curve $\ell_2$ (but not on any $(-2)$-curve) in a degree four del Pezzo surface of arithmetic type $52$.

We end with the geometric type $3\A_2$; we blow up three points $p_1,p_2,p_3$ in general position in the projective plane. Then, on the resulting degree $6$ ordinary del Pezzo surface, we blow up the intersection points $\ell_1\cap \ell_{12}$, $\ell_2\cap \ell_{23}$ and $\ell_3\cap \ell_{13}$. The $(-2)$-curves on the corresponding degree three del Pezzo surface are the strict transforms of these six lines. Finally, we get a surface of type $74$ when the $p_i$ are rational points, of type $75$ when one is rational and the two other ones defined over $\F_{q^2}$ and conjugate over $\F_q$, and of type $76$ when the three points are defined over $\F_{q^3}$ and conjugate over $\F_q$.

\subsection{Type 36}

It remains to construct a del Pezzo surface of degree three and arithmetic type $36$. Such a surface has geometric type $3\A_1$, and the inverse of the weight two part of its zeta function is $\Phi_1^2\Phi_2\Phi_3^2$.

We start with a degree one del Pezzo surface $S$ obtained by blowing up two degree three points $p_1,p_2=p_1^\sigma,p_3=p_1^{\sigma^2}$ and $q_1,q_2=q_1^\sigma,q_3=q_1^{\sigma^2}$ in $\P^2(\F_{q^3})$ and a degree two point $r_1,r_2=r_1^\sigma$ in $\P^2(\F_{q^2})$, such that there exists a conic passing through $p_1,p_2,q_1,q_2,r_1,r_2$ and defined over $\F_{q^3}$, but there are no other $(-2)$-curve on the resulting degree one surface than the strict transforms of this conic and its conjugates.  We report the proof of the existence of such a configuration.

The surface $S$ has geometric type $3\A_1$ by our hypothesis on its $(-2)$-curves (note that the strict transforms are separated by the blowups), and the inverse of the weight two part of its zeta function is $(X-1)(X^2-1)(X^3-1)^2=\Phi_1^4\Phi_2\Phi_3^2$. We can contract the strict transform of the line $(r_1r_2)$, and of the conic $(q_1q_2q_3r_1r_2)$: these are disjoint exceptional curves that do not meet the $(-2)$-curves, both defined over $\F_q$. In this way we get a degree three surface with geometric type $3\A_1$ and the desired zeta function.

It remains to show the existence of such a configuration of points over any finite field of odd characteristic.

We start with two elements $\theta\in \F_{q^3}\setminus \F_q$, $\eta\in \F_{q^2}\setminus \F_q$, whose respective minimal polynomials over $\F_q$ are $\pi_\theta(x):=X^3+aX^2+bX+c$ and $\pi_\eta(X)=X^2+tX+n$.

We consider the points $p_1=(\theta:\theta^2:1)$, $q_1=(\theta:\theta^2:u)$ and $r_1=(\eta:1:0)$ for some $u\in \F_{q^3}\setminus\{0,1\}$.

From Pascal's theorem, the points $p_1,p_2,q_2,r_2,r_1,q_1$ are conconic if and only if the pairs of opposite sides of the hexagon meet in three collinear points $(p_1p_2)\cap(r_1r_2)$, $(p_2q_2)\cap(r_1q_1)$ and $(p_1q_1)\cap(q_2r_2)$; we get the points
\[
(\theta-\theta^q:\theta^2-\theta^{2q}:0),~(\theta^q:\theta^{2q}:\alpha),~(\theta:\theta^2:\beta)
\]
where we have set
\[
\alpha:=u\frac{\eta\theta^{2q}-\theta^q}{\eta\theta^{2}-\theta},~\beta:=u^q\frac{\eta^q\theta^{2}-\theta}{\eta^q\theta^{2q}-\theta^q}
\]
The points are collinear if and only if we have $(u\beta-u^q\alpha)\theta^3(\theta^{q-1}-\theta^{2(q-1)})=0$, if and only if $u\beta-u^q\alpha=0$ (note that $\theta^{q-1}\neq 1$). We easily check that the preceding equality gives
\[
u\in (\eta\theta^2-\theta)(\eta^q\theta^2-\theta)\F_q^\times=(n\theta^4+t\theta^3+\theta^2)\F_q^\times
\]
In other words, for our choice of the points, there are exactly $q-1$ values of $u$ in $\F_{q^3}$ such that the points $p_1,p_2,q_2,r_2,r_1,q_1$ are conconic; note that using the Frobenius action, we also get that the points $p_2,p_3,q_3,r_2,r_1,q_2$ are conconic, so as the points $p_1,p_3,q_3,r_2,r_1,q_1$.

We remark that the conic $C_1$ passing through $p_1,p_2,q_2,r_2,r_1,q_1$ must be irreducible. Else it would be the union of two lines $d$ and $d'$. If $d$ contains $r_1$, it cannot contain $r_2=r_1^\sigma$: it would be the line $Z=0$  and the other line would contain the four remaining points; this is impossible since $p_2$ does not lie on $(p_1q_1)$. Thus $d^{\sigma^3}$ contains $r_2$, $d$ is not defined over $\F_{q^3}$, and we get $d'=d^{\sigma^3}$. In this case both $d$ and $d^{\sigma^3}$ would contain $p_1,p_2,q_1,q_2$, this is impossible.

It remains to verify that one can choose $u$ such that there is no other $(-2)$-curve on the degree one surface obtained by blowing up the eight points. In other words, no three points can be collinear, no six can lie on a conic, and there is no cubic passing through the eight points which is singular at one of them.

We first check that no three of the eight points can be collinear. If such a subset contains $r_1$ and $r_2$, this is clear. If it contains $r_1$, then the line joining $r_1$ to any degree three point is defined over $\F_{q^6}$, and cannot contain any other point defined over $\F_{q^3}$. 

We are reduced to the subsets of three points among the $p_1,p_2,p_3,q_1,q_2,q_3$. Any subset of the form $\{p_i,p_j,q_k\}$ or $\{q_i,q_j,p_k\}$ with $k\in\{i,j\}$ cannot be collinear; else the corresponding line would cross one of the three irreducible conics conjugate to $C_1$ above at three points. We are reduced to consider the subsets $\{p_1,p_2,p_3\}$, $\{q_1,q_2,q_3\}$, and (up to Galois action) $\{p_1,p_2,q_3\}$, $\{q_1,q_2,p_3\}$.

Note that a point $(x:y:z)\in \P^2(\F_{q^3})$ and its conjugates over $\F_q$ are collinear if and only if the $\F_q$ subspace generated by $x,y,z$ is strictly contained in $\F_{q^3}$. As a consequence, it follows from the definition of $p_1$ that the points $p_1,p_2,p_3$ are not collinear. Now write 
\[
u=m(n\theta^4+t\theta^3+\theta^2)=m\left((1-nb-a(t-na))\theta^2-(nc+t-na)\theta-(t-na)c\right)
\]
for some $m\in \F_q^\times$. From the above criterion, we get that the points $q_1,q_2,q_3$ are collinear if and only if $t=na$ (of course we have $c\neq 0$).

The points $p_1,p_2,q_3$ are collinear if and only if
\[
\begin{vmatrix}
1 & \theta & \theta^2 \\
1 & \theta^q & \theta^{2q} \\
u^{q^2} & \theta^{q^2} & \theta^{2q^2} \\
\end{vmatrix}=0
\]
This equation is (semi)-linear in the variable $u$, and admits a unique solution.

In the same way, the points $q_1,q_2,p_3$ are collinear if and only if
\[
u(\theta^{2q^2+q}-\theta^{q^2+2q})-u^q(\theta^{2q^2+1}-\theta^{q^2+2})+\theta^{2q+1}-\theta^{q+2}=0
\]
Setting $U=u(\theta^{2q^2+q}-\theta^{q^2+2q})$, we can rewrite the equation $U+U^q+\theta^{2q+1}-\theta^{q+2}=0$.
Since $\gcd(T^q+T,T^{q^3}-T)=T$, the map $U\mapsto U^q+U$ is an $\F_q$-linear automorphism of $\F_{q^3}$, and this equation admits exactly one solution. 

Now choose six of the eight points, and assume they lie on a conic. If $r_i$ is one of these points, then among the five remaining points, four must lie on one of the conjugates of $C_1$; from Bezout theorem, the conic must be one of these, and the subset of six conconic points is among the three we already constructed. It remains to consider the six points $p_1,p_2,p_3,q_1,q_2,q_3$; any conic passing through these points should be defined over $\F_q$, et thus have an equation of the form
\[
eX^2+hXY+iY^2+jXZ+kYZ+lZ^2
\]
Plugging the coordinates of $p_1$ and $q_1$, substracting and factoring $1-u$, we get
\[
j\theta+k\theta^2+l(1+u)=0
\]
so that the family $(\theta,\theta^2,1+u)$ does not generate the $\F_q$-vector space $\F_{q^3}$. This happens if, and only if $1+mc(t-na)=0$ -- with $u=m(n\theta^4+t\theta^3+\theta^2)$.

It remains to show that there is no cubic passing through the eight points and singular at one of them. If such a cubic $C$ exists, it has exactly one singular point, which is not defined over $\F_q$. So it is distinct from $C^\sigma$. Now these two cubics have at least ten intersection points (counted with multiplicities), which is impossible.

Summing up, we have to choose $\eta$ and $\theta$ as above, such that the coefficients of their minimal polynomials satisfy $t-na\neq 0$; then we get $q-1$ possible values for $u$ such that the three prescribed subsets of six points are conconic. We remove one possible value to verify that the six points $p_i,q_i$ do not lie on a conic, then at most one for each of the assertions $p_1,p_2,q_3$ non collinear, and $q_1,q_2,p_3$ non collinear. We get a least $q-4$ possible values for $u$, and the problem is settled for $q\geq 5$.

For $q=3$, one can verify that if $\theta$ is a root of $T^3+T^2+2$, and $\eta$ a root of $T^2+T+2$, then for $m=-1$ the points defined above have the required position.

\appendix

\section{Arithmetic types for degrees three to six}
\label{gt}

We give below the lists of arithmetic types for degrees $3\leq d \leq 6$. The tables are organized as follows

\begin{itemize}
	\item in the first column, we fix a number for each type; there are also three types (one in degree four and two in degree three) for which we add an asterisk. This means that the invariant $H^1(\Gamma,\Pic(X\otimes \overline{\F}_q))$ is non trivial (it is isomorphic to the Klein four-group $\Z/2\Z\times \Z/2\Z$ in each case), and that a surface of this type is not birational to the projective plane \cite[Section 29]{manin}. 
	\item in the second one, we describe the geometric types for the given degrees, in other words the closed and symmetric parts of a root system of type $\E_{9-d}$ up to the action of the Weyl group; they come from \cite{ct}, except for degree $3$ where they are given in \cite[Table IV (iv)]{cox}. We give the Dynkin types of the singular points and (between parentheses) the number of exceptional divisors on the surface;
	\item in the third column, we give the stabilizer attached to the geometric type;
  \item in the fourth one, we give the characteristic polynomial of the action of any element $w$ in the conjugacy class of the stabilizer on the geometric Picard group $\Pic(\overline{X})$ of a weak del Pezzo surface of the type;
	\item finally, we give the characteristic polynomial of the action of any element $w$ in the conjugacy class of the stabilizer on the geometric Picard group $\Pic(\overline{X}_s)$ of a singular del Pezzo surface of a type.
\end{itemize}

Moreover, in the table for degree $4$, we add a column, in order to describe each conjugacy class in $W(\D_5)$ by its signed cycle type.

\begin{longtable}{|c|l|c|l|l|}
\caption{Arithmetic types in degree $6$}\\
\hline
\textrm{Type~} $\mathfrak{T}$ & \textrm{Geometric type} & \textrm{Stab} & $\chi_{w,\Pic(\overline{X})}$ & $\chi_{w,\Pic(\overline{X}_s)}$ \\
\hline
\hline
\endhead
\hline
\endfoot
1 & $\A_1$ ($4$) & $\Z/2\Z$ & $\Phi_1^4$ & $\Phi_1^3$ \\
2 & &  &   $\Phi_1^3\Phi_2$ & $\Phi_1^2\Phi_2$ \\
\hline
3 & $\A_1$ ($3$) & $\S_3$ & $\Phi_1^4$ & $\Phi_1^3$ \\
4 &   &  & $\Phi_1^3\Phi_2$ & $\Phi_1^2\Phi_2$ \\
5 &   &  & $\Phi_1^2\Phi_3$ & $\Phi_1\Phi_3$ \\
\hline
6 & $2\A_1$ ($2$) & $\{e\}$ & $\Phi_1^4$ & $\Phi_1^2$ \\
\hline
7 & $\A_2$ ($2$) & $\Z/2\Z$ & $\Phi_1^4$ & $\Phi_1^2$ \\
8 & &  &   $\Phi_1^3\Phi_2$ & $\Phi_1\Phi_2$ \\
\hline
9 & $\A_2\A_1$ ($1$) & $\{e\}$ & $\Phi_1^4$ & $\Phi_1$ \\
\hline
\end{longtable}

\begin{longtable}{|c|l|c|l|l|}
\caption{Arithmetic types in degree $5$}\\
\hline
\textrm{Type~} $\mathfrak{T}$ & \textrm{Geometric type} & \textrm{Stab} & $\chi_{w,\Pic(\overline{X})}$ & $\chi_{w,\Pic(\overline{X}_s)}$ \\
\hline
\hline
\endhead
\hline
\endfoot
1 & $\A_1$ ($7$) & $\S_3$ & $\Phi_1^5$ & $\Phi_1^4$ \\
2 & &    & $\Phi_1^4\Phi_2$ & $\Phi_1^3\Phi_2$ \\
3 & &    & $\Phi_1^3\Phi_3$ & $\Phi_1^2\Phi_3$ \\
\hline
4 & $2\A_1$ ($5$) & $\Z/2\Z$ & $\Phi_1^5$ & $\Phi_1^3$ \\
5 & &  &   $\Phi_1^3\Phi_2^2$ & $\Phi_1^2\Phi_2$ \\
\hline
6 & $\A_2$ ($4$) & $\Z/2\Z$ & $\Phi_1^5$ & $\Phi_1^3$ \\
7 & &  &   $\Phi_1^4\Phi_2$ & $\Phi_1^2\Phi_2$ \\
\hline
8 & $\A_2\A_1$ ($3$) & $\{e\}$ & $\Phi_1^5$ & $\Phi_1^2$ \\
\hline
9 & $\A_3$ ($2$)& $\{e\}$ & $\Phi_1^5$ & $\Phi_1^2$ \\
\hline
10 & $\A_4$ ($1$) & $\{e\}$ & $\Phi_1^5$ & $\Phi_1$ \\
\hline
\end{longtable}

\begin{longtable}{|c|l|c|l|l|l|}
\caption{Arithmetic types in degree $4$}\\
\hline
Type~ $\mathfrak{T}$ & Geometric type  & Stab & $W(\D_5)$ & $\chi_{w,\Pic(\overline{X})}$ & $\chi_{w,\Pic(\overline{X}_s)}$ \\
\hline
\hline
\endhead
\hline
\endfoot
1 & $\A_1$ ($12$) & $\S_4\times \Z/2\Z$ & $11111$ &  $\Phi_1^6$ & $\Phi_1^5$ \\
2 & &    & $2111$ & $\Phi_1^5\Phi_2$ &  $\Phi_1^4\Phi_2$ \\
3 & &    & $2111$ & $\Phi_1^5\Phi_2$ &  $\Phi_1^4\Phi_2$ \\
4 & &    & $111\overline{11}$ & $\Phi_1^4\Phi_2^2$ &  $\Phi_1^3\Phi_2^2$ \\
5 & &    & $221$ & $\Phi_1^4\Phi_2^2$ &  $\Phi_1^3\Phi_2^2$ \\
6 & &    & $21\overline{11}$ & $\Phi_1^3\Phi_2^3$ &  $\Phi_1^2\Phi_2^3$ \\
7 & &    & $311$ & $\Phi_1^4\Phi_3$ & $\Phi_1^3\Phi_3$ \\
8 & &    & $2\overline{21}$ & $\Phi_1^2\Phi_2^2\Phi_4$ & $\Phi_1\Phi_2^2\Phi_4$ \\
9 & &    & $11\overline{21}$ & $\Phi_1^3\Phi_2\Phi_4$ & $\Phi_1^2\Phi_2\Phi_4$ \\
10 &   &  & $32$ & $\Phi_1^3\Phi_2\Phi_3$ & $\Phi_1^2\Phi_2\Phi_3$ \\
\hline

11 & $2\A_1$ ($9$) & $D_8$ & $11111$ & $\Phi_1^6$ & $\Phi_1^4$ \\
12 & &    & $2111$ & $\Phi_1^5\Phi_2$ & $\Phi_1^3\Phi_2$ \\
13 & &    & $221$ & $\Phi_1^4\Phi_2^2$ & $\Phi_1^2\Phi_2^2$ \\
14 & &    & $221$ & $\Phi_1^4\Phi_2^2$ & $\Phi_1^3\Phi_2$ \\
15 & &    & $41$ &$\Phi_1^3\Phi_2\Phi_4$ & $\Phi_1^2\Phi_4$ \\
\hline
16 & $2\A_1$ ($8$) & $\S_4\times \Z/2\Z$ & $11111$ & $\Phi_1^6$ & $\Phi_1^4$ \\
17 & &  &   $2111$ &$\Phi_1^5\Phi_2$ & $\Phi_1^3\Phi_2$ \\
18 & &  &   $111\overline{11}$ &$\Phi_1^4\Phi_2^2$ & $\Phi_1^2\Phi_2^2$ \\
19 & &  &   $111\overline{11}$ &$\Phi_1^4\Phi_2^2$ & $\Phi_1^3\Phi_2$ \\
20 & &  &   $21\overline{11}$ &$\Phi_1^3\Phi_2^3$ & $\Phi_1^2\Phi_2^2$ \\
21$^\ast$  &  &  & $1\overline{1111}$ &$\Phi_1^2\Phi_2^4$ & $\Phi_1\Phi_2^3$ \\
22 & &    & $311$ &$\Phi_1^4\Phi_3$ & $\Phi_1^2\Phi_3$ \\
23 & &    & $\overline{2}11\overline{1}$ &$\Phi_1^3\Phi_2\Phi_4$ & $\Phi_1^2\Phi_4$ \\
24 & &    & $\overline{2}11\overline{1}$ &$\Phi_1^3\Phi_2\Phi_4$ & $\Phi_1\Phi_2\Phi_4$ \\
25 & &    & $\overline{3}1\overline{1}$ &$\Phi_1^2\Phi_2^2\Phi_6$ & $\Phi_1\Phi_2\Phi_6$ \\
\hline
26 & $\A_2$ ($8$) & $D_8$ & $11111$ & $\Phi_1^6$ & $\Phi_1^4$ \\
27 & &  &   $2111$ & $\Phi_1^5\Phi_2$ & $\Phi_1^3\Phi_2$ \\
28 & &  &   $111\overline{11}$ & $\Phi_1^4\Phi_2^2$ & $\Phi_1^2\Phi_2^2$ \\
29 & &  &   $21\overline{11}$ & $\Phi_1^3\Phi_2^3$ & $\Phi_1^2\Phi_2^2$ \\
30 & &  &   $2\overline{21}$ & $\Phi_1^2\Phi_2^2\Phi_4$ & $\Phi_1\Phi_2\Phi_4$ \\
\hline
31 & $3\A_1$ ($6$) & $(\Z/2\Z)^2$ & $11111$ & $\Phi_1^6$ & $\Phi_1^3$ \\
32 & &  &   $2111$ & $\Phi_1^5\Phi_2$ & $\Phi_1^2\Phi_2$ \\
33 & &  &   $111\overline{11}$ & $\Phi_1^4\Phi_2^2$ & $\Phi_1^2\Phi_2$ \\
34 & &  &  $ 21\overline{11}$ & $\Phi_1^3\Phi_2^3$ & $\Phi_1\Phi_2^2$ \\
\hline
35 & $\A_1\A_2$ ($6$) & $\Z/2\Z$ & $11111$ & $\Phi_1^6$ & $\Phi_1^3$ \\
36 & &  &   $2111$ & $\Phi_1^5\Phi_2$ & $\Phi_1^2\Phi_2$ \\
\hline
37 & $\A_3$  ($5$) & $\Z/2\Z$ & $11111$ & $\Phi_1^6$ & $\Phi_1^3$ \\
38 & &  &   $221$ & $\Phi_1^4\Phi_2^2$ & $\Phi_1^2\Phi_2$ \\
\hline
39 & $\A_3$  ($4$) & $D_8$ & $11111$ & $\Phi_1^6$ & $\Phi_1^3$ \\
40 & &  &   $2111$ & $\Phi_1^5\Phi_2$ & $\Phi_1^2\Phi_2$ \\
41 & &  &   $111\overline{11}$ & $\Phi_1^4\Phi_2^2$ & $\Phi_1\Phi_2^2$ \\
42 & &  &   $111\overline{11}$ & $\Phi_1^4\Phi_2^2$ & $\Phi_1^2\Phi_2$ \\
43 & &  &   $\overline{2}11\overline{1}$ & $\Phi_1^3\Phi_2\Phi_4$ & $\Phi_1\Phi_4$ \\
\hline
44 & $4\A_1$ ($4$) & $D_8$ & $11111$ & $\Phi_1^6$ & $\Phi_1^2$ \\
45 & &  &   $111\overline{11}$ & $\Phi_1^4\Phi_2^2$ & $\Phi_1^2$ \\
46 & &  &   $221$ & $\Phi_1^4\Phi_2^2$ & $\Phi_1^2$ \\
47 & &  &   $111\overline{11}$ & $\Phi_1^4\Phi_2^2$ & $\Phi_1\Phi_2$ \\
48 & &  &   $2\overline{21}$ & $\Phi_1^2\Phi_2^2\Phi_4$ & $\Phi_1\Phi_2$ \\
\hline
49 & $2\A_1\A_2$  ($4$) & $\Z/2\Z$ & $11111$ & $\Phi_1^6$ & $\Phi_1^2$ \\
50 & &  &   $21\overline{11}$ & $\Phi_1^3\Phi_2^3$ & $\Phi_1\Phi_2$ \\
\hline
51 & $\A_1\A_3$ ($3$) & $\Z/2\Z$ & $11111$ & $\Phi_1^6$ & $\Phi_1^2$ \\
52 & &  &   $2111$ & $\Phi_1^5\Phi_2$ & $\Phi_1\Phi_2$ \\
\hline
53 & $\A_4$ ($3$) & $\{e\}$ & $11111$ & $\Phi_1^6$ & $\Phi_1^2$ \\
\hline
54 & $\D_4$ ($2$)& $\Z/2\Z$ & $11111$ & $\Phi_1^6$ & $\Phi_1^2$ \\
55 & &  &   $111\overline{11}$ & $\Phi_1^4\Phi_2^2$ & $\Phi_1\Phi_2$ \\
\hline
56 & $2\A_1\A_3$ ($2$) & $\Z/2\Z$ & $11111$ & $\Phi_1^6$ & $\Phi_1$ \\
57 & &  &   $111\overline{11}$ & $\Phi_1^4\Phi_2^2$ & $\Phi_1$ \\
\hline
58 & $\D_5$ ($1$) & $\{e\}$ & $11111$ & $\Phi_1^6$ & $\Phi_1$ \\
\hline
\end{longtable}

\begin{longtable}{|c|l|c|l|l|}
\caption{Arithmetic types in degree $3$}\\
\hline
\textrm{Type~} $\mathfrak{T}$ & \textrm{Geometric type} & \textrm{Stab} & $\chi_{w,\Pic(\overline{X})}$ & $\chi_{w,\Pic(\overline{X}_s)}$ \\
\hline
\hline
\endhead
\hline
\endfoot
1 & $\A_1$  ($21$) & $\S_6$ & $\Phi_1^7$ & $\Phi_1^6$ \\
2 & &  &   $\Phi_1^6\Phi_2$ & $\Phi_1^5\Phi_2$ \\
3 & &  &   $\Phi_1^5\Phi_2^2$ & $\Phi_1^4\Phi_2^2$ \\
4 & &  &   $\Phi_1^4\Phi_2^3$ & $\Phi_1^3\Phi_2^3$ \\
5 & &  &   $\Phi_1^3\Phi_3^2$ & $\Phi_1^2\Phi_3^2$ \\
6 & &  &   $\Phi_1^5\Phi_3$ & $\Phi_1^4\Phi_3$ \\
7 & &  &   $\Phi_1^4\Phi_2\Phi_3$ & $\Phi_1^3\Phi_2\Phi_3$ \\
8 & &  &   $\Phi_1^3\Phi_2^2\Phi_4$ & $\Phi_1^2\Phi_2^2\Phi_4$ \\
9 & &  &   $\Phi_1^4\Phi_2\Phi_4$ & $\Phi_1^3\Phi_2\Phi_4$ \\
10 & &  &   $\Phi_1^3\Phi_5$ & $\Phi_1^2\Phi_5$ \\
11 & &  &   $\Phi_1^2\Phi_2\Phi_3\Phi_6$ & $\Phi_1\Phi_2\Phi_3\Phi_6$ \\
\hline
12 & $2\A_1$ ($16$) & $\S_4\times \Z/2\Z$ & $\Phi_1^7$ & $\Phi_1^5$ \\
13 & &  &   $\Phi_1^6\Phi_2$ & $\Phi_1^4\Phi_2$ \\
14 & &  &   $\Phi_1^5\Phi_2^2$ & $\Phi_1^4\Phi_2$ \\
15 & &  &   $\Phi_1^5\Phi_2^2$ & $\Phi_1^3\Phi_2^2$ \\
16 & &  &   $\Phi_1^4\Phi_2^3$ & $\Phi_1^3\Phi_2^2$ \\
17$^\ast$  &  &  & $\Phi_1^3\Phi_2^4$ & $\Phi_1^2\Phi_2^3$ \\
18 & &  &   $\Phi_1^5\Phi_3$ & $\Phi_1^3\Phi_3$ \\
19 & &  &   $\Phi_1^4\Phi_2\Phi_4$ & $\Phi_1^3\Phi_4$ \\
20 & &  &   $\Phi_1^4\Phi_2\Phi_4$ & $\Phi_1^2\Phi_2\Phi_4$ \\
21 & &  &   $\Phi_1^3\Phi_2^2\Phi_6$ & $\Phi_1^2\Phi_2\Phi_6$ \\
\hline
22 & $\A_2$ ($15$) & $\S_3\wr \Z/2\Z$ & $\Phi_1^7$ & $\Phi_1^5$ \\
23 & &  &   $\Phi_1^6\Phi_2$ & $\Phi_1^4\Phi_2$ \\
24 & &  &   $\Phi_1^5\Phi_2^2$ & $\Phi_1^3\Phi_2^2$ \\
25 & &  &   $\Phi_1^4\Phi_2^3$ & $\Phi_1^3\Phi_2^2$ \\
26 & &  &   $\Phi_1^5\Phi_3$ & $\Phi_1^3\Phi_3$ \\
27 & &  &   $\Phi_1^3\Phi_3^2$ & $\Phi_1\Phi_3^2$ \\
28 & &  &   $\Phi_1^3\Phi_2^2\Phi_4$ & $\Phi_1^2\Phi_2\Phi_4$ \\
29 & &  &   $\Phi_1^4\Phi_2\Phi_3$ & $\Phi_1^2\Phi_2\Phi_3$ \\
30 & &  &   $\Phi_1^2\Phi_2\Phi_3\Phi_6$ & $\Phi_1\Phi_3\Phi_6$ \\
\hline
31 & $3\A_1$ ($12$) & $D_{12}$ & $\Phi_1^7$ & $\Phi_1^4$ \\
32 & &  &   $\Phi_1^6\Phi_2$ & $\Phi_1^3\Phi_2$ \\
33 & &  &   $\Phi_1^5\Phi_2^2$ & $\Phi_1^5\Phi_2$ \\
34 & &  &   $\Phi_1^4\Phi_2^3$ & $\Phi_1^2\Phi_2^2$ \\
35 & &  &   $\Phi_1^3\Phi_3^2$ & $\Phi_1^2\Phi_3$ \\
36 & &  &   $\Phi_1^2\Phi_2\Phi_3^2$ & $\Phi_1\Phi_2\Phi_3$ \\
\hline
37 & $\A_1\A_2$ ($11$) & $\S_3$ & $\Phi_1^7$ & $\Phi_1^4$ \\
38 & &  &   $\Phi_1^6\Phi_2$ & $\Phi_1^3\Phi_2$ \\
39 & &  &   $\Phi_1^5\Phi_3$ & $\Phi_1^2\Phi_3$ \\
\hline
40 & $\A_3$ ($10$) & $D_8$ & $\Phi_1^7$ & $\Phi_1^4$ \\
41 & &  &   $\Phi_1^6\Phi_2$ & $\Phi_1^3\Phi_2$ \\
42 & &  &   $\Phi_1^5\Phi_2^2$ & $\Phi_1^3\Phi_2$ \\
43 & &  &   $\Phi_1^5\Phi_2^2$ & $\Phi_1^2\Phi_2^2$ \\
44 & &  &   $\Phi_1^4\Phi_2\Phi_4$ & $\Phi_1^2\Phi_4$ \\
\hline
45 & $4\A_1$ ($9$)& $\S_4$ & $\Phi_1^7$ & $\Phi_1^3$ \\
46 & &  &   $\Phi_1^5\Phi_2^2$ & $\Phi_1^3$ \\
47 & &  &   $\Phi_1^5\Phi_2^2$ & $\Phi_1^2\Phi_2$ \\
48 & &  &   $\Phi_1^3\Phi_3^2$ & $\Phi_1\Phi_3$ \\
49 & &  &   $\Phi_1^3\Phi_2^2\Phi_4$ & $\Phi_1^2\Phi_2$ \\
\hline
50 & $2\A_1\A_2$ ($8$)& $\Z/2\Z$ & $\Phi_1^7$ & $\Phi_1^3$ \\
51 & &  &   $\Phi_1^4\Phi_2^3$ & $\Phi_1^2\Phi_2$ \\
\hline
52 & $\A_1\A_3$ ($7$)& $\Z/2\Z$ & $\Phi_1^7$ & $\Phi_1^3$ \\
53 &   &  & $\Phi_1^6\Phi_2$ & $\Phi_1^2\Phi_2$ \\
\hline
54 & $2\A_2$ ($7$)& $D_{12}$ & $\Phi_1^7$ & $\Phi_1^3$ \\
55 & &  &   $\Phi_1^6\Phi_2$ & $\Phi_1^2\Phi_2$ \\
56 & &  &   $\Phi_1^4\Phi_2^3$ & $\Phi_1^2\Phi_2$ \\
57$^\ast$  &  &  & $\Phi_1^3\Phi_2^4$ & $\Phi_1\Phi_2^2$ \\
58 & &  &   $\Phi_1^5\Phi_3$ & $\Phi_1\Phi_3$ \\
59 & &  &   $\Phi_1^3\Phi_2^2\Phi_6$ & $\Phi_1\Phi_6$ \\
\hline
60 & $\A_4$ ($6$) & $\Z/2\Z$ & $\Phi_1^7$ & $\Phi_1^3$ \\
61 & &  &   $\Phi_1^6\Phi_2$ & $\Phi_1^2\Phi_2$ \\
\hline
62 & $\D_4$ ($6$) & $\S_3$ & $\Phi_1^7$ & $\Phi_1^3$ \\
63 & &  &   $\Phi_1^5\Phi_2^2$ & $\Phi_1^2\Phi_2$ \\
64 & &  &   $\Phi_1^3\Phi_3^2$ & $\Phi_1\Phi_3$ \\
\hline
65 & $\A_1 2\A_2$ ($5$) & $\Z/2\Z$ & $\Phi_1^7$ & $\Phi_1^2$ \\
66 & &  &   $\Phi_1^4\Phi_2^3$ & $\Phi_1\Phi_2$ \\
\hline
67 & $2\A_1 \A_3$ ($5$) & $\Z/2\Z$ & $\Phi_1^7$ & $\Phi_1^2$ \\
68 & &  &   $\Phi_1^5\Phi_2^2$ & $\Phi_1^2$ \\
\hline
69 & $\A_1 \A_4$ ($4$) & $\{e\}$ & $\Phi_1^7$ & $\$Phi_1^3$ \\
\hline
70 & $\A_5$ ($3$) & $\Z/2\Z$ & $\Phi_1^7$ & $\Phi_1^2$ \\
71 & &    & $\Phi_1^6\Phi_2$ & $\Phi_1\Phi_2$ \\
\hline
72 & $\D_5$ ($3$) & $\{e\}$ & $\Phi_1^7$ & $\Phi_1^2$ \\
\hline
73 & $\A_1 \A_5$  ($2$) & $\{e\}$ & $\Phi_1^7$ & $\Phi_1$ \\
\hline
74 & $3\A_2$  ($3$) & $\S_3$ & $\Phi_1^7$ & $\Phi_1$ \\
75 & &    & $\Phi_1^4\Phi_2^3$ & $\Phi_1$ \\
76 & &    & $\Phi_1^3\Phi_3^2$ & $\Phi_1$ \\
\hline
77 & $\E_6$  ($1$) & $\{e\}$ & $\Phi_1^7$ & $\Phi_1$ \\
\hline
\end{longtable}

\bibliographystyle{amsplain}

\bibliography{SingularDelPezzo}

\end{document}